\documentclass[11pt,a4paper]{amsart}
\usepackage{mathrsfs}
\usepackage[english]{babel}
\usepackage{graphicx}
\usepackage{epstopdf}
\usepackage{epsfig,psfrag}
\usepackage{amsmath, amsfonts, amsthm,amssymb}
\usepackage{hyperref}
\usepackage{aurical}
\usepackage[T1]{fontenc}
\usepackage{yfonts}

\usepackage[usenames, dvipsnames]{color}
\usepackage{pifont}
\DeclareFontFamily{U}{astrosym}{}
\DeclareFontShape{U}{astrosym}{m}{n}{<-> astrosym}{}
\DeclareFontFamily{U}{dancers}{}
\DeclareFontShape{U}{dancers}{m}{n}{<-> dancers}{}
\newtheorem{theorem}{Theorem}[section]
\newtheorem{definition}[theorem]{Definition}
\newtheorem{lemma}[theorem]{Lemma}

\newtheorem{corollary}[theorem]{Corollary}
\newtheorem{remark}[theorem]{Remark}

\newtheorem{proposition}[theorem]{Proposition}

\numberwithin{equation}{section}

\newcommand{\yy}{{\mathcal Y}}
\newcommand{\s}{{\sigma}}
\newcommand{\id}{{\rm Id}}
\newcommand{\jj}{{\vec \jmath}}
\newcommand{\ii}{{\vec \imath}}
\newcommand{\li}{{\mathfrak h}}

\newcommand{\ol}{\overline}

\newcommand{\ad}{{\rm ad}}
\newcommand{\tT}{{\mathtt{T}}}
\newcommand{\rot}{{\mathrm{rot}}}

\newcommand{\bs}{{\boldsymbol\sigma}}

\newcommand{\wt}[1]{{\widetilde{#1}}}

\newcommand{\wtOmega}{{\widetilde{\Omega}}}

\newcommand{\wtcH}{{\widetilde{\mathcal H}}}
\newcommand{\wtcN}{{\widetilde{\mathcal N}}}
\newcommand{\wtcP}{{\widetilde{\mathcal P}}}

\newcommand{\wtcM}{\widetilde{\mathcal M}}

\newcommand{\whcH}{{\widehat{\mathcal H}}}

\newcommand{\whOmega}{{\widehat \Omega}}
\newcommand{\bd}{{\vec{\bf 2}}}

\newcommand{\C}{{\mathbb C}}

\newcommand{\N}{{\mathbb N}}
\newcommand{\Q}{{\mathbb Q}}
\newcommand{\R}{{\mathbb R}}

\newcommand{\T}{{\mathbb T}}
\newcommand{\Z}{{\mathbb Z}}

\newcommand{\cA}{{\mathcal A}}
\newcommand{\cB}{{\mathcal B}}
\newcommand{\cC}{{\mathcal C}}
\newcommand{\cD}{{\mathcal D}}

\newcommand{\cF}{{\mathcal F}}
\newcommand{\cG}{{\mathcal G}}
\newcommand{\cH}{{\mathcal H}}
\newcommand{\cI}{{\mathcal I}}
\newcommand{\cJ}{{\mathcal J}}
\newcommand{\cK}{{\mathcal K}}
\newcommand{\cL}{{\mathcal L}}
\newcommand{\cM}{{\mathcal M}}
\newcommand{\cN}{{\mathcal N}}
\newcommand{\cO}{{\mathcal O}}
\newcommand{\cP}{{\mathcal P}}
\newcommand{\cQ}{{\mathcal Q}}
\newcommand{\cR}{{\mathcal R}}
\newcommand{\cS}{{\mathcal S}}
\newcommand{\cT}{{\mathcal T}}
\newcommand{\cU}{{\mathcal U}}

\newcommand{\cX}{{\mathcal X}}
\newcommand{\cY}{{\mathcal Y}}
\newcommand{\cZ}{{\mathcal Z}}

\newcommand{\fA}{{\mathfrak{A}}}

\newcommand{\fg}{\mathfrak{g}}

\newcommand{\fh}{{\mathfrak{h}}}

\newcommand{\fm}{{\mathfrak{m}}}

\newcommand{\fR}{{\text{\gothfamily R}}}

\newcommand{\fZ}{{\text{\gothfamily Z}}}
\newcommand{\ccC}{{\mathscr{C}}}

\newcommand{\td}{{\mathtt{d}}}

\newcommand{\ti}{{\mathtt{i}}}

\newcommand{\tk}{{\mathtt{d}}}

\newcommand{\tm}{{\mathtt{m}}}

\newcommand{\tC}{{\mathtt{C}}}

\newcommand{\tF}{{\mathtt{F}}}
\newcommand{\tG}{{\mathtt{G}}}

\newcommand{\tK}{{\mathtt{K}}}

\newcommand{\tL}{{\mathtt{L}}}
\newcommand{\tM}{{\mathtt{M}}}

\newcommand{\ba}{{\bf a}}

\newcommand{\be}{{\bf e}}
\newcommand{\bi}{{\bf i}}
\newcommand{\bj}{{\bf j}}
\newcommand{\bk}{{\bf k}}
\newcommand{\bm}{{\bf m}}
\newcommand{\bn}{{\bf n}}

\newcommand{\al}{{\alpha}}
\newcommand{\bt}{{\beta}}
\newcommand{\La}{{\Lambda}}
\newcommand{\kk}{{\kappa}}
\newcommand{\g}{\gamma}

\newcommand{\derx}{\partial_x}

\newcommand{\norm}[1]{\| #1 \|}
\newcommand{\abs}[1]{\left| #1 \right|}
\newcommand{\0}{{(0)}}
\newcommand{\2}{{(2)}}
\newcommand{\1}{{(1)}}

\newcommand{\la}{\left\langle}
\newcommand{\ra}{\right\rangle}

\newcommand{\im}{{\rm i}}

\newcommand{\re}[1] {\mbox{\Fontlukas T}\!_{#1}}  
\newcommand{\und}[1]{\underline{#1}}
\newcommand{\di}{{\rm d}}
\newcommand{\e}{{\varepsilon}}
\newcommand{\eps}{{\varepsilon}}
\newcommand{\de}{\delta}
\newcommand{\hh}{ h }

\newcommand{\pix}{{\pi_x}}
\newcommand{\piy}{{\pi_y}}
\newcommand{\gen}{\fg}

\newcommand{\uno}{{\mathbb I}}

\newcommand{\meas}{{\rm meas}}

\newcommand{\red}[1]{{\textcolor{RubineRed}{#1}}}

\newcommand{\bnorm}[1]{{|\mkern-6mu |\mkern-6mu | \,  #1 \,  |\mkern-6mu 
		|\mkern-6mu |}  }
\newcommand{\Tan}{{\cS_0}}
\newcommand{\sS}{{\mathscr{S}}}
\newcommand{\sC}{{\mathscr{C}}}

\title[Nonlinear instability and growth of Sobolev norms near finite-gap tori]{Strong nonlinear instability and growth of Sobolev norms near quasiperiodic finite-gap tori for the 2D cubic NLS equation}

\author{
	M. Guardia}
\thanks{ Departament de Matem\`atica Aplicada I, 
	Universitat Polit\`ecnica de Catalunya, Diagonal 647, 08028
	Barcelona, Spain,  {\tt marcel.guardia@upc.edu}}
\author{Z. Hani}
\thanks{Department of Mathematics, University of Michigan, 530 Church Street,
	Ann Arbor, MI 48109-1043, USA, {\tt zhani@umich.edu}}  
\author{E. Haus}
\thanks{ Dipartimento di
	Matematica e Applicazioni ``R. Caccioppoli'', Universit\`a degli Studi di Napoli
	``Federico II'', Via Cintia, Monte S. Angelo, 80126 Napoli, Italy, {\tt emanuele.haus@unina.it}} 
\author{A. Maspero}
\thanks{ International School for Advanced Studies (SISSA), Via Bonomea 265, 34136, Trieste, Italy
	{\tt alberto.maspero@sissa.it}}
\author{ M. Procesi}
\thanks{Dipartimento di Matematica e Fisica,
	Universit\`a di Roma Tre, Largo S.L. Murialdo, 00100 Roma, Italy,
	{\tt mprocesi@mat.uniroma3.it}}

\begin{document}
\maketitle

\begin{abstract}
We consider the defocusing cubic nonlinear Schr\"odinger equation (NLS) on the 
two-dimensional torus.  The equation admits a special family of elliptic invariant 
quasiperiodic tori called finite-gap solutions. These are inherited from the 
integrable 1D model (cubic NLS on the circle) by considering solutions that 
depend only on one variable. We study the long-time stability of such invariant 
tori for the 2D NLS model and show that, under certain assumptions and over 
sufficiently long timescales, they exhibit a strong form of \emph{transverse 
instability} in Sobolev spaces $H^s(\T^2)$ ($0<s<1$). More precisely, we construct 
solutions of the 2D cubic NLS that start arbitrarily close to such invariant 
tori in the $H^s$ topology and whose $H^s$ norm can grow by any given factor. 
This work is partly motivated by the problem of infinite energy cascade for 2D NLS, 
and seems to be the first instance where (unstable) long-time nonlinear dynamics 
near (linearly stable) quasiperiodic tori is studied and constructed. 

%
\end{abstract}

\section{Introduction}

A widely held principle in dynamical systems theory is that invariant quasiperiodic tori play an important role in understanding the complicated long-time behavior of Hamiltonian ODE and PDE. In addition to being important in their own right, the hope is that such quasiperiodic tori can play an important role in understanding other, possibly more generic, dynamics of the system by acting as \emph{islands} in whose vicinity orbits might spend long periods of time before moving to other such islands. The construction of such invariant sets for Hamiltonian PDE has witnessed an explosion of activity over the past thirty years after the success of extending KAM techniques to infinite dimensions.  However, the dynamics near such tori is still poorly understood, and often restricted to the linear theory. The purpose of this work is to take a step in the direction of understanding and constructing non-trivial \emph{nonlinear} dynamics in the vicinity of certain quasiperiodic solutions for the cubic defocusing NLS equation. In line with the above philosophy emphasizing the role of invariant quasiperiodic tori for other types of behavior, another aim is to push forward a program aimed at proving infinite Sobolev norm growth for the 2D cubic NLS equation, an outstanding open problem.

\medskip

\noindent{\bf 1.1. The dynamical system and its quasiperiodic objects.} We start by describing the dynamical system and its quasiperiodic invariant objects at the center of our analysis. Consider the periodic cubic defocusing nonlinear Schr\"odinger
equation (NLS),
\begin{equation}\label{NLS}\tag{2D-NLS}
\im \partial_t u+\Delta u=|u|^2 u
\end{equation}
where $(x,y)\in\T^2=\R^2/(2\pi\Z)^2$, $t\in\R$ and
$u:\R\times\T^2\rightarrow\C$. All the results in this paper extend trivially to higher dimensions $d\geq 3$ by considering solutions that only depend on two variables\footnote{We expect that the results also extend to the focusing sign of the nonlinearity ($-|u|^2u$ on the R.~H.~S.~ of \eqref{NLS}). The reason why we restrict to the defocusing sign comes from the fact that the linear analysis around our quasiperiodic tori has only been established in full detail in \cite{Maspero-Procesi} in this case.}. This is a Hamiltonian PDE with conserved quantities: i) the Hamiltonian
\begin{equation}\label{def:Ham:Original}
H_0(u)=\int_{\T^2}\left(\left|\nabla
u(x,y)\right|^2+\frac{1}{2}|u(x,y)|^4\right)
 \di x \, \di y ,
\end{equation}
ii) the mass
\begin{equation}\label{def:NLS:mass}
 M(u)=\int_{\T^2}|u(x,y)|^2 \di x\, \di y,
\end{equation}
which is just the square of the $L^2$ norm of the solution, and iii) the momentum 
\begin{equation}\label{def:NLS:momentum}
P(u)=\im\int_{\T^2} \overline{ u(x,y)} \nabla u(x,y) \, \di x \, \di y.
\end{equation}

%

Now, we describe the invariant objects around which we will study and construct our long-time nonlinear dynamics. Of course, such a task requires a very precise understanding of the linearized dynamics around such objects. For this reason, we take the simplest non-trivial family of invariant quasiperiodic tori admitted by \eqref{NLS}, namely those inherited from its completely integrable 1D counterpart
\begin{equation}\label{def:NLS1D}\tag{1D-NLS}
 \im \partial_t q=-\partial_{xx} q+|q|^2q,\quad x\in \T.
\end{equation}
This is a subsystem of \eqref{NLS} if we consider solutions that depend only on the first spatial 
variable. It is well known that equation \eqref{def:NLS1D} is integrable and its 
phase space is foliated by tori of finite or infinite dimension with periodic, quasiperiodic, or almost periodic dynamics. The quasiperiodic orbits are 
usually called \emph{finite-gap solutions}. 

 Such tori are Lyapunov stable (for all time!) as solutions of \eqref{def:NLS1D} 
(as will be clear once we exhibit its integrable structure) and some of them are 
linearly stable as solutions of \eqref{NLS}, but we will be interested in their 
\emph{long-time nonlinear stability} (or lack of it) as invariant objects for 
the 2D equation \eqref{NLS}. In fact, we shall show that they are 
\emph{nonlinearly unstable} as solutions of \eqref{NLS}, and in a strong sense, 
in certain topologies and after very long times. Such instability is transversal 
in the sense that one drifts along the purely 2-dimensional directions: 
solutions which are initially very close to 1-dimensional become strongly 
2-dimensional after some long time scales\footnote{ 
The tranversal instability phenomenon was already studied for  solitary waves of the water waves equation \cite{RoussetT11} and the KP-I equation \cite{RoussetT12} by Rousset and Tzvetkov. However, their instability is a \emph{linear effect}, in the sense that the linearized dynamics is unstable. In contrast, our result is a fundamentally nonlinear effect, as the linearized dynamics around some of the finite gap tori is stable.}.

\medskip

\noindent{\bf 1.2. Energy Cascade, Sobolev norm growth, and Lyapunov instability.} In addition to studying long-time dynamics close to invariant objects for NLS, another purpose of this work is to make progress on a fundamental problem in nonlinear wave theory, which is the transfer of energy between characteristically different scales for a nonlinear dispersive PDE. This is called the \emph{energy cascade} phenomenon. It is a purely nonlinear phenomenon (energy is static in frequency space for the linear system), and will be the underlying mechanism behind the long-time instability of the finite gap tori mentioned above.

We shall exhibit solutions whose energy moves from very high frequencies towards low frequencies (\emph{backward or inverse cascade}), as well as ones that exhibit cascade in the opposite direction (\emph{forward or direct cascade}). Such cascade phenomena have attracted a lot of attention in the past few years as they are central aspects of various theories of turbulence for nonlinear systems. For dispersive PDE, this goes by the name of \emph{wave turbulence theory} which predicts the existence of solutions (and statistical states) of \eqref{NLS} that exhibit a cascade of energy between very different length-scales. In the mathematical community, Bourgain drew attention to such questions of energy cascade by first noting that it can be captured in a quantitative way by studying the behavior of the Sobolev norms of the solution
$$
\|u\|_{H^s}=\left(\sum_{n \in \Z^2} (1+|n|)^{2s}|\widehat u_n|^2\right)^{\frac12}.
$$
In his list of Problems on Hamiltonian PDE \cite{Bourgain00b}, Bourgain asked whether there exist solutions that exhibit a quantitative version of the forward energy cascade, namely solutions whose Sobolev norms $H^s$, with $s>1$, are unbounded in time
\begin{equation}\label{infinite growth}
 \sup_{t\geq 0} \|u(t)\|_{H^s}= +\infty, \qquad s>1.
\end{equation}

We should point out here that such growth cannot happen for $s=0$ or $s=1$ due to the conservation laws of the equations. For other Sobolev indices, there exists polynomial upper bounds for the growth of Sobolev norms (cf. \cite{Bourgain96,Staffilani97, 
CollianderDKS01,Bourgain04,Zhong08,
CatoireW10,Sohinger10a,Sohinger10b, Sohinger11,CollianderKO12,PTV17}). Nevertheless, 
results proving actual growth of Sobolev norms are much more scarce. After 
seminal works by Bourgain himself \cite{Bourgain96} and  Kuksin 
\cite{Kuksin96, Kuksin97, Kuksin97b}, the landmark result in \cite{CKSTT} played a fundamental importance in the recent progress, including this 
work: It showed that for any $s>1$, $\de\ll1$, $K\gg 1$,  
there 
exist solutions $u$ of \eqref{NLS} such that
\begin{equation}\label{eq:IteamGrowth}
 \|u(0)\|_{H^s}\leq \de \quad\text{ and }\quad  \|u(T)\|_{H^s}\geq K
\end{equation}
for some $T>0$.  Even if not mentioned in that paper, the same techniques also 
lead to the same result  for $s\in (0,1)$. This paper induced a lot of 
activity in the area \cite{GuardiaK12,Hani12,Guardia14,HaniPTV15, 
HausProcesi, GuardiaHP16} (see also 
\cite{GerardG10, Delort10,Pocovnicu11,GerardG11,Pocovnicu12,GerardG15,Maspero18g} on results about 
growth of Sobolev norms with different techniques). Despite all that, Bourgain's question about solutions exhibiting \eqref{infinite growth} remains open on $\T^d$ (however a positive answer holds for the cylindrical domain $\R\times \T^d$, \cite{HaniPTV15}).

\medskip

The above-cited works revealed an intimate connection between Lypunov instability and Sobolev norm growth. Indeed, 
the solution $u=0$ of \eqref{NLS} is an elliptic critical point and is 
linearly stable in all $H^s$. 
From this point of view, the result in \cite{CKSTT} given in 
\eqref{eq:IteamGrowth} can be 
interpreted as the Lyapunov instability in $H^s$, $s\neq 1$, of the elliptic 
critical 
point $u=0$ (the first integrals \eqref{def:Ham:Original} and  
\eqref{def:NLS:mass} imply Lyapunov stability in the $H^1$ and $L^2$ topology). It turns out that this connection runs further, particularly in relation to the question of finding solutions exhibiting \eqref{infinite growth}. As was observed in \cite{Hani12}, one way to prove the existence of such solutions is to prove that, for sufficiently many $\phi\in H^s$, an instability similar to that in \eqref{eq:IteamGrowth} holds, but with $\|u(0)-\phi\|_{H^s} \leq \delta$. In other words, proving long-time instability as in \eqref{eq:IteamGrowth} but with solutions starting $\delta-$close to $\phi$, and for sufficiently many $\phi\in H^s$ implies the existence (and possible genericness) of unbounded orbits satisfying \eqref{infinite growth}. Such a program (based on a Baire-Category argument) was applied successfully for the Szeg\"o equation on $\T$ in \cite{GerardG15}.

Motivated by this, one is naturally led to studying the Lyapunov instability of 
more general invariant objects  of \eqref{NLS} (or other Hamiltonian PDEs), or
equivalently to investigate whether one can achieve Sobolev norm explosion starting 
arbitrarily close to a given invariant object. The first work in this direction 
is by one of the authors \cite{Hani12}. He considers the plane waves $u(t,x)=A 
e^{\im(mx-\omega t)}$ with 
$\omega=m^2+A^2$, periodic orbits of \eqref{NLS}, and proves that there are 
orbits which start $\de$-close to them and undergo $H^s$ Sobolev norm
explosion, $0<s<1$. This implies that the plane 
waves are Lyapunov unstable in these topologies. Stability results for plane 
waves in $H^s$, $s>1$, on shorter time scales are provided in \cite{FaouGL14}. 

The next step in this program would be to study such instability phenomena near higher dimensional invariant objects, namely quasiperiodic orbits. This is the purpose of this work, in which we will address this question for the family of finite-gap tori of \eqref{def:NLS1D} as solutions to the \eqref{NLS}. To control the linearized dynamics around such tori, we will impose some Diophantine (strongly non-resonant) conditions on the quasiperiodic frequency parameters. This allows us to obtain a stable linearized operator (at least with respect to the perturbations that we consider), which is crucial to control the delicate construction of the unstable nonlinear dynamics.

\medskip

{\bf 1.3. Statement of results.} Roughly speaking, we will construct solutions to \eqref{NLS} that start very close to the finite-gap tori in appropriate topologies, and exhibit either backward cascade of energy from high to low frequencies, or forward cascade of energy from low to high frequencies. In the former case, the solutions that exhibit backward cascade start in an arbitrarily small vicinity of a finite-gap torus in Sobolev spaces $H^s(\T^2)$ with $0< s<1$, but grow to become larger than any pre-assigned factor $K\gg1$ in the same $H^s$ (higher Sobolev norms $H^s$ with $s>1$ decrease, but they are large for all times). In the latter case, the solutions that exhibit forward cascade start in an arbitrarily small vicinity of a finite-gap torus in $L^2(\T^2)$, but their $H^s$ Sobolev norm (for $s>1$) exhibits a growth by a large multiplicative factor $K\gg 1$ after a large time. We shall comment further on those results after we state the theorems precisely.  

To do that, we need to introduce the 
Birkhoff coordinates for equation \ref{def:NLS1D}. Gr\'ebert and Kappeler showed in 
\cite{grebert_kappeler} that there exists a globally defined  map, called the 
Birkhoff map, such that $\forall s \geq 0$
\begin{equation}\label{def:BirkhoffMap}
 \begin{split}
  \Phi :\,& H^s(\T) \longrightarrow\,\,\, h^s(\Z)\times h^s(\Z)\\
 &\quad q \quad\,\longmapsto\ (z_m,\overline z_m)_{m\in\Z},
 \end{split}
\end{equation}
such that equation \eqref{def:NLS1D} is transformed in the new coordinates $(z_m,\overline z_m)_{m\in\Z}=\Phi (q)$ to:
\begin{equation}\label{def:NLSinBirkhoff}
 \im \dot z_m=\alpha_m(I)z_m
\end{equation}
where $I=(I_m)_{m \in \Z}$ and $I_m=|z_m|^2$ are the actions, which are conserved in time (since $\alpha_m(I)\in \R$). Therefore in these 
coordinates, called 
Birkhoff coordinates, equation \eqref{def:NLS1D} becomes a chain of 
nonlinear harmonic oscillators and it is clear that  the phase space is foliated 
by finite and infinite dimensional tori 
with periodic, quasiperiodic or almost periodic dynamics, depending on how many of the actions $I_m$ (which are constant!) are nonzero and on the properties of rational dependence of the frequencies.

In this paper we are interested in the finite dimensional tori with 
quasiperiodic dynamics. Fix $\tk\in \N$ and  consider a set of modes
\begin{equation}\label{def:SetOfModes}
\cS_0=\{\tm_1, \ldots,\tm_\tk\}\subset \Z\times \{0\}.
\end{equation}
Fix also a value for the actions $I_{\tm_i}=I_{\tm_i}^0$ for $i=1,\ldots \tk$. 
Then we can define the $\tk$-dimensional torus
\begin{equation}\label{def:torus}
 \tT^\tk=\tT^\tk(\cS_0,I^0_m)=\left\{z\in\ell^2: |z_{\tm_i}|^2=I_{\tm_i}^0, 
\,\,\text{ for }\,i=1,\ldots, \tk, \ \ \ z_m=0\,\,\text{ 
for }\, m\not\in \cS_0\right\},
\end{equation}
which is supported on the set $\cS_0$. Any orbit on this torus is quasiperiodic 
(or periodic if the frequencies of the rigid rotation are completely resonant). 
We will impose conditions to have non-resonant quasiperiodic dynamics. This will imply 
that the 
orbits on $\tT^\tk$ are dense. By equation \eqref{def:NLSinBirkhoff}, it is clear 
that this torus, as an invariant object of equation \ref{def:NLS1D}, is stable for this equation for all times 
in the sense of Lyapunov.

The torus \eqref{def:torus} {(actually, its pre-image $\Phi^{-1}(\tT^\td)$ though the Birkhoff map)} is also an invariant object for 
the original equation \eqref{NLS}. The main result of this paper will show the 
instability (in the sense of Lyapunov) of this invariant object. 
Roughly speaking, we show that under certain assumptions (on the choices of modes \eqref{def:SetOfModes} and actions 
\eqref{def:torus}) these tori are unstable in the $H^s(\T^2)$
topology for $s\in (0,1)$. Even more, there exist orbits which 
start 
arbitrarily close to these tori and undergo an arbitrarily large $H^s$-norm 
explosion.

We will abuse notation, and identify $H^s(\T)$  with the closed subspace of $H^s(\T^2)$  of functions 
depending only on the $x$ variable. Consequently, $\mathcal T^\tk:=\Phi^{-1}(\tT^\tk)$ (see 
\eqref{def:BirkhoffMap}) is a closed torus of $H^s(\T)\subset H^s(\T^2)$.

%

\begin{theorem}\label{thm:main}
Fix a positive integer $\td$. For any choice of $\td$ modes $\cS_0$ (see 
\eqref{def:SetOfModes}) satisfying a genericity condition (namely 
Definition \ref{Lgenericity} with sufficiently large $\tL$), there exists $\eps_\ast>0$ 
 such that for any $\eps\in 
(0,\eps_\ast)$ there exists a positive measure Cantor-like set $\cI\subset 
(\eps/2,\eps)^\tk$ of actions, for which the following holds true for any torus  $\tT^\tk=\tT^\tk(\cS_0,I^0_m)$ with  $I_m^0\in \cI$:

\begin{enumerate}

\item For any $s\in (0,1)$, $\de>0$ small enough, and $K>0$ large enough, there 
exists an 
orbit $u(t)$ of  \eqref{NLS} and a time 
\[
0<T\leq e^{\left(\frac{K}{\de}\right)^\beta}
\]
such that $u(0)$ is $\de$-close 
to the torus $\mathcal T^\tk:=\Phi^{-1}(\tT^\tk)$ in $H^s(\T^2)$ and $\|u(T)\|_{H^s}\geq K$. Here $\beta>1$ is independent of $K, \delta$.

\item For any $s>1$, and any $K>0$ large enough, there exists an
orbit $u(t)$ of  \eqref{NLS} and a time 
\[
0<T\leq e^{K^\sigma}
\]
such that  
$$
\mathrm{dist}\left(u(0), \mathcal T^\tk \right)_{L^2(\T^2)}\leq K^{-\sigma'}\quad \text{ and }\quad\|u(T)\|_{H^s(\T^2)}\geq K\|u(0)\|_{H^s(\T^2)}.
$$
Here $\sigma, \sigma'>0$ are independent of $K$.
 \end{enumerate}
\end{theorem}

{\bf 1.4. Comments and remarks on Theorem \ref{thm:main}:}

\begin{enumerate}
\item 
The relative measure of the set $\cI$ of admissible actions can be taken as close to 1 as desired. Indeed, by taking smaller $\eps_\ast$, one has that 
the relative measure satisfies
\[
 |1-\mathrm{Meas}(\cI)|\leq C\eps_\ast^\kappa
\]
for some constant $C>0$ and $0<\kappa<1$ independent of $\eps_\ast>0$. The genericity condition on the set $\cS_0$ and the actions $(I_{\tm})_{\tm \in \cS_0}\in \cI$ ensure that the \emph{linearized} dynamics around the resulting torus $\mathcal T^\tk$ is stable for the perturbations we need to induce the nonlinear instability. In fact, a subset of those tori is even linearly stable for much more general perturbations as we remark below. 



\item \textit{Why does the finite gap solution need to be small?} To prove 
Theorem 
\ref{thm:main} we need to analyze the linearization of equation \eqref{NLS} at the finite gap solution (see Section \ref{sec:reducibility}). 
Roughly speaking, this leads to a Schr\"odinger equation with a quasi-periodic potential. Luckily, such operators can be \emph{reduced} to constant coefficients via a KAM scheme. This is known as \emph{reducibility theory} which allows one to construct a change of variables that casts the linearized operator into an essentially constant coefficient diagonal one. This KAM scheme was carried out in \cite{Maspero-Procesi}, and requires  the quasi-periodic potential, given by the finite gap solution here, to be small for the KAM iteration to converge. That being said, we suspect a similar result to be true for non-small finite gap solutions.


\item To put the complexity of this result in perspective, it is instructive to compare it with the stability result in \cite{Maspero-Procesi}. In that paper, it is shown that 
 a proper {\sl subset} $\cI' \subset \cI$ of the tori considered in Theorem \ref{thm:main} are Lyapunov stable in $H^s$, $s>1$, but for shorter time scales than those considered in this theorem. More precisely, all orbits that are initially $\de$-close to $\mathcal T^\tk$ in $H^s$ stay $C\de$-close for some 
fixed $C>0$  for time scales $t\sim \delta^{-2}$.  The same stability result (with a completely identical proof) holds if we replace $H^s$ by $\mathcal F\ell_1$ norm (functions whose Fourier series is in $\ell^1$).  In fact, by trivially modifying the proof, one could also prove stability on the $\delta^{-2}$ timescale in $\mathcal F\ell_1\cap H^s$ for $0<s<1$. What this means is that the solutions in the first part of Theorem \ref{thm:main} remains within $C\delta$ of $\tT^\tk$ up to times $\sim \delta^{-2}$ but can diverge vigorously afterwards at much longer time scales. 

It is also worth mentioning that the complementary subset $\cI \setminus \cI'$ has a positive measure subset where tori are linearly unstable since they possess a finite set of 
modes that exhibit hyperbolic behavior. In principle, hyperbolic directions are 
good for instability, but they are not useful for our purposes since they live 
at very low frequencies, and hence cannot be used (at least not by themselves 
alone) to produce a substantial growth of Sobolev norms. We avoid dealing with 
these linearly unstable directions by restricting our solution to an 
invariant subspace on which these modes are at rest.

\item 
 It is expected that a similar statement to the first part of Theorem \ref{thm:main} is also true for 
$s>1$. This would be a stronger instability compared to that in the second part (for which the initial perturbation is small in $L^2$ but not in $H^s$). Nevertheless, this case cannot be tackled with the techniques considered 
in this paper. Indeed, one of the key points in the proof is to perform a
(partial) Birkhoff normal form up to order 4 around the finite gap solution. 
The terms which lead to the instabilities in Theorem \ref{thm:main} are quasi-resonant instead of being completely resonant. Working in the $H^s$ topology with $s\in (0,1)$, such terms can be considered completely resonant with little error on the timescales where instability happens. However, this cannot be done for $s>1$, for which one might be able to eliminate those terms by a higher order normal form ($s>1$ gives a stronger 
topology and can thus handle worse small divisors). This would mean that one needs 
other resonant terms to achieve growth of Sobolev norms. The same difficulties 
were encountered in \cite{Hani12} to prove the instability of the plane waves 
of \eqref{NLS}.


\item For finite dimensional Hamiltonian dynamical systems, proving 
Lyapunov instability for quasi-periodic Diophantine elliptic (or maximal 
dimensional Lagrangian) tori is an extremely difficult task. Actually all the 
obtained results \cite{ChengZ13,GuardiaK14} deal with $C^r$ or $C^\infty$ 
Hamiltonians, and not a single example of such instability is known for analytic 
Hamiltonian systems. In fact, there are no results of instabilities in 
the vicinity of non-resonant elliptic critical points or periodic orbits for 
analytic Hamiltonian systems (see \cite{LeCalvezDou83,Douady88, KaloshinMV04} 
for results on the $C^\infty$ topology). The present paper proves the existence 
of unstable Diophantine elliptic tori in an analytic infinite dimensional Hamiltonian system. 
Obtaining such instabilities in infinite dimensions is, in some sense, easier: 
having infinite dimensions gives ``more room'' for instabilities.

\item It is well known that many Hamiltonian PDEs possess quasiperiodic invariant 
tori 
\cite{Wayne90,Poschel96a,kuksin_poschel,Bourgain98,BB1,Eliasson10,GYX,BBi10,
Wang2,PX,
BCP,PP, PP13,BBHM18}. Most of these tori are normally elliptic and thus linearly 
stable. It is widely expected that the behavior given by Theorem \ref{thm:main} 
also arises in the neighborhoods of (many of) those tori. Nevertheless, it is 
not clear how to apply the techniques of the present paper to these settings.

\end{enumerate}
\medskip

\medskip

{\bf 1.5. Scheme of the proof.}
Let us explain the main steps to prove Theorem \ref{thm:main}. 
\begin{enumerate}
 \item Analysis of the 1-dimensional cubic Schr\"odinger equation. We express 
the 1-dimensional cubic NLS in terms of the Birkhoff coordinates. We need a 
quite precise knowledge of the Birkhoff  map (see Theorem \ref{thm:dnls}). In 
particular, we  need that it ``behaves well'' in $\ell^1$.
This is done in the paper 
\cite{AlbertoVeyPaper} and summarized in  Section \ref{sec:AdaptedVarAndBirk}. In Birkhoff 
coordinates, the finite gap solutions are supported in a finite set of 
variables. We use such  coordinates to express the Hamiltonian 
\eqref{def:Ham:Original} in a more convenient way.


\item Reducibility of the 2-dimensional cubic NLS around a finite gap solution. We reduce the linearization of the vector field 
around the finite gap solutions to a constant coefficients diagonal vector 
field. This is done in \cite{Maspero-Procesi} and explained in Section 
\ref{sec:reducibility}. In Theorem \ref{thm:reducibility} 
we give the conditions to achieve full reducibility. In effect, this transforms the linearized operator around the finite gap into a constant coefficient diagonal (in Fourier space) operator, with eigenvalues $\{\Omega_{\jj}\}_{\jj\in \Z^2\setminus \cS_0}$. We give the asymptotics of these eigenvalues in Theorem 
\ref{thm:reducibility4}, which roughly speaking look like 
\begin{equation}\label{geraffe}
\Omega_{\jj}=|\jj|^2 +O(J^{-2})
\end{equation}
for frequencies $\jj=(m,n)$ satisfying $|m|, |n|\sim J$. This seemingly 
harmless $O(J^{-2})$ correction to the unperturbed Laplacian eigenvalues is sharp and will be
responsible for the restriction to $s\in (0,1)$ in the first part of Theorem 
\ref{thm:main} as we shall explain below. 

\item Degree three Birkhoff normal form around the finite gap solution. This is 
done in  \cite{Maspero-Procesi}, but we shall need more precise information from this normal form that will be crucial for Steps 5 and 6 below. This is done in \ref{sec:CubicBirkhoff} (see Theorem \ref{thm:3b}). 
\item Partial normal form of degree four. We remove all degree four monomials 
which are not (too close to) resonant. This is done in Section 
\ref{sec:QuarticBirkhoff}, and leaves us with a Hamiltonian with (close to) resonant degree-four terms plus a higher-degree part which will be treated as a remainder in our construction. 

\item We follow the paradigm set forth in \cite{CKSTT, GuardiaK12} to construct  
solutions to the truncated Hamiltonian consisting of the (close to) resonant 
degree-four terms isolated above, and then afterwards to the full Hamiltonian by 
an approximation argument. This construction will be done at frequencies 
$\jj=(m, n)$ such that $|m|,|n| \sim J$ with $J$ very large, and for 
which the dynamics is effectively given by the following system of ODE
$$
\begin{cases}
\im \dot a_{\jj} &= -|a_{ \jj}|^2a_{ \jj}+\sum_{\mathcal R(\jj)} a_{\jj_1}\overline{a_{\jj_2}} a_{\jj_3} e^{\im\Gamma t} \\
\mathcal R(\jj)&:=\{(\jj_1, \jj_2, \jj_3) \in \Z^2\setminus \cS_0: \jj_1, \jj_3\neq \jj,  \ \ \ \jj_1-\jj_2+\jj_3=\jj,  \ \ \ |\jj_1|^2-|\jj_2|^2+|\jj_3|^2=|\jj|^2\}\\
\Gamma&:= \Omega_{\jj_1}-\Omega_{\jj_2}+\Omega_{\jj_3}-\Omega_{\jj}.
\end{cases}
$$
We  remark that the conditions of the set $\mathcal R(\jj)$ are essentially 
equivalent to saying that $(\jj_1, \jj_2, \jj_3, \jj)$ form a rectangle in $\mathbb{Z}^2$. Also 
note that by the asymptotics of $\Omega_{\jj}$ mentioned above in \eqref{geraffe}, one obtains that 
$\Gamma=O(J^{-2})$ if all the frequencies involved are in $\mathcal R(\jj)$ and 
satisfy $|m|,|n| \sim J$. The idea now is to reduce this system into a 
finite dimensional system called the ``Toy Model'' which is tractable enough for 
us to construct a solution that cascades energy. An obstruction to this plan is 
presented by the presence of the oscillating factor $e^{\im \Gamma t}$ for which 
$\Gamma$ is not zero (in contrast to \cite{CKSTT}) but rather $O(J^{-2})$. 
The only way to proceed with this reduction is to approximate $e^{\im \Gamma t} 
\sim 1$ which is only possible provided $J^{-2}T \ll 1$. The solution coming 
from the Toy Model is supported on a finite number of modes $\jj\in \Z^2\setminus 
\cS_0$ satisfying $|j|\sim J$, and the time it takes for the energy to diffuse 
across its modes is $T\sim O(\nu^{-2})$ where $\nu$ is the characteristic size 
of the modes in $\ell^1$ norm. Requiring the solution to be initially close in 
$H^s$ to the finite gap would necessitate that $\nu J^s \lesssim \delta$ which 
gives that $T\gtrsim_\delta J^{-2s}$, and hence the condition $J^{-2}T \ll 1$ 
translates into the condition $s<1$. This explains the restriction to $s<1$ in 
the first part of Theorem \ref{thm:main}. If we only require our solutions to be 
close to the finite gap in $L^2$, then no such restriction on $\nu$ is needed, 
and hence there is no restriction on $s$ beyond being $s>0$ and $s\neq 1$, which 
is the second part of the theorem. 


This analysis is done in Section \ref{sec:ToyModel} and \ref{sec:Approximation}. In the former, we perform the reduction to the effective degree 4 Hamiltonian taking into account all the changes of variables performed in the previous sections; while in Section \ref{sec:Approximation} we perform the above approximation argument allowing to shadow the Toy Model solution mentioned above with a solution of \eqref{NLS} exhibiting the needed norm growth, thus completing the proof of Theorem \ref{thm:main}.

\end{enumerate}

\vspace{1em}
\noindent{\bf Acknowledgements:}
This project has received funding from the European Research Council (ERC) under the European Union's Horizon 2020 research and innovation programme (grant agreement No 757802) and under FP7- IDEAS (grant agreement No  306414). M. G. has been also partly supported by the Spanish MINECO-FEDER Grant MTM2015-65715-P. Z.~H.~ was partly supported by a Sloan Fellowship, and NSF grants DMS-1600561 and DMS-1654692.
A.M. was partly supported by  
 Progetto di Ricerca GNAMPA - INdAM 2018 ``Moti stabili ed instabili in equazioni di tipo Schr\"odinger''.

\section{Notation and functional setting}

\subsection{Notation}

For a complex number $z$, it is often convenient to use the notation
$$
z^{\sigma}=\begin{cases} z \qquad \text{if }\sigma=+1,\\
\bar z \qquad \text{if }\sigma=-1.
\end{cases}
$$
For any subset $\Gamma \subset \Z^2$, we denote by $h^s(\Gamma)$ the set of 
sequences $(a_{\jj})_{\jj \in \Gamma}$ with norm
$$
\|a\|_{h^s(\Gamma)}=\left(\sum_{\jj \in \Gamma}\langle \jj \rangle^{2s}|a_{\jj}|^2\right)^{1/2}<\infty.
$$

%

Our phase space will be obtained by an appropriate linearization around the  
finite gap solution with $\td$ frequencies/actions. For a finite set $\mathcal 
S_0 \subset \Z\times \{0\}$ of $\tk$ elements, we consider the phase space 
$\mathcal X= (\C^\tk\times\T^\tk)\times \ell^1(\Z^2 \setminus \cS_0)\times \ell^1(\Z^2 \setminus \cS_0)$. The first 
part $( \C^\tk \times \T^\tk)$ corresponds to the finite-gap sites in action angle 
coordinates, whereas  $\ell^1(\Z^2 \setminus \cS_0)\times \ell^1(\Z^2 \setminus \cS_0)$ corresponds to the 
remaining orthogonal sites in frequency space. We shall often denote the 
$\ell^1$ norm by $\|\cdot \|_1$. We shall denote variables on $\mathcal X$ by 

$$
\mathcal X\ni( \yy,\theta, \ba): \qquad     \yy \in \mathbb C^{\tk}, \  \ \theta \in \mathbb T^{\tk},  \ \  
 \ba=(a, \bar a) \in \ell^1(\Z^2\setminus \cS_0) \times \ell^1(\Z^2\setminus \cS_0).
$$
We shall use multi-index notation to write monomials like $\yy^{l}$ and $\fm_{\al, \bt}=a^\al \bar a^\bt$ where $l \in \mathbb N^{\tk}$ and $\alpha, \beta \in(\mathbb N)^{\Z^2 \setminus \cS_0}$. Often times, we will abuse notation, and simply use the notation $\ba \in \ell^1$ to mean $\ba=(a, \bar a) \in \ell^1(\Z^2\setminus \cS_0) \times \ell^1(\Z^2\setminus \cS_0)$, and $\|\ba\|_{1}=\|a\|_{\ell^1(\Z^2\setminus \cS_0)}$.

\begin{definition}\label{def:degree}
For a monomial of the form $ e^{\im \ell\cdot \theta} \, \yy^l \,  \fm_{\al,\bt}$, we define its degree to be $2|l|+|\alpha|+|\beta|-2$, where the modulus of a multi-index is given by its $\ell^1$ norm. 
\end{definition}

\subsection{Regular Hamiltonians}
Given a Hamiltonian function $F(\yy,\theta,\ba)$ on the phase space $\mathcal X$, we associate to it the Hamiltonian vector field
\[
X_F:=\{ -\partial_\theta F, \partial_\yy F, \ \ - \im \partial_{\bar a} F,  \im \partial_{a}F\},
\]
where we have used the standard complex notation to denote the Fr\'echet derivatives of $F$ with respect to the variable $\ba \in \ell^1$.

We will often need to complexify the variable $\theta \in \T^\tk$ into the domain
$$
\T^\tk_\rho:=\{\theta\in \C^\tk: {\rm Re }(\theta)\in \T^\tk \,,\quad |{\rm Im }(\theta)|\le \rho \} $$
and consider vector fields which are  functions from $$\C^\tk \times \T^\tk_\rho \times  \ell^1 \to  \C^{\tk}\times \C^\tk\times \ell^1\,:\;(\yy,\theta,\ba)\to (X^{(\yy)},X^{(\theta)},X^{(a)},X^{(\bar a )}) 
$$ which are analytic in $\yy,\theta,\ba$. 
Our vector fields will be defined on the domain 
\begin{equation}\label{def:domain}
D(\rho,r):=\T^\tk_\rho \times D(r) \ \ \ \mbox{ where } \ D(r):=\{ |\yy|\le r^2 ,\quad \|\ba\|_1 \le r \}.
\end{equation}
On the vector field, we use as norm
$$
\bnorm{X}_r:=|X^{(\theta)}|+\frac{|X^{(\yy)}|}{r^2}+ \frac{\|X^{(a)}\|_{1}}{r}+ \frac{\|X^{(\bar a)}\|_{1}}{r}.
$$
All Hamiltonians $F$ considered in this article are analytic, real valued and can be expanded 
in Taylor Fourier series which are well defined  and pointwise absolutely 
convergent
\begin{equation}
\label{h.funct}
F(\yy,\theta,\ba)= \sum_{\al ,\bt\in\N^{\Z^2\setminus\cS_0},\ell\in \Z^\tk, l\in \N^\tk } F_{\al,\bt,l,\ell} \  e^{\im \ell\cdot \theta} \, \yy^l \,  \fm_{\al,\bt}.
\end{equation}
Correspondingly we expand vector fields in Taylor Fourier series (again  well defined  and pointwise absolutely convergent)
$$
X^{(v)}(\yy,\theta,\ba)= \sum_{\al ,\bt\in\N^{\Z^2\setminus\cS_0},\ell\in \Z^\tk, l\in \N^\tk } X_{\al,\bt,l,\ell}^{(v)} \, e^{\im \ell\cdot \theta} \, \yy^l \,  \fm_{\al,\bt}\,,
$$
where $v$ denotes the components $\theta_i, \yy_i$ for $1 \leq i \leq \tk$ or 
$a_\jj,\bar a_\jj $ for   $\jj \in \Z^2\setminus\cS_0$.

To a vector field we associate its majorant 
$$
\und{X}_\rho^{(v)}[\yy,\ba]:= \sum_{\ell\in \Z^\tk,l\in \N^\tk,\al ,\bt\in\N^{\Z^2} } |X^{(v)}_{\al,\bt,\ell}| \, e^{\rho\, |\ell|} \, \yy^l \,  \fm_{\al,\bt}
$$
and require that this is an analytic  map on $D(r)$. 
Such a vector field is called \emph{majorant analytic}. Since Hamiltonian functions are defined modulo constants, we give the following definition of the {\em norm} of $F$:
$$
|F|_{\rho,r}:=\sup_{(\yy,\ba)\in D(r)} \bnorm{ \und{(X_F)}_\rho}_{r} \ . 
$$
Note that the norm $| \cdot |_{\rho,r}$ controls the $| \cdot |_{\rho',r'}$ whenever $\rho'< \rho$, $r'<r$. 

Finally, we will also consider Hamiltonians  $F(\lambda; \theta, a, \bar a) \equiv F(\lambda)$ depending on an external parameter $\lambda \in \cO\subset \R^\tk$. For those, we define the  {\em inhomogeneous Lipschitz} norm:
\[
|F|_{\rho,r}^\cO:= \sup_{\lambda \in \cO}|F(\lambda)|_{\rho,r}+  
\sup_{\lambda_1\neq \lambda_2\in \cO} \frac{|F(\lambda_1)-F(\lambda_2) 
|_{\rho,r}}{|\lambda_1-\lambda_2|} \ . 
\]


\subsection{Commutation rules}

Given two Hamiltonians $F$ and $G$, we define their Poisson bracket as {$ \{F, G\} := \di F(X_G)$}; in coordinates
$$
\{F, G\}=-\partial_\yy F \cdot \partial_\theta G+\partial_\theta F \cdot \partial_\yy G +\im \left(\sum_{\jj \in \Z^2\setminus \cS_0} \partial_{\bar a_\jj} F \partial_{a_\jj}G -\partial_{ a_\jj} F \partial_{\bar a_\jj}G\right).
$$
Given $\al,\bt\in \N^{\Z^2\setminus \cS_0}$  we denote $\fm_{\al,\bt}:= a^\al \bar a^\bt$.
To the monomial $e^{\im \ell\cdot \theta}\yy^l \fm_{\al,\bt}$ with $\ell\in \Z^\tk, l \in \mathbb N^{\tk}$ we associate various numbers.
We 
denote by 
\begin{equation}
\label{def.eta}
\eta(\alpha, \beta) := \sum_{\jj \in \Z^2\setminus \cS_0} (\alpha_\jj - \beta_\jj) \ , \quad \qquad \eta(\ell):= \sum_{i=1}^\tk \ell_i \ .
\end{equation} 
We also associate to $e^{\im \ell\cdot \theta}\yy^l\fm_{\al,\bt}$ the quantities  $\pi({\al,\bt})=(\pix,\piy)$ and $\pi(\ell)$ defined by 
\begin{equation}
\label{def.pi}
\pi({\al,\bt})=  \begin{bmatrix} \pi_x(\alpha, \beta) \\ \pi_y(\alpha, \beta) \end{bmatrix} = \sum_{\jj=(m,n)\in \Z^2\setminus \cS_0} 
\begin{bmatrix}  m \\ n \end{bmatrix} (\al_\jj-\bt_\jj) \ , \quad \qquad 
\pi(\ell)= \sum_{i=1}^\tk \tm_i \ell_i  \ . 
\end{equation}

The above quantities are associated with the following mass $\mathcal M$ and momentum $\mathcal P=(\mathcal P_x, \mathcal P_y)$ functionals given by


\begin{equation}
		\label{mp.1}
		\begin{aligned}
		&\mathcal M:=  \sum_{i=1}^\tk \yy_i + \sum_{\jj \in \Z^2\setminus \cS_0 }|a_\jj|^2 \\
		&\mathcal P_x:= \sum_{i=1}^\tk \tm_i  \yy_i + \sum_{(m,n) \in \Z^2 \setminus \cS_0}\!\!\!\!m \, |a_{(m,n)}|^2 \\
		&\mathcal P_y:= \sum_{(m,n) \in \Z^2 \setminus \cS_0} n |a_{(m,n)}|^2
		\end{aligned}
		\end{equation} 
via the following commutation rules: given a monomial $e^{\im 
\ell\cdot\theta}\yy^l\fm_{\al,\bt}$
\begin{align*}
\{\cM,e^{\im \ell\cdot\theta}\yy^l\fm_{\al,\bt}\}&=\im (\eta(\alpha, 
\beta)+\eta(\ell) )e^{\im  \ell\cdot\theta}\yy^l\fm_{\al,\bt}\\
\{\cP_x,e^{\im \ell\cdot\theta}\yy^l\fm_{\al,\bt}\}&=\im (\pi_x(\alpha, 
\beta)+\pi(\ell) )e^{\im \ell\cdot\theta}\yy^l\fm_{\al,\bt} \\ 
\{\cP_y,e^{\im \ell\cdot\theta}\yy^l\fm_{\al,\bt}\}&=\im \, \pi_y(\alpha, 
\beta)\,  e^{\im \ell\cdot\theta}\yy^l\fm_{\al,\bt}
\end{align*} 
\begin{remark}
	\label{leggi_sel}
	An analytic hamiltonian function  $\mathcal F$ (expanded as in \eqref{h.funct}) commutes with the mass $\cM$ and the momentum $\cP$ if and only if the following {\em selection rules} on its coefficients hold:
	\begin{align*}
	&\{ \cF, \cM\} = 0 \ \ \  \Leftrightarrow   \ \ \  \cF_{\alpha, \beta, l, \ell} \, (\eta(\alpha, \beta) + \eta(\ell)) = 0 \\
	& \{ \cF, \cP_x\} = 0  \ \ \ \Leftrightarrow  \ \ \    \cF_{\alpha, \beta, l, \ell} \, (\pi_x(\alpha, \beta) + \pi(\ell)) = 0 \\
	& \{ \cF, \cP_y\} = 0  \ \ \ \Leftrightarrow  \ \ \    \cF_{\alpha, \beta, l, \ell} \, (\pi_y(\alpha, \beta)) = 0 
	\end{align*}
	where $\eta(\alpha, \beta), \eta(\ell)$ are defined in \eqref{def.eta} and $\pi(\alpha, \beta), \pi(\ell)$ are defined in \eqref{def.pi}.
\end{remark}

\begin{definition}
	We will denote by $\cA_{\rho,r}$ the set of all real-valued Hamiltonians of the form \eqref{h.funct} with finite $| \cdot |_{\rho,r}$ norm and which Poisson commute with $\cM$, $\cP$. Given a compact set $\cO\subset \R^\tk$,  we denote by $\cA^\cO_{\rho,r}$ the Banach space of Lipschitz maps $\cO\to \cA_{\rho,r}$ with the norm $|\cdot|_{\rho,r}^\cO$.
\end{definition}

From now on, all our Hamiltonians will belong to some set $\cA_{\rho, r}$ for some $\rho, r>0$.

\section{Adapted variables and Hamiltonian formulation}\label{sec:AdaptedVarAndBirk}

\subsection{Fourier expansion and phase shift}\label{sec:FourierPhase}

Let us start by expanding $u$ in Fourier coefficients 
$$
u(x,y,t)= \sum_{\jj=(m,n)\in \Z^2} u_{\jj}(t) \, e^{\im(m x + n y)}.
$$
Then, the Hamiltonian $H_0$ introduced in \eqref{def:Ham:Original} can be written as
\begin{align*}
H_0(u)=&\sum_{\jj\in \Z^2}|\jj|^2 |u_{\jj}|^2 + \frac12\sum_{\jj_i\in \Z^2 \atop \jj_1-\jj_2+\jj_3-\jj_4=0}u_{\jj_1}\bar u_{\jj_2}u_{\jj_3}\bar u_{\jj_4}\\
=& \sum_{\jj\in \Z^2}|\jj|^2 |u_{\jj}|^2 -\frac12 \sum_{\jj\in \Z^2}|u_{\jj}|^4 +2\overbrace{\left(\sum_{\jj\in \Z^2}|u_\jj|^2\right)^2}^{M(u)^2}+ \frac12\sum_{\jj_i\in \Z^2 \atop \jj_1-\jj_2+\jj_3-\jj_4=0}^{\star}u_{\jj_1}\bar u_{\jj_2}u_{\jj_3}\bar u_{\jj_4} 
\end{align*}
where the $\sum^\star$ means the sum over the quadruples $\jj_i$ such that $\{\jj_1,\jj_3\}\neq \{\jj_2,\jj_4\}$.

Since the mass $M(u)$ in \eqref{def:NLS:mass} is a constant of motion, we make 
a trivial phase shift and consider an equivalent Hamiltonian 
$H(u)=H_0(u)-M(u)^2$, 
\begin{equation}
\label{parto}
H(u) =  \int_{\T^2} \abs{\nabla u(x,y)}^2 \, \di x \, \di y + 
\frac{1}{2}\int_{\T^2}\abs{u(x,y)}^4 \, \di x \, \di y -  M(u)^2
\end{equation}
corresponding to the Hamilton equation
\begin{equation}\label{piri0}
\im \partial_t u = -\Delta u + |u|^2 u -2 M(u) u \ , \qquad (x,y)\in\T^2\ .
\end{equation}
Clearly the solutions of \eqref{piri0} differ from the solutions of \eqref{NLS}  only by a phase shift\footnote{In order to show the equivalence we consider any solution $u(x,t)$ of \eqref{piri0} and consider the invertible map 
$$
u\mapsto v= u\; e^{-2 \im M(u) t } \quad \mbox{with inverse}\quad  v\mapsto u= v\; e^{2 \im M(v) t }. 
$$
Then a direct computation shows that $v$ solves 2D-NLS.
}.  Then, 
\begin{equation}
\label{Ha0}
H(u) = \sum_{\jj\in \Z^2}|\jj|^2 |u_{\jj}|^2 -\frac12 \sum_{\jj\in \Z^2}|u_{\jj}|^4 + \frac12\sum_{\jj_i\in \Z^2 \atop \jj_1-\jj_2+\jj_3-\jj_4=0}^{\star}u_{\jj_1}\bar u_{\jj_2}u_{\jj_3}\bar u_{\jj_4}.
\end{equation}

\subsection{The Birkhoff map for the 1D cubic NLS}\label{sec:1DNLS}
We devote this section to  gathering some properties of the Birkhoff map for the 
integrable 1D NLS equation. These will be used to write the Hamiltonian 
\eqref{Ha0} in a more convenient way. The main reference for this 
section is \cite{AlbertoVeyPaper}.
We shall denote by $B^{s}(r)$ the ball of radius $r$ and center $0$ in the topology of $h^s \equiv h^s(\Z)$. 

\begin{theorem}
\label{thm:dnls}
There exist $r_* >0$  and a symplectic, real analytic map $\Phi$  
with  $\di\Phi(0) = \uno $ such that $\forall s \geq0$ one has the following 
\begin{itemize}
\item[(i)]  ${\Phi} :  B^{s}(r_*) \to h^s$. More precisely, there exists a 
constant  $C>0$ such that for all $0 \leq  r \leq r_*$
\[\sup_{\norm{q}_{h^s}  \leq r} \norm{({\Phi }- \uno)(q)}_{h^s} \leq  C \, 
 r^3  \ .\]

The same estimate holds for $\Phi^{-1}-\uno$ or by replacing the space $h^s$ with the space $\ell^1$.

\item[(ii)]   Moreover, if $q\in h^s$ for $s \geq 1$,   $\Phi$ introduces local Birkhoff coordinates for (NLS-1d) in $h^s$ as follows: the integrals of motion of (NLS-1d) are real analytic functions of the actions $I_j = |z_j|^2$ where $(z_j)_{j \in \Z}=\Phi(q)$. In particular, the Hamiltonian $H_{{\rm NLS1d}}(q) \equiv  \int_\T \abs{\derx q(x)}^2 dx-M(q)^2+ \frac{1}{2} \int_\T \abs{q(x)}^4 dx$, 
the mass $M(q):= \int_\T \abs{q(x)}^2 dx$ and the momentum $P(q):= -\int_\T \bar q(x) \im \derx q(x) dx$ have the form
\begin{align}
\label{ham.bc}
&\left(H_{{\rm NLS1d}} \circ \Phi^{-1}\right)(z) \equiv h_{\rm 
nls1d}\left((|z_m|^2)_{m \in \Z}\right) = \sum_{m\in \Z} m^2 |z_m|^2  - 
\frac{1}{2} \sum_{m\in \Z} |z_m|^4  + O(|z|^6) \ , \\
\notag
&\left(M\circ \Phi^{-1}\right)(z) = \sum_{m \in \Z} |z_m|^2 \ , \\
\notag&\left(P\circ \Phi^{-1}\right)(z) = \sum_{m \in \Z}  m |z_m|^2 \ . 
\end{align}
\item[(iii)] Define the (NLS-1d)  action-to-frequency map  $I \mapsto \alpha^{{\rm nls1d}}(I)$, where
$
\alpha^{{\rm nls1d}}_m(I) := \frac{\partial h_{\rm nls1d}}{\partial I_m}$, $\forall m \in \Z.$
Then one has the asymptotic expansion 
\begin{equation}\label{freq.bc}
\alpha^{{\rm nls1d}}_m(I)  =  m^2  - I_m + \frac{\varpi_m(I)}{\la m\ra}
\end{equation}
where $\varpi_m(I)$ is at least quadratic in $I$.
\end{itemize}
\end{theorem}
\begin{proof}
Item $(i)$ is the main content of \cite{AlbertoVeyPaper}, where it is proved that the Birkhoff map is majorant analytic between some Fourier-Lebesgue spaces. Item $(ii)$ is proved in \cite{grebert_kappeler}. Item $(iii)$ is Theorem 1.3 of  \cite{KST}.
\end{proof}

\begin{remark}
 Theorem \ref{thm:dnls}  implies that all  solutions of 1D NLS have Sobolev norms uniformly bounded  in time (as it happens  for other integrable systems, like  KdV and Toda lattice, see e.g. \cite{BambusiM16, Kappeler16}). On the contrary, the Szeg\H{o} equation is an integrable system which exhibits growth of Sobolev norms \cite{GerardG15}.
\end{remark}


\subsection{Adapted variables}
The aim of this section is to write the   Hamiltonian \eqref{parto}, the mass $M$  \eqref{def:NLS:mass} and the momentum $P$ \eqref{def:NLS:momentum} in the local variables around the finite gap solution corresponding to 
\[
\begin{cases}
|z_{\tm_k}|^2&=I_k , \qquad k=1, 2, \ldots, \tk \label{finitegapZ}\\
z_m&=0 , \qquad \,\  m\in \Z\setminus \cS_0.
\end{cases}
\]
To begin with, we start from the Hamiltonian in Fourier coordinates \eqref{Ha0}, 
and set
\[
q_m:= u_{(m,0)}\quad \mbox{ if } \  m\in \Z\,,\qquad a_{\jj}= u_{\jj} \quad \mbox{ if } \ \jj=(m,n)\in \Z^2\,,\; n\neq 0 \ .
\] 
We rewrite the  Hamiltonian accordingly in increasing degree in $a$, obtaining 
\begin{align}
\notag
H(q, a)= 
& \sum_{m \in \Z} m^2 |q_m|^2 -\frac12 \sum_{m \in \Z} |q_m|^4 +  \frac12\sum_{m_i\in \Z \atop m_1-m_2+m_3-m_4=0}^{\star}q_{m_1}\bar q_{m_2}q_{m_3}\bar q_{m_4}+\\
\notag
 & + \sum_{\jj\in \Z^2\setminus \Z} |\jj|^2 |a_{\jj}|^2 + 2\sum^\star_{\jj_i=(m_i,n_i)\,,i=3,4\,,\; n_i\neq 0 \atop {
m_1-m_2+m_3-m_4=0\atop  n_3-n_4=0}} q_{m_1}\bar q_{m_2} a_{\jj_3}\bar a_{\jj_4}  +\operatorname{Re} \sum_{ \jj_i=(m_i,n_i)\,,i=2,4\,,\;n_i\neq 0 \atop {
m_1-m_2+m_3-m_4=0 \atop  n_2+n_4=0} }  \bar q_{m_1} a_{\jj_2}  \bar q_{m_3}  a_{\jj_4}\\
\notag
 & + 2 \operatorname{Re } \sum_{ \jj_i=(m_i,n_i)\,,i=2,3,4\,,\; n_i\neq 0
\atop {m_1-m_2+m_3-m_4=0\atop  -n_2+n_3-n_4=0}} q_{m_1} \bar a_{\jj_2} a_{\jj_3} \bar a_{\jj_4} \\
\notag
& + \frac{1}{2}\sum^\star_{ \jj_i=(m_i,n_i)\,,i=1,2,3,4\,,\;  n_i\neq 0\atop{ \jj_1-\jj_2+\jj_3-\jj_4 = 0}} a_{\jj_1} \bar a_{\jj_2} a_{\jj_3} \bar a_{\jj_4}-\frac12\sum_{\jj\in \Z^2\setminus \Z}|a_\jj|^4\\
&=: H_{\rm nls1d}(q)+H^{\rm II}(q, a)+H^{\rm III}(q, a)+H^{\rm IV}(a).\notag
\end{align}
{\bf Step 1:} First we do the following change of coordinates, which amounts to 
introducing Birkhoff coordinates on the line $\Z\times \{0\}$. We set
\begin{align}
\notag&\left( (z_m)_{m \in \Z}, (a_{\jj})_{\jj\in \Z^2\setminus 
\Z}\right)\mapsto \left((q_m)_{m \in \Z}, (a_{\jj})_{\jj \in \Z^2\setminus 
\Z}\right)\\
\notag&(q_m)_{m \in \Z} =\Phi^{-1}\left((z_m)_{m \in \Z}\right) , \ \ \  
a_{\jj}=u_{\jj}, \quad \jj\in \Z^2\setminus \Z.
\end{align}
In those new coordinates, the Hamiltonian becomes
\begin{align*}
H(z, a)=&H_{\rm nls1d}(\Phi^{-1}(z))+H^{\rm II}(\Phi^{-1}(z), a)+H^{\rm 
III}(\Phi^{-1}(z), a)+H^{\rm IV}(a),
 \end{align*}
where
\[
  \quad H_{\rm nls1d}(\Phi^{-1}(z))=h_{\rm nls1d}((|z_m|^2)_{m \in \Z}) .\\
\]
{\bf Step 2:} Next, we go to action-angle coordinates only on the set 
$\cS_0=\{\tm_1, \ldots, \tm_d\}\subset \Z\times \{0\}$ and rename $z_m$ for $m 
\notin \cS_0$ as $a_{(m, 0)}$, as follows
\begin{align*}
\left(\yy_i, \theta_i, a_{\jj}\right)_{\substack{1\leq i \leq \td\\\jj \in \Z^2\setminus \cS_0}} &\mapsto (z_m, a_{\jj})_{m \in \Z, \jj \in \Z^2\setminus \Z}\\
z_{\tm_i}&= \sqrt{I_i +\yy_i} \ e^{\im \theta_i}, \qquad \tm_i \in \cS_0, \\
z_{m}&=a_{(m, 0)}, \qquad m\in \Z\setminus \cS_0,\\
a_{\jj}&=a_{\jj}, \qquad \jj \in \Z^2\setminus \Z.
\end{align*}
In those coordinates, the Hamiltonian becomes (using \eqref{ham.bc})
\begin{align}
\mathcal{H}(\yy, \theta, a)=&\ h_{\rm nls1d}(I_1 +\yy_1, \ldots, I_\td +\yy_\td,  \left(|a_{(m, 0)}|^2\right)_{m \notin \cS_0}) \label{penguin1}\\
&+H^{\rm II}\left(\Phi^{-1}\left(\sqrt{I_1 +\yy_1}e^{\im \theta_1}, \ldots, \sqrt{I_\td +\yy_\td}e^{\im \theta_\td}, (a_{(m, 0)})_{m \notin \cS_0}\right), (a_{(m,n)})_{n\neq 0}\right) \label{penguin2}\\
&+H^{\rm III}\left(\Phi^{-1}\left(\sqrt{I_1 +\yy_1}e^{\im \theta_1}, \ldots, \sqrt{I_\td +\yy_\td}e^{\im \theta_\td}, (a_{(m, 0)})_{m \notin \cS_0}\right), (a_{(m,n)})_{n\neq 0}\right)\label{penguin3}\\
&+H^{\rm IV}\left((a_{(m,n)})_{n\neq 0}\right)\label{penguin4}.
\end{align}
{\bf Step 3:} Now, we expand each line by itself. By Taylor expanding around the 
finite-gap torus corresponding to $(\yy, \theta, a)=(0,\theta, 0)$ we obtain, 
up to an additive constant,

\begin{align*}
h_{\rm nls1d}\left(I_1 +\yy_1, \ldots, I_\td +\yy_\td, (|a_{(m, 0)}|^2)_{m \notin \cS_0}\right)=&\sum_{i=1}^{\tk}\partial_{\tm_i} h_{\rm nls1d}(I_1, \ldots, I_\tk, 0)\yy_i\\
&+\sum_{m \in \Z\setminus \cS_0}\partial_{m} h_{\rm nls1d}(I_1, \ldots, I_\tk, 0)|a_{(m, 0)}|^2\\
&-\frac{1}{2}\left(|\yy|^2+\sum_{m \in \Z\setminus \cS_0}|a_{(m, 0)}|^4\right)\\
&+O\left(|I| \left\{\sum_{j=1}^\tk \yy_j +\sum_{m \notin \cS_0}|a_{(m,0)}|^2\right\}^2\right)\\
&+O\left( \left\{\sum_{j=1}^\tk \yy_j +\sum_{m \notin 
\cS_0}|a_{(m,0)}|^2\right\}^3\right),
\end{align*}
where we have used formula \eqref{ham.bc} in order to deduce that $\frac{\partial^2 h_{nls1d}}{\partial I_m \partial I_n}(0)=-\delta_{n}^m$ where $\delta_n^m$ is the Kronecker delta.

The following lemma follows easily from Theorem \ref{thm:dnls} (particularly formulae \eqref{ham.bc} and \eqref{freq.bc}):

\begin{lemma}[Frequencies around the finite gap torus] Denote 
$$
\partial_{I_{\tm_j}} h_{\rm nls1d}(I_1, \ldots, I_\tk, 0)=\tm_j^2-\widetilde \lambda_j(I_1, \ldots, I_\tk).
$$
Then,
\begin{enumerate}
\item 
The map $(I_1, \ldots, I_{\tk}) \mapsto \widetilde\lambda(I_1, \ldots, I_\tk)=(\widetilde\lambda_i (I_1, \ldots, I_{\tk}))_{1\leq i \leq \tk}$ is a diffeomorphism from a small neighborhood of 0 of $\R^\tk$ to a small neighborhood of 0 in $\R^\tk$. Indeed, $\widetilde \lambda=$Identity +(quadratic in $I$).
More precisely, there exists $\e_{1d}>0$ such that if $0<\e< \e_{1d}$ and  
$$
\widetilde \lambda(I_1, \ldots, I_\tk)=\e \lambda , \quad \lambda \in 
\left(\frac12, 1\right)^\tk
$$
then $(I_1, \ldots, I_{\tk})=\e \lambda +O(\e^2)$. From now on, and to simplify notation, we will use the vector $\lambda$ as a parameter as opposed to $(I_1, \ldots, I_{\tk})$, and we shall set the vector
$$
\omega_i(\lambda)=\tm_i^2-\e \lambda_i, \qquad 1\leq i\leq \tk
$$ 
to denote the frequencies at the tangential sites in $\cS_0$.
\item For $m \in \Z\setminus \cS_0$, denoting $\Omega_m(\lambda):=\partial_{I_m} h_{\rm nls1d}(I_1(\lambda), \ldots, I_\tk(\lambda), 0)$, we have
\begin{equation}\notag
\Omega_m(\lambda) := m^2 +\frac{\varpi_m(I(\lambda))}{\la m\ra}\,,\quad \mbox{with} \ \ \  \sup_{\lambda\in (\frac{1}{2}, 1)^\td}\, \sup_{m \in \Z }|\varpi_m(I(\lambda))| \le C\e^2 \ . 
\end{equation}
\end{enumerate}
\end{lemma}
With this in mind,  line \eqref{penguin1} becomes 
\begin{align*}
h_{\rm nls1d}\left(I_1 +\yy_1, \ldots, I_\td +\yy_\td, (|a_{(m, 0)}|^2)_{m \notin \cS_0}\right)=\,&\omega(\lambda)\cdot \yy +\sum_{m \in \Z\setminus \cS_0}\Omega_m(\lambda) \left|a_{(m,0)}\right|^2\\
&-\frac{1}{2}\left(|\yy|^2+\sum_{m \in \Z\setminus \cS_0}\left|a_{(m, 0)}\right|^4\right)\\
&+O\left(|I| \left\{\sum_{j=1}^\tk \yy_j +\sum_{m \notin \cS_0}\left|a_{(m,0)}\right|^2\right\}^2\right)\\
&+O\left( \left\{\sum_{j=1}^\tk \yy_j +\sum_{m \notin \cS_0}\left|a_{(m,0)}\right|^2\right\}^3\right).
\end{align*}
We now analyze \eqref{penguin2}. This is given by
\begin{align*}
\eqref{penguin2}=\sum_{\jj\in \Z^2\setminus \Z} |\jj|^2 |a_{\jj}|^2 + 2\sum^\star_{\jj_i=(m_i,n_i)\,,i=3,4\,,\; n_i\neq 0 \atop {
m_1-m_2+m_3-m_4=0\atop  n_3-n_4=0}} q_{m_1}\bar q_{m_2} a_{\jj_3}\bar a_{\jj_4}  +\operatorname{Re} \sum_{ \jj_i=(m_i,n_i)\,,i=2,4\,,\;n_i\neq 0 \atop {
m_1-m_2+m_3-m_4=0 \atop  n_2+n_4=0} }  \bar q_{m_1} a_{\jj_2}  \bar q_{m_3}  a_{\jj_4}\\
\end{align*}
where we now think $q_m$ as a function of $\yy, \theta, a$.  By Taylor expanding it at $\yy = 0$ and $a = 0$, 
\begin{equation}\label{qmExpand}
\begin{split}
q_m=q_m(\lambda; \yy, \theta, (a_{(m_1, 0)})_{m_1 \in \Z\setminus \cS_0})
  = & \overbrace{q_m(\lambda;0, \theta, 0)}^{=:q_m^{\rm fg}(\lambda; \theta)} + \sum_{i=1}^\td \frac{\partial q_m}{\partial \yy_i}(\lambda;0,\theta,0) \yy_i \\
  &+ 
   \sum_{m_1 \in \Z \setminus \cS_0} \frac{\partial q_m}{\partial a_{(m_1,0)}}(\lambda;0,\theta,0) a_{(m_1, 0)}+\frac{\partial q_m}{\partial \bar a_{(m_1,0)}}(\lambda;0,\theta,0) \overline a_{(m_1, 0)} \\
   & + \sum_{\substack{m_1, m_2 \in \Z\setminus \cS_0\\ \sigma_1, \sigma_2 =\pm 
1}} Q_{m,m_1 m_2}^{\sigma_1 \sigma_2}(\lambda; \theta) 
a_{(m_1,0)}^{\sigma_1}a_{(m_2,0)}^{\sigma_2}+\cO(\yy^2, \yy a , a^3),
  \end{split}
  \end{equation}
  where we have denoted $(q_m^{\rm fg}(\lambda; \theta))_{m \in \Z}$ the finite gap 
torus (which corresponds to $\yy = 0$, $\ba = 0$), and 
  $$
 Q_{m,m_1 m_2}^{\sigma_1 \sigma_2}(\lambda; \theta)=\frac{1}{2} \frac{\partial^2 q_m}{\partial a_{m_1}^{\sigma_1}\partial a_{m_2}^{\sigma_2}}(\lambda;0,\theta,0).
  $$
Therefore, we obtain
\begin{align*}
\eqref{penguin2}=&\sum_{\jj\in \Z^2\setminus \Z} |\jj|^2 |a_{\jj}|^2 + 2\sum^\star_{\jj_i=(m_i,n_i)\,,i=3,4\,,\; n_i\neq 0 \atop {
m_1-m_2+m_3-m_4=0\atop  n_3-n_4=0}} q^{\rm fg}_{m_1}(\lambda; \theta) \bar q^{\rm fg}_{m_2}(\lambda; \theta)a_{\jj_3}\bar a_{\jj_4} \\
& +\operatorname{Re} \sum_{ \jj_i=(m_i,n_i)\,,i=2,4\,,\;n_i\neq 0 \atop {
m_1-m_2+m_3-m_4=0 \atop  n_2+n_4=0} }  
\bar q^{\rm fg}_{m_1}(\lambda; \theta) a_{\jj_2}\bar q^{\rm fg}_{m_3}(\lambda; \theta)  a_{\jj_4}\\
&+\left\{2\sum^\star_{\jj_i=(m_i,n_i)\,,i=3,4\,,\; n_i\neq 0 \atop {
m_1-m_2+m_3-m_4=0\atop  n_3-n_4=0}}\sum_{m_2'\in \Z\setminus \cS_0} {\frac{\partial \bar q_{m_2}}{\partial \bar a_{(m_2',0)}}(\lambda;0,\theta,0)} q^{\rm fg}_{m_1}(\lambda; \theta)\bar a_{(m_2', 0)} a_{\jj_3}\bar a_{\jj_4} +\text{similar cubic terms in } (a, \bar a)\right\}\\
&+\eqref{penguin2}^{(2)}+\eqref{penguin2}^{(\geq 3)}
\end{align*}
where $\eqref{penguin2}^{(2)}$ are degree 2 terms (cf. Definition \ref{def:degree}), $\eqref{penguin2}^{(\geq 3)}$ are of degree $\geq 3$. More precisely,  
\begin{equation}\label{penguin22}
\begin{split}
\eqref{penguin2}^{(2)}=&2\sum^\star_{\substack{\jj_i=(m_i,n_i)\,,i=3,4\,,\; n_i\neq 0 \\
m_1-m_2+m_3-m_4=0\atop  n_3-n_4=0\\1\leq i\leq \tk}} q^{\rm fg}_{m_1}(\lambda; \theta)  \frac{\partial \bar q_{m_2}}{\partial \yy_i}(\lambda;0,\theta,0) \yy_i a_{\jj_3}\bar a_{\jj_4}+\text{similar terms}\\
&+\sum^\star_{\substack{\jj_i=(m_i,n_i)\,,i=3,4\,,\; n_i\neq 0\\
m_1-m_2+m_3-m_4=0\\  n_3-n_4=0\\\sigma_1, \sigma_2=\pm 1,  m_1', m_2'\in \Z\setminus \cS_0}} L_{m_1, m_2, m_1', m_2'}^{\sigma_1, \sigma_2}(\lambda; \theta) a_{(m_1', 0)}^{\sigma_1}a_{(m_2', 0)}^{\sigma_2}a_{\jj_3}\bar a_{\jj_4}+\text{similar terms},
\end{split}
\end{equation}
for some uniformly bounded coefficients $L_{m_1, m_2, m_1', m_2'}^{\sigma_1, \sigma_2}$.

Next, we move on to \eqref{penguin3}, for which we have using equation \eqref{qmExpand}
\begin{equation}\label{penguin32}
\begin{split}
\eqref{penguin3}=&2 \operatorname{Re } \sum_{ \jj_i=(m_i,n_i)\,,i=2,3,4\,,\; n_i\neq 0
\atop {m_1-m_2+m_3-m_4=0\atop  -n_2+n_3-n_4=0}} q^{\rm fg}_{m_1}(\lambda; \theta) \bar a_{\jj_2} a_{\jj_3} \bar a_{\jj_4} \\
&+\underbrace{2 \operatorname{Re } \sum_{ \substack{\jj_i=(m_i,n_i)\,,i=2,3,4\,,\; n_i\neq 0
\\ m_1-m_2+m_3-m_4=0\\  -n_2+n_3-n_4=0}} \frac{\partial q_{m_1}}{\partial 
a_{(m_1',0)}}(\lambda;0,\theta,0) a_{(m'_1, 0)} \bar a_{\jj_2} a_{\jj_3} \bar 
a_{\jj_4} +\text{similar 
terms}}_{\eqref{penguin3}^{(2)}}+\eqref{penguin3}^{(\geq 3)},
\end{split}
\end{equation}
where $\eqref{penguin3}^{(2)}$ are terms of degree 2 and $\eqref{penguin3}^{(\geq 3)}$ are terms of degree $\geq 3$.

In conclusion, we obtain 

\begin{align}
\label{H.2}
\mathcal{H}(\lambda;\yy, \theta, \ba)  
 =  & \cN+\cH^\0(\lambda; \theta, {\bf a})+\cH^\1(\lambda; \theta, {\bf a})+\cH^\2(\lambda;  \yy, \theta, {\bf a})+\cH^{(\geq 3)}(\lambda;  \yy,\theta,  {\bf a}),
\end{align}
where 
\begin{equation}
\label{def:N}
\cN =    \sum_{i=1}^\tk \omega_{\tm_i} (\lambda) \yy_i + \sum_{m\notin \cS_0} \Omega_m(\lambda) |a_{(m,0)}|^2+ \sum_{\jj=(m,n) \in \Z^2 \atop  n\neq 0} |\jj|^2 |a_{\jj}|^2 \\
\end{equation}
\begin{align}
\cH^\0(\lambda; \theta, {\bf a})=\,&2\sum^\star_{\jj_i=(m_i,n_i)\,,i=3,4\,,\; n_i\neq 0 \atop {
m_1-m_2+m_3-m_4=0\atop  n_3-n_4=0}} q^{\rm fg}_{m_1}(\lambda; \theta) \bar q^{\rm fg}_{m_2}(\lambda; \theta)a_{\jj_3}\bar a_{\jj_4}  \label{def of H0}\\&+\operatorname{Re} \sum_{ \jj_i=(m_i,n_i)\,,i=2,4\,,\;n_i\neq 0 \atop {
m_1-m_2+m_3-m_4=0 \atop  n_2+n_4=0} }  \bar q^{\rm fg}_{m_1}(\lambda; \theta) a_{\jj_2}\bar q^{\rm fg}_{m_3}(\lambda; \theta)  a_{\jj_4} \notag\\
\cH^{(1)}(\lambda; \theta, {\bf a})=\,&2 \operatorname{Re } \sum_{ \jj_i=(m_i,n_i)\,,i=2,3,4\,,\; n_i\neq 0
\atop {m_1-m_2+m_3-m_4=0\atop  -n_2+n_3-n_4=0}} q^{\rm fg}_{m_1}(\lambda; \theta) \bar a_{\jj_2} a_{\jj_3} \bar a_{\jj_4} \label{def of H1}\\
&+2\sum^\star_{\jj_i=(m_i,n_i)\,,i=3,4\,,\; n_i\neq 0 \atop {
m_1-m_2+m_3-m_4=0\atop  n_3-n_4=0}}\sum_{m_2' \in \Z\setminus \cS_0} {\frac{\partial \bar q_{m_2}}{\partial \bar a_{(m_2',0)}}(\lambda;0,\theta,0)} q^{\rm fg}_{m_1}(\lambda; \theta)\bar a_{(m_2', 0)} a_{\jj_3}\bar a_{\jj_4}\notag\\
&+ \text{similar cubic terms in } (a, \bar a)\notag\\
\cH^{(2)}(\lambda; \theta, {\bf a})=\,&H^{\rm IV}\left((a_{(m,n)})_{n\neq 0}\right)-\frac{1}{2}\left(|\yy|^2+\sum_{m \in \Z\setminus \cS_0}|a_{(m, 0)}|^4\right) \label{giraffe2}\\
&+O\left(\varepsilon\left\{\sum_{j=1}^\tk \yy_j +\sum_{m \notin \cS_0}|a_{(m,0)}|^2\right\}^2\right)
+\eqref{penguin2}^{(2)}+\eqref{penguin3}^{(2)},\notag
\end{align}
where $\eqref{penguin2}^{(2)}$ and $\eqref{penguin3}^{(2)}$ were defined in \eqref{penguin22} and \eqref{penguin32} respectively. Finally, $\cH^{(\geq 3)}$ collects all remainder terms of degree $\geq 3$.

For short we write $\cN$ as $\cN=\omega(\lambda) \cdot \yy +  \cD$ where $\cD$ is the diagonal operator
\begin{equation}\notag
\cD:=\sum_{\jj=(m,n) \in \Z^2\setminus\cS_0} \Omega_\jj^\0\,  |a_\jj|^2   
\end{equation}
and the normal frequencies $\Omega_\jj^\0$ are defined by
\begin{equation}
\label{def:Omega}
 \Omega^{(0)}_\jj := \left\{ \begin{array}{ll} |\jj|^2 & \text{if}\; \jj=(m,n) \;\text{with} \; n\neq 0 \\ \Omega_m(\lambda) & \text{if}\; \jj=(m,0) , \ m \notin \cS_0 
 \end{array}\right. \ .
\end{equation}

Proceeding as in \cite{Maspero-Procesi}, one can prove the following result:
 \begin{lemma}
 \label{lem:norm.ham}
Fix $\rho >0$. There exists $\e_\ast>0$ and for any  $0 \leq \e \leq \e_\ast$, there exist $r_\ast \leq  \sqrt{\e}/4 $  and $C >0$ such that    $\cH^\0, \cH^\1, \cH^{(2)}$ and $\cH^{(\geq 3)}$    belong to $\cA_{\rho,r_\ast}^\cO$ and 
$\forall 0 < r \leq r_*$
\begin{equation}
\label{lem:norm.ham1}
|\cH^\0|_{\rho,r}^\cO \le C\e\,,\qquad |\cH^\1|_{\rho,r}^\cO \le C\sqrt{\e}r\,,\qquad |\cH^{(2)}|_{\rho,r}^\cO \le C r^2, \qquad 
|\cH^{(\geq 3)}|_{\rho,r}^\cO \le C \frac{r^3}{\sqrt{\e}} .
\end{equation}
\end{lemma}

\section{Reducibility theory of the quadratic part}\label{sec:reducibility}

In this section, we review the reducibility of the quadratic part $\cN+\cH^\0$ (see \eqref{def:N} and \eqref{def of H0})
of the Hamiltonian, which is the main part of the work \cite{Maspero-Procesi}. 
This will be a symplectic linear change of coordinates that transforms the 
quadratic part into an effectively diagonal,  time independent expression. 

\subsection{Restriction to an invariant sublattice $\Z^2_N$} For $N\in \N$, we define the sublattice $\Z_N^2:= \Z\times N\Z$ and remark that it is invariant for the flow in the sense that the subspace 
	$$
	E_N:=\{a_\jj=\bar a_\jj=0\,,\quad \mbox{for} \quad \jj\, \notin \Z^2_N\}
	$$
	is invariant for the original NLS dynamics and that of the Hamiltonian \eqref{H.2}. From now on, we restrict our system to this invariant sublattice, with 
	\begin{equation}\label{def:SizeN}
	N > \max_{1 \leq i \leq \td} |\tm_i|. 
	\end{equation}
	The reason for this restriction is that it simplifies (actually eliminates the need for) some genericity requirements that are needed for the work \cite{Maspero-Procesi} as well as some of the normal forms that we will perform later.

It will also be important to introduce the following two subsets of $\mathbb Z_N^2$: 
\begin{equation}\label{def:SetZ}
\sS:=\{(\tm, n): \tm \in \cS_0, \ n \in N \Z, \ n \neq 0\}, \qquad \fZ=\Z^2_N \setminus (\sS \cup \cS_0).
\end{equation}

%

\begin{definition}[$\tL-$genericity]\label{Lgenericity}
Given $\tL\in \N$, we say that $\cS_0$ is $\tL$-generic if it satisfies the condition
\begin{equation}\label{pop}
\sum_{i = 1}^\td \ell_i\tm_i\neq 0 \qquad \forall \; 0<|\ell|\le \tL.
\end{equation}
\end{definition}

\subsection{Admissible monomials and reducibility}

The reducibility of the quadratic part of the Hamiltonian will introduce a change of variables that modifies the expression of the mass $\cM$ and momentum $\cP$ as follows. 
Let us set
	\begin{equation}
		\label{mp.4}
		\begin{aligned}
		&\wtcM:=  \sum_{i=1}^\tk \yy_i + \sum_{(m , n) \in \fZ  }|a_j|^2  , \\
		&\wtcP_x:= \sum_{i=1}^\tk  \tm_i  \yy_i + \sum_{(m,n) \in\fZ}\!\!\!\!m \, |a_{(m,n)}|^2 ,  \\
		&\wtcP_y:= \sum_{(m,n) \in \Z_N^2} n |a_{(m,n)}|^2.
		\end{aligned}
		\end{equation} 
		
These will be the expressions for the mass and momentum after the change of variables introduced in the following two theorems. Notice the absence of the terms $\sum_{\substack{1\leq i \leq \tk\\n\in N\Z}} |a_{(\tm_i, n)}|^2$ and $\sum_{\substack{1\leq i \leq \tk\\n\in N\Z}} \tm_i |a_{(\tm_i, n)}|^2$ from the expressions of $\wtcM$ and $\wtcP_x$ above. These terms are absorbed in the new definition of the $\yy$ and ${\bf a}$ variables.

		\begin{definition}[Admissible monomials]
		\label{rem:adm3}
		Given $\bj = (\jj_1, \ldots, \jj_p) \in (\Z^2_N\setminus\cS_0)^p$, 
$\ell \in \Z^\tk$, $l \in \N^\td$, and $\sigma= (\sigma_1, \ldots, \sigma_p) \in 
\{-1, 1\}^p$, we  say that $(\bj , \ell, \sigma)$ is {\em admissible}, and 
denote $(\bj , \ell, \sigma) \in \mathfrak A_p$,  if the monomial $\mathfrak m= 
e^{\im \theta \cdot \ell} \cY^l\, a_{\jj_1}^{\sigma_1} \, \ldots 
a_{\jj_p}^{\sigma_p}$ Poisson commutes with $\wtcM,\wtcP_x, \wtcP_y$. We call a 
monomial $ e^{\im \theta \cdot \ell} \cY^l\, a_{\jj_1}^{\sigma_1} \, \ldots 
a_{\jj_p}^{\sigma_p}$ admissible if $(\bj , \ell, \sigma)$ is admissible. 
		\end{definition}

\begin{definition}\label{def:R2}
We define the {\em resonant set at degree 0}, 
\begin{equation}
\label{res2}
\fR_2:=\{ (\jj_1, \jj_2, \ell, \sigma_1, \sigma_2) \} \in \mathfrak A_2: 
\ell=0,  \ \ \sigma_1=-\sigma_2, \ \  \jj_1=\jj_2\}.
\end{equation}
\end{definition}

\begin{theorem}
	\label{thm:reducibility}
 	Fix $\e_0>0$ sufficiently small. There exist positive $\rho_0, \gamma_0, \tau_0, 
r_0, \tL_0$ (with $\tL_0$ depending only on $\td$) such that the following holds 
true uniformly for all $0<\e\le \e_0$:  For an $\tL_0$-generic choice of the set 
$\Tan$  (in the sense of Definition \ref{Lgenericity}), there exist a  compact 
{\em domain} $\cO_0 \subseteq (1/2,1)^\tk$, satisfying $| (1/2,1)^\tk\setminus 
\cO_0|\leq \e_0$, 
and Lipschitz (in $\lambda$) 
functions $\{\Omega_\jj\}_{\jj\in \Z_N^2\setminus \cS_0}$ defined on $\cO_0$ 
(described more precisely in Theorem \ref{thm:reducibility4} below) such that:
	
	\begin{enumerate}
	
\item  The set 
	\begin{equation}
	\label{2.mc}
	\cC^{(0)}:=\left\{\lambda\in \cO_0:\;	\abs{\omega \cdot \ell + \s_1\Omega_{\jj_1}(\lambda, \e)+ \s_2 \Omega_{\jj_2}(\lambda, \e)} \geq \gamma_0 \frac{\e}{\la \ell \ra^{\tau_0}} \ ,\;\forall (\jj,\ell,\s)\in \fA_2\setminus \fR_2\right\}
		\end{equation}
has positive measure. In fact $|\cO_0\setminus \cC^{(0)}|\lesssim \e_0^{\kappa_0}$  for some $\kappa_0>0$ independent of $\e_0$. 

\item For each $\lambda \in \cC^{(0)}$ and all $r \in [0, r_0]$, $\rho \in [\frac{\rho_0}{64},  \rho_0]$, there exists an invertible symplectic change of variables $\cL^{(0)}$, that is well defined and majorant analytic from $D(\rho/8, \zeta_0 r) \to D(\rho,r)$ (here $\zeta_0>0$ is a constant depending only on $\rho_0,\max(|\tm_k|^2)$) and such that if $\ba\in h^1(\Z_N^2 \setminus \cS_0)$, then
\begin{equation}\notag
(\cN+\cH^{(0)})\circ \cL^{(0)}(\yy, \theta, \ba) = \omega \cdot \yy + \sum_{\jj \in \Z_N^2 \setminus \cS_0} \Omega_\jj\, |a_\jj|^2.
\end{equation}

\item The mass $\cM$ and the momentum $\cP$ (defined in \eqref{mp.1}) in the new coordinates are given by 
\begin{equation}\label{mass.momentum.L}
\cM\circ \cL^{(0)}= \wtcM\,,\quad \cP\circ \cL^{(0)}= \wtcP \ , 
\end{equation}
where $\wtcM$ and $\wtcP$ are defined in \eqref{mp.4}.

\item The map $\cL^{(0)}$
  maps $h^1$ to itself and has the following form
$$
\cL^{(0)}:\quad 	\ba \mapsto L(\lambda; \theta, \e)\ba, \qquad \yy \mapsto  \yy + (\ba, Q(\lambda; \theta,\e)\ba), \qquad \theta \mapsto \theta.
	$$
The same holds for the inverse map $(\cL^{(0)})^{-1}$.

\item The linear maps $L(\lambda; \theta, \e)$ and $Q(\lambda; \theta,\e)$ are block diagonal in the $y$ Fourier modes, in the sense that $L={\rm diag}_{n\in N\N}(L_{n})$ with each $L_{n}$ acting on the sequence $\{a_{(m,n)},a_{(m,-n)}\}_{m\in \Z}$ (and similarly for $Q$). Moreover, $L_0={\rm Id}$ and $L_n$ is of the form ${\rm Id}+S_n$ where $S_n$ is a smoothing operator in the following sense: with the smoothing norm $\lceil \cdot \rfloor_{\rho,-1}$ defined in \eqref{def.smoothingnorm} below
$$
\sup_{n\neq 0} \lceil S_n\circ P_{\{|m|\geq (\tm_\tk+1)\}} \rfloor_{\rho,-1} \lesssim \varepsilon,
$$
where $P_{\{|m|\geq K\}}$ is the orthogonal projection of a sequence $(c_m)_{m \in \Z}$ onto the modes $|m| \geq K$.
\end{enumerate}
\end{theorem}

The above smoothing norm is defined as follows: Let $S(\lambda; \theta, \e)$ be an operator acting on sequences $(c_k)_{k \in \Z}$ through its matrix elements $S(\lambda; \theta, \e)_{m, k}$. Let us denote by $S(\lambda; \ell, \e)_{m, k}$ the $\theta$-Fourier coefficients of $S(\lambda; \theta, \e)_{m, k}$. 
For $\rho, \nu >0$ we define $\lceil S(\lambda; \theta, \e) \rfloor_{\rho, \nu}$ as:
\begin{equation}\label{def.smoothingnorm}
\lceil S(\lambda; \theta, \e) \rfloor_{\rho, \nu}:=\sup_{\|c\|_{\ell^1}\leq 1} \left\|\left(\sum_{\substack{k\in \Z\\ \ell\in \Z^\tk}} e^{\rho|\ell|} |S_{m, k}(\lambda; \ell, \e) | \langle k \rangle^{-\nu} c_k\right)\right\|_{\ell^1} .
\end{equation}
This definition is equivalent to the more general norm used in Definition 3.9 of \cite{Maspero-Procesi}. Roughly speaking, the boundedness of this norm means that, in terms of its action on sequences, $S$ maps $\langle k \rangle^\nu \ell^1 \to \ell^1$. As observed in Remark 3.10 of \cite{Maspero-Procesi}, thanks to the conservation of momentum this also means that $S$ maps $\ell^1 \to \langle k \rangle^{-\nu} \ell^1$.

\begin{remark}
		Note that in \cite{Maspero-Procesi}  Theorem \ref{thm:reducibility} is proved in $h^s$ norm with $s>1$, for instance in \eqref{def.smoothingnorm} the $\ell^1$ norm is substituted with the $h^s$ one. However the proof only relies on momentum conservation and on the fact that $h^s$ is an algebra w.r.t. convolution, which holds true also for $\ell^1$.
		 Hence the proof of our case is identical and we do not repeat it. 
\end{remark}

%
%
We are able to describe quite precisely the asymptotics of the frequencies $\Omega_\jj$ of Theorem \ref{thm:reducibility}. 

\begin{theorem}
	\label{thm:reducibility4}
	 For any $0<\e\le\e_{0}$ and $\lambda\in \cC^\0$,  the frequencies  $\Omega_\jj \equiv \Omega_\jj(\lambda,\e)$, $\jj = (m,n)\in \Z_N^2\setminus \cS_0$, introduced in Theorem \ref{thm:reducibility} have the following asymptotics:
\begin{equation}
\label{as.omega}
\Omega_\jj(\lambda, \e) =
\begin{cases}
\wtOmega_\jj(\lambda, \e)+\displaystyle{\frac{\varpi_m(\lambda, \e)}{\langle m \rangle}}, \qquad n=0 \\
 \wtOmega_\jj(\lambda, \e) + \displaystyle{\frac{\Theta_{m}(\lambda, \e)}{\la m \ra^2} + \frac{\Theta_{m,n}(\lambda, \e)}{\la m \ra^2 + \la n \ra^2}},  \qquad n\neq 0 
\end{cases} \ , 
\end{equation}
	where 
	\begin{equation}\notag
	\wtOmega_\jj (\lambda, \e) := 
	\begin{cases}
	m^2, & \jj=(m,0), m \notin \cS_0\\ 
	m^2 + n^2   , & \jj=(m,n) \in \fZ  \ , n \neq 0 \\
	\e \mu_i(\lambda) + n^2   \ , & \jj=(\tm_i, n)\in \sS , n\neq 0
	\end{cases}
	\end{equation}
where $\fZ$ and  $\sS$ are the sets defined in \eqref{def:SetZ}.
	
	Here the $\{\mu_i(\lambda)\}_{1 \leq i \leq \tk }$ are the roots of the polynomial 
	\begin{equation}\notag
	P(t,\lambda):= \prod_{i=1}^\tk (t + \lambda_i) - 2 \sum_{i=1}^\tk \lambda_i \, \prod_{k \neq i} (t + \lambda_k),
	\end{equation}
which is irreducible over $\mathbb{Q}(\lambda)[t]$.

Finally $\mu_i(\lambda)$, $\{\varpi_m(\lambda, \e)\}_{m \in \Z\setminus \cS_0}$, $\{\Theta_m(\lambda, \e)\}_{m \in \Z}$ and $ \{\Theta_{m,n}(\lambda, \e)\}_{(m,n) \in \Z_N^2\setminus \Tan}$ fulfill 
\begin{equation}
\label{theta.est} 
\sum_{1 \leq i \leq \tk} | \mu_i(\cdot) |^{\cO_0} +
 \sup_{\e \leq \e_0  } \frac{1}{\e^2}\Big( \sup_{m \in \Z\setminus \cS_0} |\varpi_m(\cdot, \e)|^{\cO_0} +\sup_{m \in \Z} |\Theta_m(\cdot, \e)|^{\cO_0} + \sup_{\substack{(m,n) \in \Z_N^2\\n\neq 0}} |\Theta_{m,n}(\cdot, \e)|^{\cO_0} \Big) \leq  \tM_0 \ 
\end{equation}
for some $\tM_0$ independent of $\e$. 

\end{theorem}
Theorems \ref{thm:reducibility} and \ref{thm:reducibility4} follow from Theorems 5.1 and 5.3 of \cite{Maspero-Procesi}, together with the observation that the set $\sC$ defined in Definition 2.3 of \cite{Maspero-Procesi} satisfies $\sC \cap \Z_N^2 = \emptyset$ if $N>\max_i |\tm_i|$.

We conclude this section with a series of remarks.

\begin{remark}
\label{rem:mu}\label{rmk:mus}
Notice that the $\{\mu_i(\lambda)\}_{1 \leq i \leq \tk}$ depend on the number $\tk$ of tangential sites but \emph{ not on the $\{\tm_i\}_{1 \leq i \leq \td}$}. 
\end{remark}

\begin{remark}
\label{rem:asym}
The asymptotic  expansion \eqref{as.omega} of the normal frequencies does not 
contain any constant term. The reason is that we canceled such a term when we 
subtracted the quantity $M(u)^2$ from the Hamiltonian at the very beginning (see the footnote in Section \ref{sec:FourierPhase}). Of 
course if we had not  removed $M(u)^2$, we would have had a constant correction 
to the frequencies, equal to $\norm{q(\omega t, \cdot)}^2_{L^2}$. Since 
$q(\omega t, x)$ is a solution of \eqref{NLS}, it enjoys mass conservation, and 
thus $\norm{q(\omega t, \cdot)}^2_{L^2} = \norm{q(0, \cdot)}^2_{L^2}$ is 
independent of time.
\end{remark}

\begin{remark}
	\label{leggi_sel1}
	In the new variables, the {\em selection rules} of Remark \ref{leggi_sel} become (with $\cH$ expanded as in \eqref{h.funct}):
	\begin{align*}
	&\{ \cH, \widetilde \cM\} = 0 \ \ \  \Leftrightarrow   \ \ \  \cH_{\alpha, \beta, \ell} \, (\widetilde \eta(\alpha, \beta) + \eta(\ell)) = 0 \\
	& \{ \cH, \widetilde \cP_x\} = 0  \ \ \ \Leftrightarrow  \ \ \    \cH_{\alpha, \beta, \ell} \, (\widetilde \pi_x(\alpha, \beta) + \pi(\ell)) = 0 \\
	& \{ \cH, \widetilde \cP_y\} = 0  \ \ \ \Leftrightarrow  \ \ \    \cH_{\alpha, \beta, \ell} \, (\pi_y(\alpha, \beta)) = 0 
	\end{align*}
	where $\eta(\ell)$ is defined in \eqref{def.eta},  $\pi_y(\alpha, \beta), \pi(\ell)$ in \eqref{def.pi}, while
	$$
	\widetilde \eta(\alpha, \beta):= \sum_{\jj \in \fZ }(\al_\jj-\bt_\jj) \ , 
	$$
	$$
	\widetilde{\pi}_x(\alpha, \beta):= \sum_{\jj=(m,n) \in \fZ} m(\al_\jj-\bt_\jj). 
	$$
\end{remark}

\section{Elimination of cubic terms}\label{sec:CubicBirkhoff}

If we apply the change $ \cL^{(0)}$ obtained in Theorem \ref{thm:reducibility} to Hamiltonian \eqref{H.2},  we obtain
\begin{equation}
\label{ham.bnf3}
\begin{split}
\cK (\lambda; \yy, \theta, \ba)&:= \cH\circ \cL^{(0)}(\lambda; \yy, \theta, \ba)=\omega \cdot \yy + \sum_{\jj \in \Z_N^2 \setminus \cS_0} \Omega_\jj\, |a_\jj|^2 + \cK^{\1} + \cK^{(2)} +\cK^{(\geq 3)}, \\
 \cK^{(j)}&=\cH^{{(j)}} \circ \cL^{(0)}\quad (j=1, 2), \qquad \cK^{(\geq 3)}=\cH^{(\geq 3)}\circ \cL^{(0)}.
 \end{split}
\end{equation}
As a direct consequence of Lemma \ref{lem:norm.ham} and Theorem \ref{thm:reducibility}, estimates 
 \eqref{lem:norm.ham1} hold also for $\cK^{(j)}$, $j=1,2$ and $\cK^{(\geq 3)}$.

We now perform one step of Birkhoff normal form change of variables which cancels out $\cK^\1$ completely. In order to define such a change of variables we need to impose third order Melnikov conditions, which  hold true on a subset of the set $\cC^{(0)}$ of Theorem \ref{thm:reducibility}. 
\begin{lemma}\label{lemma:cubic:MeasureEstimate} Fix $0 <\e_1<\e_0$ sufficiently small and $\tau_1 >\tau_0$ sufficiently large. There exist constants $\gamma_1>0, \tL_1>\tL_0$ (with $\tL_1$ depending only on $\td$), such that for all $0<\e\le \e_1$ and for an $\tL_1$-generic choice of the set $\Tan$  (in the sense of Definition \ref{Lgenericity}), the set
	\begin{equation}\notag
	\cC^{(1)}:=\left\{\lambda\in \cC^{(0)}:\;	\abs{\omega \cdot \ell + \s_1\Omega_{\jj_1}(\lambda, \e)+ \s_2 \Omega_{\jj_2}(\lambda, \e)+\s_3 \Omega_{\jj_3}(\lambda, \e)} \geq \gamma_1 \frac{\e }{\la \ell \ra^{\tau_1}} \ ,\;\forall (\jj,\ell,\s)\in \fA_3\right\},
		\end{equation}
where $\fA_3$ is introduced in Definition \ref{rem:adm3}, has positive measure. More precisely, $|\cC^{(0)}\setminus \cC^{(1)}  |\lesssim \e_1^{\kappa_1}$  for some constant $\kappa_1>0$ independent of  $\e_1$. 
\end{lemma}
This lemma is proven in Appendix C of  \cite{Maspero-Procesi}.

The main result of this section is the following theorem.
\begin{theorem}
\label{thm:3b}
Assume the same hypotheses and use the same notation as in Lemma \ref{lemma:cubic:MeasureEstimate}.
Consider the constants $\tL_1$, $\gamma_1$, $\tau_1$ given by Lemma 
\ref{lemma:cubic:MeasureEstimate}, the associated set $\cC^{(1)}$, and the constants $\eps_0$, $\rho_0$ and $r_0$ given in Theorem \ref{thm:reducibility}. There 
exist $0<\e_1\leq \e_0$, $0<\rho_1\leq \rho_0/64$, $0<r_1\leq r_0$ such that the following holds true for all $0<\e\le \e_1$. 
For each $\lambda \in \cC^{(1)}$ and all $0<r \leq r_1$, $0<\rho \leq \rho_1$, there exists a symplectic change of variables $\cL^{(1)}$, that is well defined and majorant analytic from $D(\rho/2, r/2) \to D(\rho,r)$ such that applied to Hamiltonian $\cK$ in \eqref{ham.bnf3} leads to
	\begin{equation}\label{def:HamAfterCubic}
	\cQ:=\cK \circ \cL^{(1)}(\lambda; \yy, \theta, \ba) =\omega \cdot \yy +  \sum_{\jj\in \Z_N^2\setminus \Tan}\Omega_\jj(\lambda, \e) |a_\jj|^2  +\cQ^{( 2)} +\cQ^{(\geq 3)}\ , 
	\end{equation}
	where 
\begin{itemize}
\item[(i)] the map $\cL^{(1)}$ is the time-1 flow of a cubic hamiltonian $\chi^\1$ such that 
$|{\chi}^\1|_{\rho/2,r/2}^{\cC^\1} \lesssim \frac{r }{\sqrt{\e}}$.
\item[(ii)]  $\cQ^{( 2)} $ is of degree 2 (in the sense of Definition \ref{def:degree}) and is given by
	\begin{equation}\label{zebra2}
	\cQ^{(2)}=\cK^{(2)} +\frac12 \{ \cK^{(1)}, \chi^{\1}\},
	\end{equation}
	and satisfies $|\cQ^{(2)}|_{\rho/2,r/2}\lesssim {r^2}$ .
\item[(iii)] $\cQ^{( \geq 3)} $ is of degree at least 3 and satisfies

\begin{equation}\label{zebra3}
|\cQ^{(\geq 3)}|_{\rho/2, r/2}^{\cC^{(1)}} \lesssim \frac{r^3}{\sqrt\varepsilon}.
\end{equation}
	\item [(iv)] $\cL^{(1)}$ satisfies $\wt \cM \circ \cL^{(1)} = \wt \cM$ and $\wt \cP \circ \cL^{(1)} = \wt \cP$.
	\item[(v)] $\cL^{(1)}$ maps $D(\rho/2, r/2) \cap h^1  \to D(\rho, r)\cap h^1$, and if we denote $( \widetilde \cY, \widetilde\theta, \widetilde {\bf a})=\cL^{(1)}( \cY, \theta,  {\bf a})$, then 
\begin{equation}\label{cubic difference}
 \left\| \widetilde{\bf a}-\ba\right\|_{\ell^1}\lesssim \|\ba\|_{\ell^1}^2.
\end{equation}
\end{itemize}
 \end{theorem}
To prove this theorem, we state the following lemma, which is proved in \cite{Maspero-Procesi}.


\begin{lemma}\label{lemma:Estimates}
 For every $ \rho,r>0$ the following holds true:
\begin{itemize}
	\item[(i)] Let $h,\, f \in \cA_{\rho,r}^\cO$. For any $0<\rho' <\rho$ and  $0<r' < r$, one has
		\[
		\left|\{f,g\}\right|_{\rho',r'}^\mathcal{O}\leq \upsilon^{-1}C
		\left|f\right|_{\rho,r}^\mathcal{O} \left|g\right|_{\rho,r}^\mathcal{O}.
		\]
		where   $\upsilon := \min \left( 1-\frac{r'}{r}, \rho-\rho'\right)$.
	 If $\upsilon^{-1} |f|_{\rho,r}^\cO<\zeta$ sufficiently small then the (time-1 flow of the) Hamiltonian vector field $X_f$ defines a close to identity canonical change of variables
	$\cT_f$ such that 
	$$
	 |h\circ\cT_f|_{\rho', r'}^\cO \leq (1+C\zeta)|h|_{\rho,r}^\cO \ , \qquad\text{for all }\, 0<\rho' <\rho \ , \ \ 0<r' < r \ . 
	$$
	\item[(ii)] Let  $f,g \in \cA_{\rho,r}^\cO$  of minimal  degree respectively $\td_f$ and   $\td_g$ (see Definition \ref{def:degree}) and define the function 
	\begin{equation}\label{taylor}
	\re{\ti}(f; h)=\sum_{l=\ti}^\infty  \frac{(\ad f)^l}{l!} h\,,\quad \ad(f)h:= \{h,f\} \ . 
	\end{equation} 
Then $\re{\ti}(f; g)$ is of minimal  degree $\td_f \ti   +\td_g$ 
	and we have the bound 
	$$
	\abs{\re{\ti}(f; h)}_{\rho',r'}^\cO \leq C(\rho) \upsilon^{-\ti} \left(|f|_{\rho,r}^\cO\right)^\ti \,  |g|_{\rho,r}^\cO \ ,  \qquad \forall 0<\rho' <\rho \ , \ \ 0<r' < r  \ .
	$$
\end{itemize}
\end{lemma}
 
 \begin{proof}[Proof of Theorem \ref{thm:3b}]
We look for $\cL^{(1)}$ as the time-one-flow of a Hamiltonian $\chi^\1$.  
With $\widehat\cN := \omega \cdot \yy +  \sum_{\jj\in \Z_N^2\setminus \Tan}\Omega_\jj(\lambda, \e) |a_\jj|^2$ and $\displaystyle{\re{j}(\chi^{(1)}; \,\cdot )=\sum_{k \geq j} \frac{{\rm ad}(\chi^\1)^{k-1}[ \{ \cdot,  \chi^\1 \}]}{k!}}$, 	we have 
\begin{align}
\label{blu.10}
\cK\circ \cL^{(1)}    = & \  \widehat\cN +  \{\widehat\cN ,  \chi^\1  \} + \cK^\1  \\
\label{blu.20}
& + \re2(\chi^\1; \, \widehat\cN ) + \{ \cK^\1, \chi^\1\}+ \re2(\chi^\1; \, \cK^\1 )   \\
\label{blu.30}
& + \cK^{\2} +\re1(\chi^\1; \,  \cK^\2 ) + \cK^{(\geq 3)} \circ \cL^{(1)}
\end{align}	
	We choose $\chi^\1$  to solve the homological equation $  \{ \widehat\cN , \chi^\1 \} + \cK^\1 = 0$. Thus we  set
	$$
\cK^\1=	\sum_{\ell,\bj,\vec{\s}\in\; \fA_3} K^{\vec{\sigma}}_{\ell,\bj}(\lambda, \e)\,   e^{\im \theta\cdot \ell}a^{\sigma_1}_{\jj_1}a^{\sigma_2}_{\jj_2}a^{\sigma_3}_{\jj_3}
\,,
\qquad 
\chi^\1=	\sum_{\ell,\bj,\vec{\s}\in\; \fA_3} \chi^{\vec{\sigma}}_{\ell,\bj}(\lambda, \e)\,  e^{\im \theta\cdot \ell}a^{\sigma_1}_{\jj_1}a^{\sigma_2}_{\jj_2}a^{\sigma_3}_{\jj_3}
	$$
	with 
	$$
	\chi^{\vec{\sigma}}_{\ell,\bj}(\lambda, \e) :=
	 \frac{\im K^{\vec{\sigma}}_{\ell,\bj}(\lambda, \e)}{\omega \cdot \ell + \s_1\Omega_{\jj_1}(\lambda, \e)+ \s_2 \Omega_{\jj_2}(\lambda, \e)+ \s_3 \Omega_{\jj_3}(\lambda, \e)} \ .
	$$
Since $\lambda \in \cC^\1$, we have
	$$
|{\chi}^\1|_{\frac{\rho}{2},r}^{\cC^\1} \lesssim \frac{r }{\sqrt{\e}},
	$$
since the terms $q_m^{\rm{fg}}$ appearing in $\cH^\1$ (and hence $\cK^\1$) are $O(\sqrt\e)$. 
	We come to the terms of line \eqref{blu.20}. First we use the homological equation  $  \{ \widehat\cN , \chi^\1 \} + \cK^\1 = 0$ to get that
	\begin{align*}
\re2\left( \chi^\1; \widehat\cN \right) & 
 = \sum_{k \geq 2} \frac{{\rm ad}(\chi^\1)^{k-1}[ \{  \widehat\cN, \chi^\1 \}]}{k!} = -\frac12\{ \cK^\1, \chi^\1\}- \sum_{k \geq 2} \frac{{\rm ad}(\chi^\1)^{k}[\cK^\1 ]}{(k+1)!}.  
\end{align*}
Therefore, we set $\cQ^\2$ as in \eqref{zebra2} and $$\cQ^{(\geq 3)}=\re2(\chi^\1; \, \cK^\1 ) +\re1(\chi^\1; \,  \cK^\2 ) + \cK^{(\geq 3)} \circ \cL^{(1)}- \sum_{k \geq 2} \frac{{\rm ad}(\chi^\1)^{k}[\cK^\1 ]}{(k+1)!}.$$
By Lemma \ref{lemma:Estimates},  $\cQ^{(\geq 3)}$ has degree at least 3 and fulfills the quantitative estimate \eqref{zebra3}.
To prove $(iv)$, we use the fact that $\{ \wt\cM, \chi^\1\} = \{ \wt \cP, \chi^\1\} = 0$ follows since $\cK^\1$ commutes with $\wt\cM$ and $\wt \cP$, hence its monomials fulfill the selection rules of Remark \ref{leggi_sel1}. 
	By the explicit formula for $\chi^\1$ above, it follows that the same selection rules hold for $\chi^\1$, and consequently $\cL^{(1)}$ preserves $\wt \cM$ and $\wt \cP$. 

\medskip 

It remains to show the mapping properties of the operator $\cL^{(1)}$. First we show that it maps $D(\rho/2, r/2) \to D(\rho, r)$. 
Let us denote by
 $( \widetilde \cY, \widetilde\theta, \widetilde {\bf a})=\cL^{(1)}( \cY, \theta, {\bf a})$, then $( \widetilde \cY, \widetilde\theta, \widetilde {\bf a})=( \widetilde \cY(s), \widetilde\theta(s), \widetilde {\bf a}(s))\big|_{s=1}$ where $( \widetilde \cY(s),\widetilde\theta(s), \widetilde {\bf a}(s))$ is the Hamiltonian flow generated by $\chi^{(1)}$ at time $0\leq s\leq 1$. Using the identity
$$
( \widetilde \cY(t),\widetilde\theta(t), \widetilde {\bf a}(t))=(  \cY, \theta, {\bf a})+\int_0^t X_{\chi^{(1)}}\left( \widetilde \cY(s), \widetilde\theta(s), \widetilde {\bf a}(s)\right) \di s
$$
where $X_{\chi^{(1)}}$ is the Hamiltonian vector field associated with $\chi^{(1)}$ above, and a standard continuity (bootstrap) argument, we conclude that $( \widetilde \cY, \widetilde\theta,\widetilde {\bf a}) \in D(\rho, r)$. Similarly, one also concludes estimate \eqref{cubic difference}. Finally, to prove that $\cL^{(1)}$ maps $D(\rho/2, r/2) \cap h^1 \to h^1$, we note that $\widehat \cN$ is equivalent to the square of the $h^1$ norm, and 
$$
\widehat \cN \circ \cL^{(1)}=   \widehat\cN +  \re1(\chi^\1; \, \widehat\cN )=\widehat\cN - \sum_{k \geq 0} \frac{{\rm ad}(\chi^\1)^{k}[\cK^\1 ]}{(k+1)!}=\widehat \cN +O(\sqrt \e r^3),
$$
and this completes the proof.

\end{proof}

\section{Analysis of the quartic part of the 
Hamiltonian}\label{sec:QuarticBirkhoff}

At this stage, we are left with the Hamiltonian $\cQ$ given in \eqref{def:HamAfterCubic}. The aim of this section is to eliminate non-resonant terms from $\cQ^{(2)}$.
First note that  $\cQ^\2$ contains monomials which have one of the two following 
forms
$$
e^{\im \theta \cdot \ell} \, a_{\jj_1}^{\sigma_1} \, a_{\jj_2}^{\sigma_2} \, 
a_{\jj_3}^{\sigma_3}\, a_{\jj_4}^{\sigma_4} \quad \mbox{ or } \quad e^{\im \theta 
\cdot \ell} \,\yy^l  a_{\jj_1}^{\sigma_1} \, a_{\jj_2}^{\sigma_2} \ \ \ \mbox{with} \ \ |l| = 1.
$$
In order to cancel out the terms quadratic in $a$ by a Birkhoff Normal form procedure, we only need the {\sl second Melnikov conditions} imposed in \eqref{2.mc}. In order to cancel out the quartic tems in $a$ we need {\sl fourth Melnikov conditions}, namely to control expressions of the form

\begin{equation}
\label{4m}
\omega(\lambda) \cdot \ell + 
\sigma_1 \Omega_{\jj_1}(\lambda, \e) + 
\sigma_2 \Omega_{\jj_2}(\lambda, \e) +
\sigma_3 \Omega_{\jj_3}(\lambda, \e) +
\sigma_4 \Omega_{\jj_4}(\lambda, \e) \,,\quad \s_i=\pm 1.
\end{equation}
We start by defining the following set $\fR_4\subset\fA_4$ (see Definition \ref{rem:adm3}),
\begin{align}
\label{def:R4}
\fR_4 := \Big\{(\bj, \ell, \sigma) \colon
&  \ell = 0 \mbox{ and } \jj_1, \jj_2, \jj_3, \jj_4\notin \sS  \mbox{ form a 
	rectangle}\\
& \ell = 0 \mbox{ and } \jj_1 , \jj_2 \notin\sS , \jj_3, \jj_4 \in \sS  \mbox{ 
	form a  horizontal rectangle (even degenerate)}\notag\\
&\ell \neq 0, \ \jj_1, \jj_2, \jj_3 \in \sS, \ \jj_4 \not\in \sS \mbox{ and } 
|m_4|< M_0, \mbox{ where } M_0 \mbox{ is a universal constant} \notag\\
&\ell =  0, \ \jj_1, \jj_2, \jj_3, \jj_4 \in \sS \mbox{ form a  horizontal 
	trapezoid}
\Big\} \notag
\end{align}
where $\sS$ is the set defined in \eqref{def:SetZ}. Here a trapezoid (or a rectangle) is said to be {\em horizontal} if two 
sides are parallel to the $x$-axis.
\begin{figure}[ht]\centering
	\vskip-10pt\begin{minipage}[c]{5cm}
		\hskip-122pt	{\centering
			\includegraphics[width=12cm]{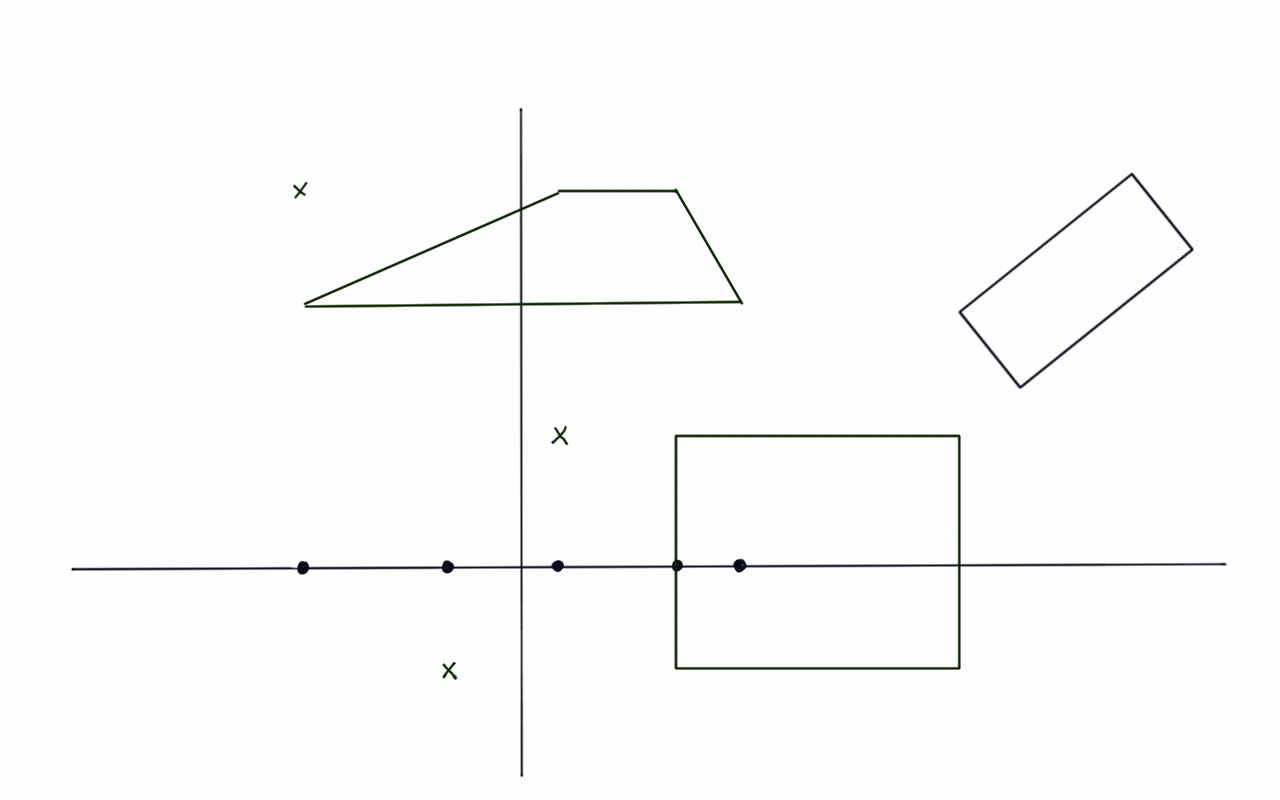}
		}
	\end{minipage}
	\caption{The black dots, are the points in $\cS_0$. The two rectangles 
		and the trapezoid correspond to cases 1,2,4 in $\fR_4$. In order to represent 
		case 3. we have highlighted three points in $\cS$. To each such triple  we may 
		associate at most one $\ell\neq 0$ and  one $\jj_4\in \fZ$, which form a 
		resonance of type 3.}
\end{figure}

\begin{proposition}
	\label{hopeful thinking}
	Fix  $0<\e_2<\e_1$ sufficiently small and $\tau_2>\tau_1$ sufficiently large. There exist positive  $\gamma_2>0, \tL_2\ge \tL_1$ (with $\tL_2$ depending only on $\td$), such that for all $0<\e\le \e_2$ and for an $\tL_2$-generic choice of the set $\Tan$  (in the sense of Definition \ref{Lgenericity}), the set
	\begin{equation}\notag
\begin{split}
	\cC^{(2)}:=\Big\{\lambda\in \cC^{(1)} :&\;	\abs{\omega \cdot \ell + \s_1\Omega_{\jj_1}(\lambda, \e)+ \s_2 \Omega_{\jj_2}(\lambda, \e)+\s_3 \Omega_{\jj_3}(\lambda, \e)+\s_4\Omega_{\jj_4}(\lambda, \e)} \geq  \frac{\gamma_2\e}{\la \ell \ra^{\tau_2}} \ ,\\
	&\;\forall (\jj,\ell,\s)\in \fA_4 \setminus\fR_4\Big\}
	\end{split}
	\end{equation}
	has positive measure and $\abs{\cC^{(1)} \setminus \cC^{(2)}} \lesssim \e_2^{\kappa_2}$ for some $\kappa_2>0$ independent of $\e_2$.
\end{proposition}
The proof of the proposition, being quite technical, is postponed to 	 Appendix \ref{app:mes.m}.

An immediate consequence, following the same strategy as for the proof of Theorem \ref{thm:3b}, is the following result. We define $\Pi_{\fR_4}$ as the projection of a function in $D(\rho, r)$ onto the sum of monomials with indexes in $\fR_4$. Abusing notation, we define  analogously $\Pi_{\fR_2}$ as the projection onto monomials $e^{\im\ell\cdot\theta}\yy^l a_{\jj_1}^{\sigma_1} a_{\jj_2}^{\sigma_2}$ with $|l|=1$ and $(\jj_1,\jj_2,\ell, \sigma_1,\sigma_2)\in\fR_2$.

\begin{theorem}\label{prop:Birkhoff4}
There exist $0 < r_2\leq r_1$, $0<\rho_2 \leq\rho_1$  such that for all $0<\e\leq \e_2$, for all $\lambda\in \cC^\2$ and for all $r \in [0,  r_2]$, $\rho \in [\frac{\rho_2}{2} , \rho_2]$  there exists a symplectic 
	change of variables $\cL^\2$  well defined and majorant analytic from  $D(\rho/2,r/2)\to 
	D(\rho,r)$ such that
	\begin{equation}\label{def:HamAfterBirkhoff4}
	\cQ \circ \cL^\2 (\yy, \theta, \ba)= \omega \cdot \yy +  
	\sum_{\jj\in \Z^2\setminus 
		\cS_0}\Omega_\jj(\lambda, \e) |a_\jj|^2  +\cQ^{(2)}_{\rm Res}+\wt \cQ^{(\ge 3)}
	\end{equation}
	where  
	\begin{equation}\label{q2res}
	\cQ^{(2)}_{\rm Res}= \Pi_{\fR_4} \cQ^{(2)} + \Pi_{\fR_2} 
	\cQ^\2
	\end{equation}
	with $\fR_4$ defined in \eqref{def:R4}, $\fR_2$ defined in \eqref{res2}  and 
	$$
	|\cQ^{(2)}_{\rm Res}|_{\rho/2,r/2}\lesssim {r^2} \ , \qquad 
	|\wt \cQ^{(\ge 3)}|_{\rho/2,r/2}\lesssim \frac{r^3}{\sqrt{\e}} . 
	$$
	Moreover
	$\cL^{(2)}$ maps $D(\rho/2, r/2) \cap h^1  \to D(\rho, r)\cap h^1$, and if we denote $( \widetilde \cY,\widetilde\theta, \widetilde {\bf a})=\cL^{(2)}( \cY, \theta,  {\bf a})$, then 
	\begin{equation}\notag
	\left\| \widetilde{\bf a}-\ba\right\|_{\ell^1}\lesssim \|\ba\|_{\ell^1}^3.
	\end{equation}
\end{theorem}
\begin{proof}
The proof is analogous to the one of Theorem \ref{thm:3b}, and we skip it.
\end{proof}

\section{Construction of the toy model}\label{sec:ToyModel}
Once we have performed (partial) Birkhoff normal form up to order 4, we can 
start applying the ideas developed in \cite{CKSTT} to Hamiltonian 
\eqref{def:HamAfterBirkhoff4}. Note that throughout this section $\eps>0$ 
is a fixed parameter. Namely, we do 
not use its smallness and we do not modify it. 

We first perform the (time dependent) change of variables to 
rotating 
coordinates
\begin{equation}\label{def:rotating}
 a_\jj=\beta_\jj \ e^{\im \Omega_\jj(\lambda, \e)t},
\end{equation}
to the Hamiltonian \eqref{def:HamAfterBirkhoff4}, which leads to the {corrected} Hamiltonian
\begin{equation}\label{def:HamRotAfterNF}
 \cQ_\rot(\yy, \theta, \beta,t)=\cQ \circ \cL^\2 \left(\yy, \theta, \{\beta_\jj 
\ e^{\im \Omega_\jj(\lambda, \e)t}\}_{\jj\in\mathbb{Z}_N^2\setminus\cS_0}\right) - \sum_{\jj \in \Z^2_N\setminus \cS_0} \Omega_\jj(\lambda, \e) |\beta_\jj|^2.
\end{equation}
We split this Hamiltonian as a suitable first order truncation 
$\cG$ plus two remainders,
\[
 \cQ_\rot(\yy, \theta, \beta, t)=\cG(\yy, \theta,  \beta) + 
\cJ_1(\yy, \theta, \beta, t)+\cR(\yy, \theta, \beta, t)
 \]
with 
\begin{equation}\label{def:HamTruncRotatingSimpl}
\begin{split}
\cG(\yy, \theta,  \beta)&=\omega \cdot \yy + \cQ^{(2)}_{\rm Res}(\yy,  
\theta, \beta)\\
\cJ_1(\yy, \theta, \beta,t)&=\cQ^{(2)}_{\rm Res}\left(\yy,\theta,  
\{\beta_\jj \ 
e^{\im \Omega_\jj(\lambda, 
\e)t}\}_{\jj\in\mathbb{Z}_N^2\setminus\cS_0}\right)-\cQ^{(2)}_{\rm Res}(\yy,  
\theta, \beta)\\
\cR(\yy, \theta, \beta,t)&=\wt \cQ^{(\ge 3)}\left(\yy, \theta, \{\beta_\jj \ 
e^{\im \Omega_\jj(\lambda, 
\e)t}\}_{\jj\in\mathbb{Z}_N^2\setminus\cS_0}\right)
\end{split}
\end{equation}
where $\cQ^{(2)}_{\rm Res}$ and  $\wt \cQ^{(\ge 3)}$ are the Hamiltonians introduced in Theorem \ref{prop:Birkhoff4}.


For the rest of this section  we focus our 
study on the truncated Hamiltonian $\cG$. Note that the 
remainder $\cJ_1$ is not smaller than $\cG$. Nevertheless it will be 
smaller when evaluated on the particular solutions we consider. The term $\cR$ is 
smaller than $\cG$ for small data since it is the remainder of the normal form 
obtained in 
Theorem \ref{prop:Birkhoff4}.
Later in Section \ref{sec:Approximation} we show that including the 
dismissed terms $\cJ_1$ and $\cR$ barely alters the dynamics of the  solutions of 
$\cG$ that we analyze.



\subsection{The finite set $\Lambda$}

We  now start constructing special dynamics for the Hamiltonian $\cG$ with the aim of treating the contributions of $\cJ_1$ and $\cR$ as remainder terms. Following \cite{CKSTT}, we do not study the full dynamics of $\cG$ 
but we restrict the dynamics to invariant subspaces. 
Indeed, we shall construct a set
$\Lambda\subset \fZ:=(\Z\times N\Z)\setminus (\cS_0\cup \sS)$ for some large $N$, in such a way that it generates an invariant subspace (for the dynamics of $\cG$) given by
\begin{equation}\label{def:ULambda}
    U_\Lambda:= \{\beta_\jj=0:\jj\not\in\Lambda\}.
\end{equation}
Thus, we consider the following definition.
\begin{definition}[Completeness]\label{completeness}
We say that a set $\Lambda\subset \fZ$ is {\em complete} if $U_\Lambda$ is 
invariant under
the dynamics of $\cG$.
\end{definition}

\begin{remark}
It can be easily seen that if $\Lambda$ is complete, $U_\Lambda$ is also 
invariant under
the dynamics of $\cG+\cJ_1$.
\end{remark}


We construct a complete set 
$\Lambda\subset \fZ$  (see Definition \ref{completeness}) and 
we study the restriction on it of the dynamics of the Hamiltonian $\cG$ in 
\eqref{def:HamTruncRotatingSimpl}. Following  
\cite{CKSTT}, we impose several conditions on 
$\Lambda$ to obtain dynamics as simple as possible.

The set $\Lambda$ is constructed in two steps. First we construct a 
preliminary set ${\Lambda_0} \subset \Z^2$ on which we  impose numerous 
geometrical conditions. Later on we scale ${\Lambda_0}$ by a factor $N$ to 
obtain  $\Lambda\subset  
(N\Z\times N\Z)\subset \fZ$.

The set ${\Lambda_0}$ is ``essentially'' the one described in \cite{CKSTT}.  The 
crucial point in that paper is to choose  carefully the modes so that 
each mode in $\Lambda_0$ only belongs to two rectangles with vertices in 
$\Lambda_0$. This allows to simplify considerably the dynamics and makes it 
easier to analyze. Certainly, this requires imposing several 
conditions on $\Lambda_0$. We add some extra conditions to adapt the set 
$\Lambda_0$ to the particular setting of the present paper.

We 
start by describing them. We split $\Lambda_0$ into $\gen$ disjoint generations $\La_0=\La_{01}\cup\ldots
\cup \La_{0\gen}$. We  call a quadruplet $(\jj_1,\jj_2, \jj_3, \jj_4) \in 
\Lambda_0^4$ a \emph{nuclear family} if $\jj_1, \jj_3 \in \Lambda_{0k}$, 
$\jj_2, \jj_4 \in \Lambda_{0,k+1}$, and the four vertices form a 
non-degenerate rectangle. Then, we require the following conditions.
\begin{itemize}
	\item Property $\mathrm{I}_{\Lambda_0}$ (Closure): If $\jj_1,\jj_2, \jj_3 \in 
{\Lambda_0}$ are three vertices of a rectangle, then the fourth vertex of that 
rectangle is also in ${\Lambda_0}$. 
	
	\item Property $\mathrm{II}_{{\Lambda_0}}$ (Existence and uniqueness of 
spouse 
and children): For each $1\leq k <\gen$ and every $\jj_1\in {\Lambda_0}_k$, 
there 
exists a unique spouse $\jj_3\in {\Lambda_0}_k$ and unique (up to trivial 
permutations) children $\jj_2,\jj_4 \in \Lambda_{0,k+1}$ such that 
$(\jj_1,\jj_2,\jj_3,\jj_4)$ is a nuclear family in ${\Lambda_0}$.
	
	\item Property $\mathrm{III}_{{\Lambda_0}}$ (Existence and uniqueness 
of parents 
and siblings): For each $1\leq k <\gen$ and every $\jj_2 \in \Lambda_{0,k+1}$ 
there 
exists a unique sibling $\jj_4\in \Lambda_{0,k+1}$ and unique (up to 
permutation) parents $\jj_1,\jj_3 \in {\Lambda_0}_k$ such that $(\jj_1,\jj_2,\jj_3,\jj_4)$ 
is a nuclear family in ${\Lambda_0}$.
	
	\item Property $\mathrm{IV}_{\Lambda_0}$ (Non-degeneracy): A sibling of 
any 
frequency $\jj$ is never equal to its spouse.
	\item Property $\mathrm{V}_{\Lambda_0}$ (Faithfulness): Apart from 
nuclear 
families, ${\Lambda_0}$ contains no other rectangles. In fact, by the closure 
property $\mathrm{I}_{\Lambda_0}$, this also means that it contains no right 
angled 
triangles other than those coming from vertices of nuclear families.
	
	
	\item Property $\mathrm{VI}_{\Lambda_0}$:  There 
are no 
two elements in ${\Lambda_0}$ such that
	$\jj_1 \pm \jj_2 = 0 $. 
	There are no three elements in ${\Lambda_0}$ such that $\jj_1-\jj_2+\jj_3=0$. 
	If four points in ${\Lambda_0}$ satisfy $\jj_1-\jj_2+\jj_3-\jj_4=0$ then either 
the relation is trivial or such points form a family.

	\item Property $\mathrm{VII}_{\Lambda_0}$:  There are no points in 
${\Lambda_0}$ with one of the coordinates equal to zero  i.e. $${\Lambda_0} \cap \big(\Z\times 
\{0\} \cup \{0\}\times \Z\big)= \emptyset.$$
	\item Property $\mathrm{VIII}_{\Lambda_0}$:  There are no two points in 
${\Lambda_0}$ 
which form a right angle with $0$.
\end{itemize}

Condition $\mathrm{I}_{\Lambda_0}$ is just a rephrasing of the completeness 
condition introduced in Definition \ref{completeness}. 
Properties  $\mathrm{II}_{\Lambda_0}$, $\mathrm{III}_{\Lambda_0}$, 
$\mathrm{IV}_{\Lambda_0}$, $\mathrm{V}_{\Lambda_0}$ correspond to being a 
family tree as stated in \cite{CKSTT}.

\begin{theorem}\label{thm:SetLambda}
Fix $\tK\gg 1$ and $s\in (0,1)$. Then, there exists
$\gen\gg 1$, $A_0\gg 1$, $\eta>0$, and a set $\Lambda_0\subset \Z^2$ with
\[
 \Lambda_0=\Lambda_{01}\cup\ldots\cup\Lambda_{0\gen},
\]
which satisfies  conditions $\mathrm{I}_{\Lambda_0}$ -- $\mathrm{VIII}_{\Lambda_0}$ and also
\begin{equation}\label{def:Growth}
\frac{\sum_{\jj\in\Lambda_{0,\gen-1}}|\jj|^{2s}}{\sum_{\jj\in\Lambda_{03}}|\jj|^{2s}}
\geq
 \dfrac 12 2^{(1-s)(\gen-4)}\ge \tK^2.
\end{equation}
Moreover, for any $A\geq A_0$, there exist  $\gen$ and  a 
function 
$f(\gen)$ 
satisfying 
\begin{equation}\label{estimate on fg}
e^{A^{\gen}}\leq f(\gen)\leq e^{2(1+\eta)A^{\gen}} 
\qquad \text{for $\gen$ large enough,}
\end{equation}
 such that each generation $\Lambda_{0k}$ has 
$2^{\gen-1}$
disjoint 
frequencies $\jj$ satisfying 
\begin{equation}\label{eq:BoundsS1:0}
C^{-1}f(\gen)\leq |\jj|\leq C3^\gen f(\gen),\ \ \jj\in\Lambda_{0k},
\end{equation}
and
\begin{equation}\label{eq:BoundsSN:0}
\frac{\sum_{\jj\in\Lambda_{0k}}|\jj|^{2s}}{\sum_{\jj\in\Lambda_{0i}}|\jj|^{2s}}\leq 
Ce^{s\gen}
\end{equation}
for any $1\leq i <k\leq \gen$ and some constant $C>0$ independent of $\gen$.
\end{theorem}

The construction of such kind of sets   was done first  in 
 \cite{CKSTT} (see also  \cite{GuardiaK12, GuardiaK12Err, Guardia14, 
GuardiaHP16}) where the authors construct sets $\Lambda$ satisfying Properties 
$\mathrm{I}_\Lambda$-$\mathrm{V}_\Lambda$ and estimate \eqref{eq:BoundsSN:0}.
The proof 
of Theorem  \ref{thm:SetLambda} follows the same lines as the ones in those 
papers. Indeed,  Properties 
$\mathrm{VI}_\Lambda$-$\mathrm{VIII}_\Lambda$ can be obtained through the same 
density argument. Finally, the estimate \eqref{eq:BoundsS1:0}, even if it is 
not stated explicitly in \cite{CKSTT}, it is an easy consequence of the proof 
in that paper (in  \cite{GuardiaK12, GuardiaK12Err,  GuardiaHP16} a slightly 
weaker estimate is used).

\begin{remark}
Note that  $s\in (0,1)$ implies that were are constructing a backward cascade orbit (energy is transferred from high to low modes). This means that the modes in each generation of $\Lambda_0$ are just switched oppositely $\Lambda_{0j} \leftrightarrow \Lambda_{0, \gen -j +1}$ compared to the ones constructed in \cite{CKSTT}. The second statement of Theorem \ref{thm:main} considers $s>1$ and therefore a forward cascade orbit (energy transferred from low to high modes). For this result, we need a set $\Lambda_0$ of the same kind as that of \cite{CKSTT}, which thus satisfies
\[
\frac{\sum_{\jj\in\Lambda_{0,\gen-1}}|\jj|^{2s}}{\sum_{\jj\in\Lambda_{03}}|\jj|^{2s}}
\geq
 \dfrac 12 2^{(s-1)(\gen-4)}\ge \tK^2
\]
instead of estimate \eqref{def:Growth}.
\end{remark}

We now scale ${\Lambda_0}$ by a factor $N$ satisfying \eqref{def:SizeN} 
and we denote by 
${\Lambda}:= N{\Lambda_0}$. Note that the listed properties $\mathrm{I}_{\Lambda_0}$ -- $\mathrm{VIII}_{\Lambda_0}$ are invariant under scaling. Thus, if they are satisfied by $\Lambda_0$, they are satisfied by $\Lambda$ too.

\begin{lemma}
	\label{lem:rangle}
There exists a set $\Lambda$ satisfying all statements of Theorem 
\ref{thm:SetLambda} (with a different $f(\gen)$ satisfying \eqref{estimate on fg}) and also the following additional properties.
\begin{enumerate}
\item If two points $\jj_1, \jj_2 \in 
{\Lambda}$ form a right angle with a point 
$(m,0) \in \Z \times \{0\}$, then  
	$$
	|m| \geq \sqrt{f(\gen)} \ .
	$$
\item $\Lambda\subset N\Z\times N\Z$ with 
\[
 N=f(\gen)^{\frac{4}{5}}.
\]
\end{enumerate}
\end{lemma}
\begin{proof}
Consider any of the sets $\Lambda$ obtained in Theorem \ref{thm:SetLambda}. By property $\mathrm{VIII}_{{\Lambda_0}}$ one has $m \neq 0$. Define 
$\jj_3= (m,0)$. 
The condition for orthogonality is  either
	$$
	(i) \ (\jj_1- \jj_2)\cdot (\jj_3- \jj_2 ) = 0 \ \text{ or } \ \ \ (ii)  \ (\jj_1- \jj_3)\cdot 
(\jj_2- \jj_3) = 0 \ .
	$$
	Taking $\jj_i=(m_i,n_i)$, $i=1,2$, condition $(i)$ implies (after some computations) that $m$ is 
given by
	$$
	m = \frac{(n_1- n_2)n_2 + (m_1-m_2)m_2}{m_1- m_2} \ .
	$$
	Then since $|m_1- m_2| \leq 2 Cf(\gen)3^\gen$ and the numerator is not zero, we 
have 
\begin{equation}\label{eq:RightAngleCondition}
 |m|\geq \frac{1}{4Cf(\gen) 3^\gen}\geq\frac{1}{(f(\gen))^{3/2}}.
\end{equation}
Now we consider condition $(ii)$. One gets that $m$ is a root of the 
quadratic equation
	$$
	m^2 - (m_1+m_2)m + (m_1m_2 + n_1n_2) = 0 \ .
	$$
	First we note that $m_1m_2 + n_1n_2 \neq 0$ by property
$\mathrm{VIII}_{\Lambda_0}$, since $m=0$ cannot be a solution.
	Now consider the discriminant $\Delta= (m_1+m_2)^2 - 4(m_1m_2 + 
n_1n_2)$. If $\Delta <0$, then no right angle is possible. If $\Delta = 0$, then 
clearly $|m|\geq 1/2$, since once again $m = 0$ is not a solution.
	Finally let $\Delta >0$. Then 
	$$
	m = \frac{(m_1+ m_2)}{2} \left(1 \pm \sqrt{1- \frac{4(m_1m_2+ 
n_1n_2)}{(m_1+ m_2)^2}} \right) \ . 
	$$
	
	Denoting by $\gamma:=\frac{4(m_1m_2+ n_1n_2)}{(m_1+ m_2)^2}$, the condition $\Delta>0$ implies that $-\infty <\gamma< 1$. Splitting in two cases: $|\gamma|\leq 1$ and $\gamma<-1$ one can easily obtain that either way $m$ satisfies 
\eqref{eq:RightAngleCondition}. Now it only remains to scale the set $\Lambda$ 
by a factor $(f(\gen))^{4}$. Then, taking as new $f(\gen)$, $\wt 
f(\gen):=(f(\gen))^5$, the obtained set $\Lambda$ satisfies all statements of 
Theorem \ref{thm:SetLambda} and also the statements of Lemma \ref{lem:rangle}.
\end{proof}

\subsection{The truncated Hamiltonian on the finite set $\Lambda$ and the 
\cite{CKSTT} toy model}

We use the properties of the set $\Lambda$ given by Theorem \ref{thm:SetLambda} 
and Lemma \ref{lem:rangle} to compute the restriction of 
the Hamiltonian $\cG$ in  \eqref{def:HamTruncRotatingSimpl} to the invariant 
subset 
$U_\Lambda$ (see \eqref{def:ULambda}).

\begin{lemma}\label{lemma:ResonantHamRectangles}
Consider the set $\Lambda\subset  N\Z\times N\Z$ obtained in Theorem 
\ref{thm:SetLambda}. Then, the set 
\[
\cM_\Lambda=\left\{(\yy, \theta, \beta):\yy=0, \ \  \beta\in U_\Lambda\right\}
\]
is invariant under the flow associated to the Hamiltonian $\cG$. Moreover, 
$\cG$ restricted to $\cM_\Lambda$ can be written as 
\begin{equation}\label{def:HamToyModelLattice}
\cG\big\vert_{\cM_\Lambda}(\theta, \beta)=\cG_0(\beta) + 
\cJ_2(\theta,\beta)
	\end{equation}
	where 
\begin{equation}\label{def:HamLambdaIteam}
\cG_0(\beta)=-\frac12 
\sum_{\jj\in\Lambda}|\beta_\jj|^4+ \frac12 
\sum_{\substack{(\jj_1,\jj_2,\jj_3,\jj_4)\in 
\Lambda^4\\ \jj_i \text{ form a rectangle}}}^* \beta_{\jj_1}\bar 
\beta_{\jj_2}\beta_{\jj_3}\bar \beta_{\jj_4}
\end{equation}
and the remainder $\cJ_2$ satisfies
\begin{equation}\label{def:BoundJ2}
|\cJ_2|_{\rho,r}\lesssim  r^2(f(\gen))^{-\frac{4}{5}} .
\end{equation}
\end{lemma}
\begin{proof}
First we note that, since $\yy=0$ on $\cM_\Lambda$,
\[
\cG\big\vert_{\cM_\Lambda}=\cQ^{(2)}_\mathrm{Res}\big\vert_{\cM_\Lambda}= 
\Pi_{\fR_4} \cQ^{(2)} \big\vert_{\cM_\Lambda}
\]
where $\cQ^{(2)}_\mathrm{Res}$ is the Hamiltonian defined in Theorem 
\ref{prop:Birkhoff4}. We start by analyzing the Hamiltonian $\cQ^\2$ introduced 
in Theorem \ref{thm:3b}, which is defined as 
$$\cQ^\2= \cK^\2  + \frac{1}{2}\left\{\cK^\1,\chi^\1\right\}.$$
We analyze each term.  Here it plays a crucial role 
that $\Lambda\subset N\Z\times N\Z$ with $N=f(\gen)^{4/5}$.

In order to estimate $\cK^\2$, defined in 
\eqref{ham.bnf3}, we recall that $\Lambda$ does not have any mode in the 
$x$-axis and therefore the original quartic Hamiltonian has not been modified 
by the Birkhoff map \eqref{def:BirkhoffMap} (this is evident from the formula for $\cH^{(2)}$ in \eqref{giraffe2}). Thus,  it is enough to analyze how 
the quartic Hamiltonian has been modified by the linear change $\cL^{(0)}$ 
analyzed in Theorems \ref{thm:reducibility} and \ref{thm:reducibility4}. Using 
the   smoothing property of the change of coordinates $\cL^{(0)}$ given in  
Statement 5 of Theorem \ref{thm:reducibility}, one obtains
$$ \Pi_{\fR_4} \cK^{(2)} \big\vert_{\cM_\Lambda}= -\frac12 
\sum_{\jj\in\Lambda}|a_\jj|^4+ \frac12 \sum_{{\rm Rectangles}\subset \Lambda}  
a_{\jj_1}\bar a_{\jj_2}a_{\jj_3}\bar a_{\jj_4} + O\left(\frac{ r^2 
}{N}\right).$$
 
Now we deal with the term $\{\cK^\1,\chi^\1\}$. 
Since we only need to analyze 
$\Pi_{\fR_4}\{\cK^\1,\chi^\1\}\big\vert_{\cM_\Lambda}$, 
 we only need to consider monomials in $\cK^\1$ and in $\chi^\1$ which  have 
at least two 
indexes in $\Lambda$.   
We represent this by setting
\begin{equation}\notag
\chi^\1= \chi^\1_{\#\Lambda\le 1}+ \chi^\1_{\#\Lambda \ge 2} \,, 
\end{equation}
where $\#\Lambda \ge 2$ means that  we restrict to those monomials  which  have 
at least two indexes in $\Lambda$.    We then have
$$
\{\cK^\1,\chi^\1\}\big\vert_{\cM_\Lambda}= \{\cK^\1,\chi^\1_{\#\Lambda \ge 
2}\}\big\vert_{\cM_\Lambda}.
$$
We estimate the size of $\chi^\1_{\#\Lambda \ge 2}$. As explained in the 
proof of Theorem \ref{thm:3b}, $\chi^\1_{\#\Lambda \ge 2}$ has coefficients 
\begin{equation}\label{def:Chi1}
\chi^\1_{\ell,\bj,\vec{\s}}=\frac{\im \cK^\1_{\ell,\bj,\vec{\s}}}{\omega \cdot \ell + 
\s_1\Omega_{\jj_1}(\lambda, \e)+ \s_2 \Omega_{\jj_2}(\lambda, \e)+ \s_3 
\Omega_{\jj_3}(\lambda, \e)}
\end{equation}
with $\jj_2,\jj_3\in \Lambda$.

%
We first estimate the tails (in $\ell$) of  $\chi^\1$ and then we analyze the 
finite number of cases left. For the tails, it is enough to use Theorem 
\ref{thm:3b} to deduce  the following estimate for any $\rho\leq \rho_1/2$, 
where $\rho_1$ is the constant introduced in that theorem,
$$
\left| \sum_{|\ell|>\sqrt[4]N}
\chi^\1_{\ell,\bj,\vec{\s}} \, e^{\im \theta \cdot \ell} a_{\jj_1}^{\sigma_1}a_{\jj_2}^{\sigma_2}a_{\jj_3}^{\sigma_3}
\right|_{\rho,r}^{\cC^\1}\lesssim 
e^{-(\rho_1-\rho)\sqrt[4]{N}}
\left| \chi^\1 \right|_{
\rho_1 , r }^{\cC^\1} \le 
r e^{-(\rho_1-\rho)\sqrt[4]{N}}.
$$

We restrict our attention to monomials with $|\ell|\le \sqrt[4]{N}$. We take 
$\jj_2,\jj_3\in \Lambda$ and we consider different cases depending on $\jj_1$ 
and the properties of the monomial. In each case we  show that the denominator of \eqref{def:Chi1} is larger than $N$.

\paragraph{Case 1.} Suppose that $\jj_1\notin \sS$.   The 
selection rules  are (according to Remark \ref{leggi_sel1})
$$
\eta(\ell)+ \s_1+\s_2+\s_3=0\,,\quad \vec\tm\cdot \ell+ \s_1 
m_1+\s_2m_2+\s_3m_3=0\,,\quad \s_1 n_1+\s_2n_2+\s_3n_3=0
$$
and the leading term in the denominator of \eqref{def:Chi1} is 
\begin{equation}\label{def:SmallDivisorChi}
\vec{\tm}^2\cdot \ell+ \s_1 |\jj_1|^2+\s_2|\jj_2|^2+\s_3|\jj_3|^2
\end{equation}
where $\vec{\tm}^2=(\tm_1^2,\dots,\tm_\tk^2)$. 
We consider the following subcases:
\begin{itemize}
 \item[{\bf A1}]   $\s_3=\s_1=+1$, $\s_2=-1$. In this case  $\jj_1 - \jj_2 + \jj_3  -\tt v = 0$, where $\mathtt v:=( -\vec\tm\cdot \ell,0)$. We rewrite 
\eqref{def:SmallDivisorChi} as 
$$
\vec{\tm}^2\cdot \ell + (\vec\tm\cdot \ell)^2 - (\vec\tm\cdot \ell)^2 + 
|\jj_1|^2-|\jj_2|^2+|\jj_3|^2= \vec{\tm}^2\cdot \ell + (\vec\tm\cdot \ell)^2 - 
2\Big( \mathtt v -\jj_3,\jj_3-\jj_2\Big) .
$$
Assume first $\jj_2\neq \jj_3$. Since the set $\Lambda$ satisfies Lemma 
\ref{lem:rangle} 1.
and $\abs{\vec \tm \cdot \ell} \lesssim \sqrt[4]{N} \lesssim f(\gen)^{1/5}$,  we can ensure that $\jj_2$ and $\jj_3$ do not form a right angle with $\tt v$, thus 
\[\Big( \mathtt v -\jj_2,\jj_3-\jj_2\Big)\in \Z \setminus \{0\}.\]
Actually by the second statement of Lemma \ref{lem:rangle}, 
$\jj_3-\jj_2\in N\Z^2$ and hence, using also  $|\ell|\le \sqrt[4]{N}$, 
$$
\Big| \vec{\tm}^2\cdot \ell + (\vec\tm\cdot \ell)^2 - 2\Big( \mathtt v 
-\jj_3,\jj_3-\jj_2\Big)\Big|\ge 2N -  N/8 >N.
 $$

Now it remains the case $\jj_2=\jj_3$. Such monomials cannot exist in 
$\cH^{(1)}$ in \eqref{def of H1} since the monomials with two equal modes 
have been removed in \eqref{Ha0} (it does not support degenerate rectangles). 
Naturally a degenerate rectangle may appear after we apply the 
change $\cL^{(0)}$ introduced in Theorem \ref{thm:reducibility}. Nevertheless, 
the map $\cL^{(0)}$ is  identity plus smoothing (see statement 5 of that 
theorem), which leads to the needed  $N^{-1}$ factor.

 \item[{\bf B1}]  $\s_3=\s_2=+1$, $\s_1=-1$. Now the selection rule reads $-\jj_1 + \jj_2 + \jj_3  -\tt v = 0$, with again $\mathtt v=( -\vec\tm\cdot \ell,0)$. We rewrite 
\eqref{def:SmallDivisorChi} as 
$$
\vec{\tm}^2\cdot \ell + (\vec\tm\cdot \ell)^2 - (\vec\tm\cdot \ell)^2 - 
|\jj_1|^2+|\jj_2|^2+|\jj_3|^2= \vec{\tm}^2\cdot \ell + (\vec\tm\cdot \ell)^2 - 
2\Big( \mathtt v -\jj_3,\mathtt v-\jj_2\Big) . 
$$
By the first statement of  Lemma 
\ref{lem:rangle},   $\Big( \mathtt v -\jj_2,\mathtt v-\jj_3\Big)\neq 0$.
By Property $\mathrm{VIII}_\Lambda$ and 
the 
second statement of Lemma \ref{lem:rangle}, one has $|(\jj_2, 
\jj_3)|\geq N^2$ and estimate \eqref{eq:BoundsS1:0} implies $|\jj_2|, |\jj_3| 
\leq 
N^{3/2}$.  Then
 \begin{equation}\notag
 \abs{\Big( \mathtt v -\jj_2,\mathtt v-\jj_3\Big)}  \geq  |(\jj_2, \jj_3)| - 
|(\mathtt v,\jj_2+ \jj_3)| - |\mathtt v|^2 \geq N^2/4 \  
 \end{equation}
and one concludes as in {\bf A1}.

 
\item[{\bf C1}] \; $\s_1=\s_3=\s_2=+1$. The denominator 
\eqref{def:SmallDivisorChi} satisfies
 $$
| \vec{\tm}^2\cdot \ell + |\jj_1|^2+|\jj_2|^2+|\jj_3|^2|\ge 2N 
-|\vec{\tm}^2\cdot \ell| \ge  2N-N/8\ge N.
 $$
\end{itemize}
This completes the proof of Case 1.

\paragraph{Case 2.} Suppose that $\jj_1\in \sS$.   The selection rules are
$$
\eta(\ell)+\s_2+\s_3=0\,,\quad \vec\tm\cdot \ell+\s_2m_2+\s_3m_3=0\,,\quad \s_1 
n_1+\s_2n_2+\s_3n_3=0
$$
and the leading term in the denominator is 
\begin{equation}
\label{miao}
\vec{\tm}^2\cdot \ell+ \s_1 n_1^2+\s_2|\jj_2|^2+\s_3|\jj_3|^2,
\end{equation}
where 
$\vec{\tm}^2=(\tm_1^2,\dots,\tm_\tk^2)$. We can 
reduce Case 2 to  Case 1. 

\begin{itemize}
\item[{\bf B2}] $\s_2=\s_3=+1$, $\s_1=-1$.  Assume w.l.o.g. that 
$\jj_1=(\tm_1,n_1)$. Define $\tilde{\ell}=\ell+\be_1$, and obtain from the selection rules  and \eqref{miao} that
$$
 \vec\tm\cdot \tilde\ell - \tm_1 +m_2+m_3= \vec\tm\cdot \ell+m_2+m_3 =0 
\,.
$$
Then the leading term in the denominator becomes
$$
\vec\tm^2\cdot \tilde\ell  - (\tm_1^2+n_1^2 )+|\jj_2|^2+|\jj_3|^2 
$$
and  one proceeds as in case {\bf B1} with $\tilde \ell$ in place of $\ell$. 
\end{itemize}
The cases {\bf A2} and {\bf 
C2} are 
completely equivalent. 

\smallskip


In conclusion we have proved that 
\begin{equation}\label{def:ChiEstimate}
\left|\chi^\1_{\#\Lambda \ge 2}\big\vert_{\cM_\Lambda}\right|_{\rho,r}^{\cC^\1}\le  
r
N^{-1}.
\end{equation}
Item $(i)$ of  Lemma \ref{lemma:Estimates}, jointly with estimate \eqref{def:ChiEstimate}, implies that, for 
$\rho'\in (0,  \rho/2]$ and  $r' \in (0, r/2]$
$$
\left|\left\{\cK^\1,\chi^\1_{\#\Lambda \ge 
2}\right\}\big\vert_{\cM_\Lambda}\right|_{\rho',r'}^{\cC^\1}\lesssim  r^2 
N^{-1}.
$$
This completes the proof of Lemma \ref{lemma:ResonantHamRectangles}.
%
%
\end{proof}
The Hamiltonian $\cG_0$ in \eqref{def:HamLambdaIteam} is the Hamiltonian that  
the 
I-team derived to construct their toy model. A posteriori we will 
check that the remainder $\cJ_2$ plays a small role in our analysis.

The properties of $\Lambda$ imply that  
the equation 
associated to $\cG_0$
reads
\begin{equation}\label{eq:InftyODE:FirstFiniteReduction}
 \im \dot \beta_\jj=-\beta_\jj|\beta_\jj|^2+2
\beta_{\jj_{\mathrm{child}_1}}\beta_{\jj_{\mathrm{child}_2}}\ol{\beta_{
\jj_\mathrm{
spouse}}}+2
\beta_{\jj_{\mathrm{parent}_1}}\beta_{\jj_{\mathrm{parent}_2}}\ol{\beta_{
\jj_\mathrm {
sibling}}}
\end{equation}
for each $\jj\in \Lambda$. In the first and last generations, the parents and
children are set to zero respectively. 
Moreover, the particular form of this equation 
implies the following corollary.

\begin{corollary}[\cite{CKSTT}]\label{coro:Invariant}
Consider the subspace
\[
 \widetilde U_\Lambda=\left\{\beta\in U_\Lambda: 
\beta_{\jj_1}=\beta_{\jj_2}\,\,\text{for 
all
}\jj_1,\jj_2\in\Lambda_k\,\,\text{for some }k\right\},
\]
where all the members of a generation take the same value. Then, $\wt U_\Lambda$ is 
invariant under the flow associated to the Hamiltonian $\cG_0$. Therefore, 
equation \eqref{eq:InftyODE:FirstFiniteReduction} restricted to $ \widetilde 
U_\Lambda$ becomes
\begin{equation}\label{def:model}
 \im \dot  b_k=-b_k^2\overline b_k+2  \ol
b_k\left(b_{k-1}^2+b_{k+1}^2\right),\,\,k=1,\ldots \gen , 
\end{equation}
where
\begin{equation}\label{def:ChangeToToyModel}
 b_k=\beta_\jj\,\,\,\text{ for any }\jj\in\Lambda_k.
\end{equation}
\end{corollary}

The dimension of  $ \widetilde U_\Lambda$ is $2\gen$, where $\gen$ is  the 
number of 
generations. In the papers \cite{CKSTT} and \cite{GuardiaK12}, the authors construct certain orbits of the  \emph{toy model} \eqref{def:model} which shift its mass from being localized at $b_3$ to being localized at $b_{\gen-1}$. These orbits will lead to orbits of the original equation \eqref{NLS} undergoing growth of Sobolev norms. 
\begin{theorem}[\cite{GuardiaK12}]\label{thm:ToyModelOrbit}
Fix a large $\gamma\gg 1$.  Then for any large enough $\gen$ and
$\mu=e^{-\gamma \gen}$, there exists an orbit of system \eqref{def:model},
$\kk>0$ (independent of $\gamma$ and $\gen$) and $T_0>0$ such that
\[
\begin{split}
 |b_3(0)|&>1-\mu^\kk\\
|b_i(0)|&< \mu^\kk\qquad\text{ for }i\neq 3
\end{split}
\qquad \text{ and }\qquad
\begin{split}
 |b_{\gen-1}(T_0)|&>1-\mu^\kk\\
|b_i(T_0)|&<\mu^\kk \qquad\text{ for }i\neq \gen-1.
\end{split}
\]
Moreover, there exists a constant $K>0$
independent of  $\gen$ such that $T_0$ satisfies
\begin{equation}\notag
 0<T_0< C \gen \ln \left(\frac 1 \mu \right)=C\,\gamma\,\gen^2.
\end{equation}
\end{theorem}
This theorem is proven in \cite{CKSTT} without time estimates. The time estimates were obtained in  \cite{GuardiaK12}.

\section{The approximation argument}\label{sec:Approximation}
In Sections \ref{sec:reducibility}, \ref{sec:CubicBirkhoff} and 
\ref{sec:QuarticBirkhoff} we have applied several transformations and in 
Sections \ref{sec:QuarticBirkhoff} and \ref{sec:ToyModel} we have  removed 
certain small remainders. This has allowed us to derive a simple equation, 
called toy model in \cite{CKSTT}; then, in Section \ref{sec:ToyModel}, we have analyzed some special orbits of this system.
The last step of the proof of Theorem \ref{thm:main} is to show 
that when  incorporating back the removed remainders ($\cJ_1$ and $\cR$ 
in \eqref{def:HamTruncRotatingSimpl} and $\cJ_2$ in 
\eqref{def:HamToyModelLattice}) and undoing the changes 
of coordinates performed in Theorems \ref{thm:reducibility} and \ref{thm:3b}, 
in 
Proposition \ref{prop:Birkhoff4} and in \eqref{def:rotating}, the toy model 
orbit obtained in Theorem \ref{thm:ToyModelOrbit} 
leads to a solution of the original equation \eqref{NLS} undergoing growth of 
Sobolev norms.

Now we analyze each remainder and each change of 
coordinates. From the orbit obtained in Theorem \ref{thm:ToyModelOrbit} and using 
  \eqref{def:ChangeToToyModel} one can obtain an orbit of Hamiltonian 
\eqref{def:HamLambdaIteam}. 
Moreover, both the equation of Hamiltonian
\eqref{def:HamLambdaIteam} and 
\eqref{def:model} are invariant under the scaling
\begin{equation}\label{def:Rescaling}
b^\nu(t)=\nu^{-1}b\left(\nu^{-2}t\right).
\end{equation}
By Theorem \ref{thm:ToyModelOrbit}, the time spent by the solution 
$b^\nu(t)$ is
\begin{equation}\label{def:Time:Rescaled}
 T=\nu^2 T_0\le \nu^2C\gamma \gen^2,
\end{equation}
where $T_0$ is the time obtained in Theorem \ref{thm:ToyModelOrbit}.

Now we prove that one can construct a solution of Hamiltonian 
\eqref{def:HamRotAfterNF} ``close''  to  the 
orbit $\beta^\nu$ of Hamiltonian \eqref{def:HamLambdaIteam} 
defined as
\begin{equation}\label{def:RescaledApproxOrbit}
\begin{split}
\beta^\nu_\jj(t)&=\nu^{-1}b_k\left(\nu^{-2}t\right)
\,\,\,\text{ for each }\ \jj\in\Lambda_k\\
\beta^\nu_\jj(t)&=0\,\,\,\text{ for each }\ \jj\not\in\Lambda,
\end{split}
\end{equation}
where $b(t)$ is the orbit given by Theorem \ref{thm:ToyModelOrbit}. Note that 
this implies incorporating the remainders in 
\eqref{def:HamTruncRotatingSimpl} and 
\eqref{def:HamToyModelLattice}.

We take a 
large $\nu$ so that \eqref{def:RescaledApproxOrbit} is small. In the original 
coordinates this will correspond to solutions close to the finite gap 
solution. 
Taking $\cJ=\cJ_1+\cJ_2$ (see \eqref{def:HamTruncRotatingSimpl} and 
\eqref{def:HamToyModelLattice}), the equations for $\beta$ and $\yy$  associated 
to Hamiltonian \eqref{def:HamRotAfterNF} can be written as
\begin{equation}\label{eq:betay}
\begin{split}
 \im\dot\beta&=\partial_{\ol \beta} \cG_0(\beta)+\partial_{\ol \beta} 
\cJ(\yy,\theta, \beta)+\partial_{\ol \beta} \cR(\yy,\theta, \beta)\\
 \dot \yy&=-\partial_\theta \cJ(\yy,\theta, \beta)-\partial_{\theta} 
\cR(\yy,\theta, \beta).
\end{split}
\end{equation}

Now we obtain estimates of the closeness of the orbit of the toy model obtained 
in Theorem \ref{thm:ToyModelOrbit} and orbits of Hamiltonian  
 \eqref{def:HamRotAfterNF}.

\begin{theorem}\label{thm:Approximation}
Consider a solution $(\yy, \theta, \beta)=(0,\theta_0, \beta^\nu(t))$ of 
Hamiltonian \eqref{def:HamLambdaIteam} for any $\theta_0\in \T^\tk$, where  
$\beta^\nu(t)=\{\beta^\nu_\jj(t)\}_{\jj\in \Z_N^2\setminus \cS_0}$ is  the 
solution  given
by \eqref{def:RescaledApproxOrbit}. Fix 
$\sigma$ small
 independent of $\gen$ and $\gamma$. Assume
\begin{equation}\label{def:LambdaOfN}
\frac12 \left(f(\gen)\right)^{1-\sigma}\leq \nu\leq \left(f(\gen)\right)^{1-\sigma}.
\end{equation}
Then any solution   $(\yy(t), \theta(t),\widetilde 
\beta(t))$ of \eqref{def:HamRotAfterNF}  with  initial condition 
$\widetilde\beta(0)=\widetilde\beta^0\in\ell^1$, $\yy(0)=\yy^0\in\R^\tk$  
with 
$\|\widetilde\beta^0-\beta^\nu(0)\|_{\ell^1}\leq 
\nu^{-1-4\sigma}$ and 
$|\yy^0|\leq \nu^{-2-4\sigma}$ and any $\theta(0)=\theta_1\in \T^\tk$, satisfies
\begin{equation}\notag
\left\|\widetilde \beta_\jj(t)-\beta^\nu_\jj(t)\right\|_{\ell^1}
 \leq \nu^{-1-\sigma}, \qquad \left|\yy(t)\right|\leq 
\nu^{-2-\sigma},
\end{equation}
for $0<t<T$, where $T$ is the time defined in \eqref{def:Time:Rescaled}.
\end{theorem}

The proof of this theorem is deferred to Section \ref{sec:Approx}.  Note that 
the change to rotating coordinates in 
\eqref{def:rotating} does not alter the $\ell^1$ norm and therefore a similar 
result as this theorem can be stated for orbits of Hamiltonian 
\eqref{def:HamAfterBirkhoff4} (modulus adding the rotating phase).
%
%

\begin{proof}[Proof of Theorem \ref{thm:main}]
We use Theorem \ref{thm:Approximation} to obtain a solution of Hamiltonian \eqref{H.2} undergoing
growth of Sobolev norms. We consider the solution $(\yy^*(t), 
\theta^*(t),\ba^*(t))$ of this Hamiltonian with
initial condition
\begin{equation}\label{def:IniCond}
 \begin{split}
  \yy^*&=0\\
  \theta^*&=\theta_0\\
 a^*_{\jj}&=\nu^{-1}b_k(0)\qquad\text{for each }\jj\in\Lambda_k\\
  a^*_{\jj}&=0\qquad\qquad\,\quad\text{for each }\jj\not\in\Lambda_k
 \end{split}
\end{equation}
for an arbitrary choice of $\theta_0\in \mathbb T^\tk$.

We need to prove that Theorem \ref{thm:Approximation} applies to this solution.
To this end, we perform the changes of coordinates given in Theorems
\ref{thm:reducibility}, \ref{thm:3b} and   
\ref{prop:Birkhoff4}, keeping track of the
$\ell^1$ norm.

For $\cL^{(j)}$, $j=1,2$, Theorems \ref{thm:3b} and 
\ref{prop:Birkhoff4} imply the following. Consider $(\yy,\theta,\ba)\in D(\rho,r)$
and define $\pi_\ba(\yy,\theta,\ba):=\ba$. Then, we have
\begin{equation}\label{def:EstimatesChange12}
 \left\|\pi_\ba \cL^{(j)}(\yy,\theta,\ba)-\ba\right\|_{\ell^1}\lesssim \|\ba\|_{\ell^1}^2.
\end{equation}
This estimate is not true for the change of coordinates $\cL^{(0)}$ given in Theorem
\ref{thm:reducibility}. Nevertheless,  this change is smoothing (see Statement 
5 of Theorem \ref{thm:reducibility}). This implies that if all $\jj\in
\mathrm{supp}\{\ba\}$ satisfy $|\jj|\geq J$ then 
\begin{equation}\label{def:EstimatesChange0}
 \left\|\pi_{\bf a} \cL^{(0)}\left(\yy, \theta,\ba\right)-\ba\right\|_{\ell^1}\lesssim J^{-1}
\|\ba\|_{\ell^1}.
\end{equation}
Thanks to Theorem \ref{thm:SetLambda} (more precisely \eqref{eq:BoundsS1:0}), we can apply this estimate to  \eqref{def:IniCond} with $J=C f(\gen)$. Using the fact that 
$
 \left\|\ba^*\right\|_{\ell^1}\lesssim \nu^{-1}\gen 2^\gen
$
and the condition on
$\nu$ in \eqref{def:LambdaOfN}, one can check
\[
 \left\|\pi_{\bf a} \cL^{(0)}\left(0, \theta^*,\ba^*\right)-\ba^*\right\|_{\ell^1}\lesssim \nu^{-1}\gen 2^\gen f(\gen)^{-1}\leq \nu^{-3/2}.
\]
Therefore, we can conclude
\begin{equation}\notag
\left\|\pi_\ba \left(
\cL^{(2)}\circ\cL^{(1)}\circ \cL^{(0)}\left(0, \theta^*,\ba^*\right)\right)-\ba^*\right\|_{\ell^1}\lesssim 
\nu^{-3/2}.
\end{equation}
We define $( \widetilde \yy^*,\widetilde \theta^*, \widetilde \ba^*)$ the image
of the point \eqref{def:IniCond} under the composition of these three changes.
We apply Theorem \ref{thm:Approximation} to the solution of
\eqref{def:HamRotAfterNF} with this initial condition. Note that Theorem 
\ref{thm:Approximation} is stated in rotating coordinates (see 
\eqref{def:rotating}). Nevertheless, since this change is the identity on the 
initial conditions, one does not need to make any further modification. 
Moreover, the 
change \eqref{def:rotating} leaves invariant both the $\ell^1$ and Sobolev 
norms.
We show that such solution  $( \widetilde \yy^*(t),\widetilde 
\theta^*(t), \widetilde \ba^*(t))$
expressed in the original coordinates satisfies the desired growth of
Sobolev norms.

Define
\begin{equation}\notag
 S_i=\sum_{\jj\in\Lambda_i}|\jj|^{2s}  \text{ for }i=1,\dots, \gen.
\end{equation}
To estimate the initial Sobolev norm of the solution $( \yy^*(t), 
\theta^*(t), \ba^*(t))$, we first prove that
\begin{equation}\notag
 \left \|\ba^*(0)\right\|^2_{h^s}\lesssim \nu^{-2}S_3.
\end{equation}
The initial condition of the considered orbit  given in
\eqref{def:IniCond} has support  $\Lambda$ (recall that $\yy=0$). Therefore,
\[
\left\|\ba^*(0)\right\|^2_{h^s}=\sum_{i=1}^{\gen}\sum_{\jj\in\Lambda_i}|\jj|^{2s}
\nu^{-2}\left|b_i(0)\right|^2.
\]
Then, taking into account Theorem \ref{thm:ToyModelOrbit},
\[
\begin{split}
\sum_{i=1}^{\gen}\sum_{\jj\in\Lambda_i}|\jj|^{2s}
\nu^{-2}\left|b_i(0)\right|^2&\leq \nu^{-2}
S_3+\nu^{-2}\mu^{2\kk}\sum_{i\neq3}S_i\\
&\leq \nu^{-2}S_3\left(1+\mu^{2\kk}\sum_{i\neq3}\frac{S_i}{S_3}\right).
\end{split}
\]
From Theorem \ref{thm:SetLambda} we know that for $i\neq 3$,
\[
\frac{S_i}{S_3}\lesssim e^{s\gen}.
\]
Therefore, to bound these terms we use the definition of  $\mu$ from Theorem
 \ref{thm:ToyModelOrbit}. Taking  $\gamma>\frac{1}{2\kappa} $ and taking $\gen$ large enough, we
have that 
 \[
 \|\ba^*(0)\|_{h^s}^2\leq 2\nu^{-2}S_3.
 \]
To control the initial Sobolev norm, we need that $2\nu^{-2}S_3\leq\de^2$. 
To
this end, we need to use the estimates for $\nu$ given in Theorem 
\ref{thm:Approximation}, and the estimates for $|\jj| \in \Lambda$  and  for $f(\gen)$ given in Theorem 
\ref{thm:SetLambda}. Then, if we choose $\nu=(f(\gen))^{1-\sigma}$, we have
\[
 \|\ba^*(0)\|_{h^s}^2\lesssim (f(\gen))^{-2(1-\sigma-s)}  3^{2s\gen}2^\gen\leq  
e^{-2(1-\sigma-s)A^{\gen}} 3^{2s\gen}2^\gen.
 \]
Note that Theorem \ref{thm:Approximation} is valid for any fixed small 
$\sigma>0$. 
Thus, {\bf provided $s <1$}, we can choose  $0<\sigma<1-s$ and  
take $\gen$ large enough, so that we obtain an arbitrarily small initial 
Sobolev norm.
\begin{remark}
\label{rem:l2}
In case we ask only  the  $\ell^2$ norm of $\ba^*(0)$ to be small we can drop 
the condition $s <1$. Indeed 
$ \|\ba^*(0)\|_{\ell^2}\lesssim \nu^{-1} 2^{\gen }\gen$ which can be made 
arbitrary small by simply taking $\gen$ large enough (and $\nu$ as in \eqref{def:LambdaOfN}).
\end{remark}

Now we estimate the final Sobolev norm. First we bound $\|\ba^*(T)\|_{h^s}$ 
in terms of
$S_{\gen-1}$. Indeed,
\begin{equation}\label{what the hell}
\left \|\ba^*(T)\right\|^2_{h^s}\geq \sum_{\jj\in\Lambda_{\gen-1}}|\jj|^{2s} 
\left|a_\jj^*(T)\right|^2\geq 
S_{\gen-1}\inf_{\jj\in\Lambda_{\gen-1}}\left|a_\jj^*(T)\right|^2.
\end{equation}
Thus, it is enough to obtain a lower bound for $\left|a^*_\jj(T)\right|$ for 
$\jj\in\Lambda_{\gen-1}$. To obtain this estimate we need to express $\ba^*$ in 
normal 
form coordinates and use Theorem \ref{thm:Approximation}.  We split 
$|a^*_\jj(T)|$ 
as follows. Define $( \widetilde \yy^*(t),\widetilde 
\theta^*(t), \widetilde \ba^*(t))$ the image
of the orbit with initial condition \eqref{def:IniCond} under the changes 
of variables in Theorems
\ref{thm:reducibility} and \ref{thm:3b}, Proposition  
\ref{prop:Birkhoff4} and in \eqref{def:rotating}. Then,  
\[
\left|a^*_\jj(T)\right|\geq 
\left|\beta_\jj^\nu(T)\right|-
\left|\widetilde a^*_\jj(T)-\beta_\jj^\nu(T)e^{\im \Omega_\jj(\lambda,\eps)T}\right|- 
\left|\widetilde 
a^*_\jj(T)-a^*_\jj(T)\right|.
\]
The first term, by Theorem \ref{thm:ToyModelOrbit}, satisfies 
$|\beta_\jj^\nu(T)|\geq 
\nu^{-1}/2$. For the second one, using Theorem \ref{thm:Approximation}, we have
\[
\left|\widetilde 
a^*_\jj(T)-\beta_\jj^\nu(T)e^{\im \Omega_\jj(\lambda,\eps)T}\right| 
\leq \nu^{-1-\sigma}. 
\]
Finally, taking into account the estimates  
\eqref{def:EstimatesChange12} and \eqref{def:EstimatesChange0}, the third one 
can be bounded as
\[
 \left|\widetilde a^*_\jj(T)-a^*_\jj(T)\right|\leq \left\|\widetilde 
\ba^*(T)-\ba^*(T)\right\|_{\ell^1}\lesssim 
\|\ba^*(T)\|_{\ell^1}^2+\frac{\|\ba^*(T)\|_{\ell^1}}{|\jj|}.
\]
Now, by Theorem \ref{thm:Approximation} and Theorem \ref{thm:SetLambda} (more precisely the fact that $|\jj|\gtrsim 
f(\gen)$ for $\jj\in\Lambda$), 
\[
 \left|\widetilde a^*_\jj(T)-a^*_\jj(T)\right|\leq \nu^{-1-\sigma}.
\]
Thus, by \eqref{what the hell}, we can conclude that 
\[
\left \|\ba^*(T)\right\|^2_{h^s}\geq \frac{\nu^{-2}}{2}S_{\gen-1}, 
\]
which, by  Theorem \ref{thm:SetLambda}, implies
\[
\frac{\left \|\ba^*(T)\right\|^2_{h^s}}{\left \|\ba^*(0)\right\|^2_{h^s}}\geq 
\frac{S_{\gen-1}}{4S_3} \geq \frac{1}{8}2^{(1-s)(\gen-4)} . 
\]
Thus, taking $\gen$ large enough we obtain growth by a factor of $K/\de$. The time estimates can be easily deduced by \eqref{def:Time:Rescaled}, \eqref{def:LambdaOfN}, \eqref{estimate on fg} and Theorem \ref{thm:ToyModelOrbit}, which concludes the proof of the first statement of Theorem \ref{thm:main}.

For the proof of the second statement of Theorem \ref{thm:main} it is enough to point out that the condition $s<1$ has only been used in imposing that the initial Sobolev norm is small. The estimate for the $\ell^2$ norm can be obtained as explained in Remark \ref{rem:l2}.
%

\end{proof}

\subsection{Proof of Theorem 
\ref{thm:Approximation}}\label{sec:Approx}
To prove Theorem \ref{thm:Approximation}, we define
\[
 \xi=\beta-\beta^\nu(t).
\]
We use the equations in \eqref{eq:betay} to deduce an equation for $\xi$. It 
can be written as 
\begin{equation}\label{eq:ForXi}
 \im\dot \xi=\cZ_0(t)+\cZ_1(t)\xi+\cZ_1'(t)\ol \xi+\cZ_1''(t)\yy+\cZ_2(\xi,\yy,t),
\end{equation}
where 
\begin{equation}\label{def:Zs}
\begin{split}
 \cZ_0(t)=&\partial_{\ol \beta} 
\cJ(0,\theta,\beta^\nu)+\partial_{\ol \beta} \cR(0,\theta,\beta^\nu)\\
 \cZ_1(t)=&\partial_{\beta\ol \beta} \cG_0(\beta^\nu)+\partial_{\beta\ol \beta} 
\cJ(0,\theta,\beta^\nu)\\
\cZ_1'(t)=&\partial_{\ol\beta\ol \beta} \cG_0(\beta^\nu)+\partial_{\ol\beta\ol 
\beta} 
\cJ(0,\theta,\beta^\nu)\\
 \cZ_1''(t)=&\partial_{\yy\ol \beta} \cG_0(\beta^\nu)+\partial_{\yy\ol \beta} 
\cJ(0,\theta,\beta^\nu)\\
 \cZ_2(t)=&\partial_{\ol \beta} 
\cG_0(\beta^\nu+\xi)-\partial_{\ol \beta}\cG_0(\beta^\nu) 
-\partial_{\beta\ol 
\beta} 
\cG_0(\beta^\nu)\xi-\partial_{\ol\beta\ol \beta} 
\cG_0(\beta^\nu)\ol\xi\\
 &\partial_{\ol \beta} 
\cJ(\yy,\theta,\beta^\nu+\xi)-\partial_{\ol \beta}\cJ(0,\theta,\beta^\nu) 
-\partial_{\beta\ol 
\beta} 
\cJ(0,\theta,\beta^\nu)\xi-\partial_{\ol\beta\ol \beta} 
\cJ(0,\theta,\beta^\nu)\ol\xi\\
&-\partial_{\yy\ol \beta} 
\cJ(0,\theta,\beta^\nu)\yy
+\partial_{\ol \beta} \cR(\yy,\theta,\beta^\nu+\xi)-\partial_{\ol \beta} 
\cR(0,\theta,\beta^\nu).
\end{split}
\end{equation}
We analyze the equations for $\xi$ in \eqref{eq:ForXi}
and $\yy$ in \eqref{eq:betay}. 

\begin{lemma}\label{lemma:EqForxi}
Assume that $(\beta^\nu,\yy)$,$(\beta^\nu+\xi,\yy)\in D(r_2)$ (see \eqref{def:domain}) where $r_2$ has been given by Theorem \ref{prop:Birkhoff4}.
Then, the function $\|\xi\|_{\ell^1}$ satisfies
\[
\begin{split}
\frac{d}{dt}\|\xi\|_{\ell^1}\leq & C \nu^{-4}\gen^42^{4\gen}+ 
C \nu^{-3}\gen^32^{3\gen}
\left(f(\gen)^{-\frac{4}{5}} + t f(\gen)^{-2}\right)\\
&+
C\nu^{-2}\gen^22^{2\gen}\|\xi\|_{\ell^1}+C\nu^{-1}\gen 2^\gen 
|\yy|+C\nu^{-1}\gen 2^\gen \|\xi\|_{\ell^1}
^2+C\|\xi\|_{\ell^1}|\yy|+C|\yy|^2
\end{split}
\]
for some constant $C>0$ independent of $\nu$.
\end{lemma}

\begin{proof}
We compute estimates for each term in  \eqref{def:Zs}. For $\cZ_0$, we use the 
fact that the definition of $\cR$ in \eqref{def:HamTruncRotatingSimpl} and Theorem \ref{prop:Birkhoff4} imply $\|\partial_{\ol \beta} 
\cR(0,\theta,\beta^\nu)\|_{\ell^1}\leq \mathcal O(\|\beta^\nu\|_{\ell^1}^4)$.
Thus, it only 
remains to use the results in Theorems \ref{thm:ToyModelOrbit} (using 
\eqref{def:Rescaling}) and Theorem \ref{thm:SetLambda}, to obtain 
\[
\|\partial_{\ol \beta} 
\cR(0,\theta,\beta^\nu)\|_{\ell^1}\leq C\nu^{-4}\gen^4 2^{4\gen}.
\]
To bound $\partial_{\ol \beta} 
\cJ(0,\theta,\beta^\nu)$, the other term in $\cZ_0$, recall that 
$\cJ=\cJ_1+\cJ_2$ (see \eqref{def:HamTruncRotatingSimpl} and \eqref{def:HamToyModelLattice}). Then, we split into two terms $\partial_{\ol 
\beta} \cJ(0,\theta,\beta^\nu)=\partial_{\ol 
\beta} \cJ_1(0,\theta,\beta^\nu)+\partial_{\ol 
\beta} \cJ_2(\theta,\beta^\nu)$ as 
\begin{align}
\partial_{\ol 
\beta} \cJ_1(0,\theta,\beta^\nu)&=\partial_{\ol 
\beta} \left\{ \cG\left(0,\theta, (\beta^\nu_\jj 
e^{\im\Omega_\jj(\lambda, \e)t})_{\jj\in\Z^2_N\setminus\cS_0}\right) 
- \cG\left(0,\theta,\beta^\nu\right) \right\}\notag\\ 
&= \partial_{\ol 
\beta}\left\{ \cQ^\2_{\rm Res} \left(0,\theta, (\beta^\nu_\jj 
e^{\im\Omega_\jj(\lambda, \e)t} )_{\jj\in\Z^2_N\setminus\cS_0}\right) 
- \cQ^\2_{\rm Res}\left(0,\theta,\beta^\nu\right) \right\}\label{j1}\\ 
\partial_{\ol 
\beta} \cJ_2(\theta,\beta^\nu)&=  \partial_{\ol 
\beta} \left\{ \cG\left( 0,\theta, (\beta^\nu_\jj  e^{\im\Omega_\jj(\lambda, 
\e)t})_{\jj\in\Z^2_N\setminus\cS_0}\right)
- \cG_0\left((\beta^\nu_\jj 
e^{\im\Omega_\jj(\lambda, \e)t})_{\jj\in\Z^2_N\setminus\cS_0}\right) 
\right\}\label{j2}
\end{align}
To bound \eqref{j1}, 
recall that $\cQ^\2_{\rm Res}$ defined in \eqref{q2res} is the sum of two 
terms. 
Since $\Pi_{\fR_2}\cQ^\2$ is action preserving, the only terms contributing to  
\eqref{j1} are the ones coming from $\Pi_{\fR_4}\cQ^\2$. Since $\beta^\nu$ is 
supported on $\Lambda$, it follows from \eqref{def:R4} that 
\begin{align}
\partial_{\ol 
\beta} \cJ_1(0,\theta,\beta^\nu) = 
\left(\sum_{\jj_1, \jj_2, \jj_3 \in \Lambda \atop
|\jj_1|^2 - |\jj_2|^2 - |\jj_3|^2 - |\jj|^2 = 0}\left( e^{\im t 
(\Omega_{\jj_1} - \Omega_{\jj_2} + \Omega_{\jj_3} - \Omega_{\jj})} - 1 
\right)\, 
\cJ_{\jj_1 \jj_2 \jj_3 \jj } \, \beta^\nu_{\jj_1} \,  
\overline{\beta^\nu_{\jj_2}} \, \beta^\nu_{\jj_3} \right)_{\jj \in 
\Lambda} . 
\end{align}
In order to bound the oscillating factor, we  use the formula for the 
eigenvalues given in 
Theorem \ref{thm:reducibility4}, to obtain that, for $\jj_1, \jj_2, \jj_3, 
\jj \in \Lambda$, one has 
$$
 \abs{ e^{\im t (\Omega_{\jj_1} - \Omega_{\jj_2} + \Omega_{\jj_3} - 
\Omega_{\jj})} - 1 } \lesssim |t| \abs{\Omega_{\jj_1} - \Omega_{\jj_2} + 
\Omega_{\jj_3} - \Omega_{\jj}}
\lesssim  \frac{|t|}{f(\gen)^2} .  
$$
Hence,
for $t\in [0,T]$, using the estimate for $\mathcal{Q}_\mathrm{Res}^{(2)}$ given by Theorem \ref{prop:Birkhoff4},
\[
\|\partial_{\ol 
\beta} \cJ_1(0,\theta,\beta^\nu)\|_{\ell^1}\leq C t f(\gen)^{-2} 
\|\beta^\nu\|_{\ell^1}^3\leq 
Ct\nu^{-3}\gen^32^{3\gen}f(\gen)^{-2}.
\]
To bound \eqref{j2}, it is enough to use \eqref{def:BoundJ2} and 
\eqref{estimate on fg} to obtain 
\[
\|\partial_{\ol 
\beta} \cJ_2(\theta,\beta^\nu)\|_{\ell^1}\leq 
Cf(\gen)^{-\frac{4}{5}}\|\beta^\nu\|_{\ell^1}^3\leq 
C\nu^{-3}\gen^32^{3\gen}f(\gen)^{-\frac{4}{5}}.
\]
For the linear terms, one can easily see that 
\[
 \left\|\cZ_1(t)\xi\right\|_{\ell^1}\leq 
C\|\beta^\nu\|^2_{\ell^1}\|\xi\|_{\ell^1}\leq 
C\nu^{-2}\gen^22^{2\gen}\|\xi\|_{\ell^1}
\]
and the same for $ \left\|\cZ_1'(t)\ol\xi\right\|_{\ell^1}$. Analogously, 
\[
 \left\|\cZ_1''(t)\yy\right\|_{\ell^1}\leq C\|\beta^\nu\|_{\ell^1}|\yy|\leq 
C\nu^{-1}\gen 2^\gen|\yy|.
\]
Finally, it is enough to use the definition of $\cZ_2$,  the definition of $\cR$ in \eqref{def:HamTruncRotatingSimpl} and Theorem \ref{prop:Birkhoff4}  to show
\[
 \begin{split}
  \|\cZ_2\|\leq& \, 
\|\beta^\nu\|_{\ell^1}
\|\xi\|^2_{\ell^1}+ \|\beta^\nu\|_{\ell^1}^2 |\yy| + \|\xi\|_{\ell^1}|\yy|+\|\beta^\nu\|_{\ell^1}
^3\|\xi\|_{\ell^1}+|\yy|^2\\
\leq &\, C\nu^{-1}\gen 2^\gen |\yy|
\|\xi\|^2_{\ell^1}+ C\nu^{-2}\gen^2 2^{2\gen} 
\|\xi\|_{\ell^1}|\yy|+C\nu^{-3}\gen^3 2^{3\gen}
\|\xi\|_{\ell^1}+|\yy|^2.
 \end{split}
\]
\end{proof}

\begin{lemma}\label{lemma:EqFory}
Assume that $(\beta^\nu,\yy)$, $(\beta^\nu+\xi,\yy)\in D(r_2)$ (see \eqref{def:domain}) where $r_2$ has been given by Theorem \ref{prop:Birkhoff4}. Then, the function $|\yy|$ satisfies
\[
\begin{split}
 \frac{d}{dt}|\yy|\leq&\, C\nu^{-5}\gen^52^{5\gen}
+C\nu^{-3}\gen^32^{3\gen}\|\xi\|_{\ell^1}
^2\\&+C\nu^{-1}\gen2^{\gen}\|\xi\|_{\ell^1}
^3+C\|\xi\|_{\ell^1}|\yy|+C\nu^{-3}\gen^32^{3\gen}|\yy|^2
\end{split}
\]
for some constant $C>0$ independent of $\nu$.
\end{lemma}

\begin{proof}
 Proceeding as for $\dot \xi$, we write the equation for $\dot\yy$ as 
 \begin{equation}\label{eq:ForY}
  \dot \yy= 
\cX_0(t)+\cX_1(t)\xi+\cX_1'(t)\ol\xi+\cX_1''(t)\yy+\cX_2(\xi,\yy,t),
 \end{equation}
with
\[
 \begin{split}
  \cX_0(t)=& -\partial_\theta \cJ(0,\theta,\beta^\nu)-\partial_{\theta} 
\cR(0,\theta,\beta^\nu).
\\
  \cX_1(t)=& \partial_{\beta\theta} 
\cJ(0,\theta,\beta^\nu)
\\
  \cX_1'(t)=&\partial_{\ol\beta\theta} 
\cJ(0,\theta,\beta^\nu)\\
  \cX_1''(t)=&\partial_{\yy\theta} 
\cJ(0,\theta,\beta^\nu)\\
  \cX_2(t)=& -\partial_\theta \cJ(\yy,\theta,\beta^\nu+\xi)+\partial_\theta 
\cJ(0,\theta,\beta^\nu)\\
&  -\partial_{\theta} 
\cR(\yy,\theta,\beta^\nu+\xi)+\partial_{\theta} 
\cR(0,\theta,\beta^\nu).
   \end{split}
\]
We claim that $\cX_1(t)$ and $\cX_1'(t)$ are identically zero. Then, proceeding as in the proof of Lemma \ref{lemma:EqForxi}, one can bound 
each term and complete the proof of Lemma \ref{lemma:EqFory}. 

To explain the absence of linear terms, consider first $\partial_{\beta 
\theta}\cJ(0, \theta, \beta^\nu)$. It contains two types of monomials: those 
coming from  $\fR_2$ (see \eqref{res2}) which however do not depend on 
$\theta$, and those coming from $\fR_4$ (see \eqref{def:R4}). But also these last monomials  do not 
depend on $\theta$ once they are restricted on the set $\Lambda$ (indeed the 
only monomials of $\fR_4$ which are $\theta$ dependent are those  of the third 
line of \eqref{def:R4}, which are supported outside $\Lambda$). 
Therefore $\partial_{\beta \theta}\cJ(0, \theta, \beta^\nu) \equiv 0$ (and so 
$\partial_{\ol\beta \theta}\cJ(0, \theta, \beta^\nu)$ and  $\partial_{\yy 
\theta}\cJ(0, \theta, \beta^\nu)$).

\end{proof}

We define
\begin{equation}\notag
 M=\|\xi\|_{\ell^1}+\nu |\yy|.
\end{equation}
As a conclusion of these two lemmas, we can deduce that 
\[
 \dot M\leq C\left(\nu^{-4}\gen^42^{4\gen}+ 
\nu^{-3}\gen^32^{3\gen}\left(f(\gen)^{-\frac{4}{5}} + t 
f(\gen)^{-2}\right)\right)+C\nu^{-2}\gen^2
2^{2\gen}M+\nu^{-1} \gen 2^\gen M^2 . 
\]
Now we  apply a bootstrap argument. 
Assume that for some $T^*>0$ and $0<t<T^*$ we have
\begin{equation}\notag
 M(t)\leq C\nu^{-1-\sigma/2}. 
\end{equation}
Recall that for $t=0$ we know that it is already satisfied since 
$M(0)\leq\nu^{-1-4\sigma}$. \emph{A posteriori} we will show that the 
time $T$ in \eqref{def:Time:Rescaled}
satisfies $0<T<T^*$ and therefore
the bootstrap assumption holds. Note that, taking $\gen$ large enough (and recalling \eqref{estimate on fg} and \eqref{def:LambdaOfN}), the bootstrap estimate  implies that $(\beta^\nu,\yy)$ and $(\beta^\nu+\xi,\yy)$ satisfy the assumption of Lemmas \ref{lemma:EqForxi} and \ref{lemma:EqFory}.
With the boostrap assumption then, we have 
\[
 \dot M\leq C\left(\nu^{-4}\gen^42^{4\gen}+ 
\nu^{-3}\gen^32^{3\gen}\left(f(\gen)^{-\frac{4}{5}}+tf(\gen)^{-2}
\right)\right)+C\nu^{-2}\gen^2
2^{2\gen}M.
\]
Applying Gronwall inequality,
\[
 M\leq C\left(M(0)+\nu^{-4}\gen^42^{4\gen}t+ 
\nu^{-3}\gen^32^{3\gen}\left(tf(\gen)^{-\frac{4}{5}}+t^2f(\gen)^{-2}
\right)\right)
e^{\nu^{-2}\gen^22^{2\gen} t }
\]
and thus, using \eqref{def:Time:Rescaled} and the estimates for $T_0$ in 
Theorem \ref{thm:ToyModelOrbit},
\[
 M\leq C\left(M(0)+\nu^{-2}\gen^62^{4\gen}+
\nu^{-1}\gen^52^{3\gen}f(\gen)^{-\frac{4}{5}}+
\nu\gen^72^{3\gen}f(\gen)^{-2}\right)e^{C\gen^4
2^{2\gen}}.
\]
%
Since we are assuming \eqref{def:LambdaOfN} and we can take $A$ large enough (see Theorem \ref{thm:SetLambda}),
we obtain that for $t\in [0,T]$, provided $\gen$ is sufficiently large
\[
 M(t)\leq \nu^{-1-\sigma},
\]
which implies that $T\leq T^*$. That is, the bootstrap assumption was 
valid. This completes the proof.


\appendix
\section{Proof of Proposition \ref{hopeful thinking}}
\label{app:mes.m}
We split the proof in several steps. We first perform an algebraic analysis of the nonresonant monomials.
\subsection{Analysis of monomials of the form $e^{\im \theta \cdot \ell} \, 
a_{\jj_1}^{\sigma_1} \, a_{\jj_2}^{\sigma_2} \, a_{\jj_3}^{\sigma_3}\, 
a_{\jj_4}^{\sigma_4}$} 
We analyze the small divisors \eqref{4m}
related to these monomials. Taking 
advantage of the asymptotics of the eigenvalues given in Theorem 
\ref{thm:reducibility4},
we consider a ``good'' first order approximation of the small divisor given by 
\begin{equation}
\label{res.4} 
\omega(\lambda) \cdot \ell + 
\sigma_1 \wt\Omega_{\jj_1}(\lambda, \e) + 
\sigma_2 \wt\Omega_{\jj_2}(\lambda, \e) +
\sigma_3 \wt\Omega_{\jj_3}(\lambda, \e) +
\sigma_4 \wt\Omega_{\jj_4}(\lambda, \e).
\end{equation}
Note that this is an affine function in $\eps$ and 
therefore it can be written 
as 
\[
\eqref{res.4}
\equiv  \tK_{\bj, \ell}^\sigma + \e \tF_{\bj, \ell}^{\sigma}(\lambda).
\]
We say that a monomial is  Birkhoff non-resonant if, for any $\eps>0$, this 
expression is 
not  0 as a function of $\lambda$.

\begin{lemma}
	Assume that the $\tm_k$'s do not solve any of the linear equations defined in 
	\eqref{hyperplane} (this determines $\tL_2$ in the statement of Theorem \ref{hopeful thinking}).
	Consider a monomial of the form  $e^{\im \theta \cdot \ell} \, 
	a_{j_1}^{\sigma_1} \, a_{j_2}^{\sigma_2} \, a_{j_3}^{\sigma_3}\, 
	a_{j_4}^{\sigma_4}$ with $(\bj, \ell, \sigma) \in\fA_4$. If $(\bj, \ell, 
	\sigma) \not\in \fR_4$, then it is Birkhoff non resonant.
\end{lemma}
\begin{proof}
We write explicitly the functions $\tK_{\bj, \ell}^\sigma$ and $\tF_{\bj, 
\ell}^{\sigma}(\lambda)$ as 
	\begin{align}
	\tK_{\bj, \ell}^\sigma & := \omega^\0 \cdot \ell +   
	\s_1\wtOmega_{\jj_1}(\lambda, 0) + \sigma_2 \,  \wtOmega_{\jj_2}(\lambda,0) 
	+\s_3 \wtOmega_{\jj_3}(\lambda,0)+\s_4 \wtOmega_{\jj_3}(\lambda,0)   \label{solapostdoc}\\
	\tF_{\bj, \ell}^{\sigma}(\lambda) &:=  \partial_{\e} \left. \Big(\omega(\lambda) 
	\cdot \ell +  \s_1 \wtOmega_{\jj_1}(\lambda, \e) 
	+\sigma_2\wtOmega_{\jj_2}(\lambda, \e) + \sigma_3\wtOmega_{\jj_3}(\lambda, \e) 
	+\sigma_4\wtOmega_{\jj_4}(\lambda, \e) \Big) \right|_{\e = 0}\notag\\
	& =- \lambda\cdot \ell
	+ \s_1\vartheta_{\jj_1}(\lambda) 
	+ \sigma_2 \vartheta_{\jj_2}(\lambda) 
	+ \sigma_3 \vartheta_{\jj_3}(\lambda) 
	+ \sigma_4 \vartheta_{\jj_4}(\lambda)  \label{patata}
	\end{align}
	As in \cite{Maspero-Procesi},	  $\tK_{\bj, \ell}^\sigma$ is an integer while the functions 
	$\vartheta_{\jj}(\lambda)$ belong to the  finite list of functions
$	\vartheta_{\jj}(\lambda)\in \Big\{0, \{\mu_i(\lambda)\}_{1 \leq i \leq \tk} \Big \}$
	defined in Theorem \ref{thm:reducibility4}.
		Clearly to prove that the resonance \eqref{res.4} not to hold identically, it is enough to ensure that
	\begin{equation}
	\label{K.F.0}
	\tK_{\bj, \ell}^\sigma = 0 \quad \mbox{ and }  \quad \tF_{\bj, 
		\ell}^\sigma(\lambda) \equiv 0 \ 
	\end{equation}
	cannot occur for $(\bj,\ell, \sigma)\in \fA_4\setminus\fR_4$.
	We  study all the possible combinations, each time we assume 
	that \eqref{K.F.0} holds and we deduce a contradiction.
	\begin{enumerate}
		\item $\jj_i \in \fZ$ for any $1 \leq 
		i \leq 4$. 
		In case $\ell \neq 0$, then $\tF_{\bj, \ell}^{\sigma}(\lambda) = 
		-\lambda \cdot \ell$ is not identically $0$. 
		Now take  $\ell = 0$.  By conservation of $\wt\cP_x$, $\wt\cP_y$ 
		we have that $\sum_{i=1}^4 \sigma_i \jj_i=0$ and 
$\tK_{\bj, \ell}^\sigma = 0$ 
		implies  $\sum_{i=1}^4 \sigma_i\abs{\jj_i}^2 = 0$. 
Then, using mass conservation (see Remark \ref{leggi_sel1}), since $\ell=0$, 
one has $\sum_{i=1}^4 \sigma_i=0$ and therefore  the $\jj_i$'s  form a 
		rectangle (and thus $(\bj,0,\sigma)$ belongs to $\fR_4$). 
		
		\item $\jj_1, \jj_2, \jj_3 \in \fZ$, 
		$\jj_4 \in \sS$.
		Then $\tF_{{\bf j}, \ell}^\sigma(\lambda) = -\lambda \cdot \ell 
		+ \sigma_3 \, \mu_{i}(\lambda)$ for some $1 \leq i \leq \tk$. If $\tF_{{\bf j}, 
			\ell}^\sigma(\lambda) \equiv 0$ then $\mu_i(\lambda) = \sigma_3  \lambda \cdot 
		\ell$ is a  root in $\Z[\lambda]$ of the polynomial $P(t, \lambda)$ defined in 
		Theorem \ref{thm:reducibility4}; however $P(t, \lambda)$  is irreducible over $\Q(\lambda)[t]$, thus leading to a contradiction.
		
		%
		\item $\jj_1, \jj_2 \in \fZ$, $\jj_3, 
		\jj_4 \in \sS$. W.l.o.g. let $\jj_3 = (\tm_{i}, n_3)$, $\jj_4 = (\tm_{k}, n_4)$ 
		for some $1 \leq i, k \leq \td$. Then
		$$
		\tF_{{\bf j}, \ell}^\sigma(\lambda) = -\lambda \cdot \ell + 
		\sigma_3\,  \mu_{i}(\lambda) + \sigma_4 \mu_{k}(\lambda) \ . 
		$$
		{\bf Case  $\ell \neq 0$}. Then $\tF_{\bj, \ell}^\bs(\lambda)\equiv 0$ iff 
		$ \mu_i(\lambda) \equiv - \sigma_3 \sigma_4 \mu_k(\lambda)+ \sigma_3\lambda \cdot \ell$.  
		This means that $\mu_k(\lambda)$ is a common root of $P(t,\lambda)$ and  $P(-\sigma_3\sigma_4 t + \sigma_3 \lambda \cdot \ell,\lambda ) $. However this last polynomial is irreducible as well, being the translation of an irreducible polynomial. 
		Hence the two polynomials must be  equal (or opposite). A direct computation shows that this does  not happen (see  Lemma 6.1 of \cite{Maspero-Procesi} for details). \\
{\bf Case  $\ell = 0$}.  Then 		 $\tF_{{\bf j}, \ell}^\sigma(\lambda) \equiv 0$ iff 
		$\mu_i(\lambda) \equiv -\sigma_3\sigma_4 \mu_k(\lambda)$. 
		\begin{itemize}
			\item[-] If $i \neq k$ and $\sigma_3\sigma_4 = -1$, then $P(t, 
			\lambda)$ would have a root with multiplicity 2. But  $P(t, \lambda)$, being an irreducible polynomial, has no multiple roots. 
			\item[-] If $i \neq k$ and $\sigma_3\sigma_4 =  1$, then $P(t, 
			\lambda)$ and $P(-t, \lambda)$ would have $\mu_k(\lambda)$ as a common root. 
			However $P(-t, \lambda)$ is irreducible on $\Z[\lambda]$ as well, and two irreducible polynomials sharing a common root must coincide (up to a sign), i.e.  $P(t, \lambda) \equiv \pm 
			P(-t, \lambda)$. A direct computation using the explicit expression of $P(t, \lambda)$ shows that  this is not true.
			\item[-] If $i=k$ and $\sigma_3\sigma_4 = 1$ then 
			$\mu_i(\lambda) \equiv 0$ would be a root of $P(t, \lambda)$. But $P(t, \lambda)$ is irreducible over $\Z[\lambda]$, does it cannot have $0$ as a root.
			\item[-] If $i=k$ and $\sigma_3\sigma_4 = -1$ (w.l.o.g. assume 
			$\sigma_3=1$, $\sigma_4=-1$),  by mass conservation one has  $\sigma_1 + \sigma_2 = 0 $ and 
			by conservation of $\wt\cP_x$ one has  $\sigma_1 m_1 + \sigma_2 m_2 = 0$, thus 
			$m_1 = m_2$. Then  by conservation of $\wt\cP_y$ we get $n_1 - n_2 + n_3 - n_4= 
			0$, which together with $0 = \tK_{\bj, \ell}^\sigma =n_1^2 - n_2^2 + n_3^2 - 
			n_4^2 $ gives $\{n_1, n_3\}   = \{n_2, n_4\}$. One verifies easily that in such a 
			case the sites $\jj_r$'s form a horizontal rectangle (that could be even 
			degenerate), and therefore they belong to $\fR_4$.
		\end{itemize}

		\item $\jj_1, \jj_2, \jj_3 \in \sS$, $\jj_4 \in \fZ$. W.l.o.g. let  $\jj_1=(\tm_{i_1}, n_1)$, $\jj_2=(\tm_{i_2}, 
		n_2)$, $\jj_3= (\tm_{i_3}, n_3)$ for some $1 \leq i_1,i_2,i_3 \leq \tk$ and  
		$n_1, n_2, n_3 \neq 0$.   Then  
		$$
		\tF_{\bj, \ell}^{\sigma}(\lambda) = -\lambda \cdot \ell + \sigma_1 
		\mu_{i_1}(\lambda)+ \sigma_2 \mu_{i_2}(\lambda) + \sigma_3 \mu_{i_3}(\lambda) \ 
		. 
		$$
		By conservation of mass $\eta(\ell) + \sigma_4 = 0$, hence $\ell \neq 
		0$.  Assume $\tF_{{\bf j}, \ell}^\sigma(\lambda) \equiv 0$. This can only happen for (at most) a unique choice of 
		$\ell^{({\bf i}, \sigma)} \in \Z^\tk$ uniquely, ${\bf i}:=(i_1, i_2, i_3)$.  By 
		conservation of $\wt\cP_x$ we have  $\sum_{k} \tm_k  \ell_k^{({\bf i}, \sigma)} 
		+ \sigma_4 m_4 = 0$. These two conditions fix $m_4 \equiv m_4^{({\bf i}, 
			\sigma)}$ uniquely. In particular if $m_4$ is sufficiently large, we have a 
		contradiction. 
		
		\item $\jj_r \in \sS$, $\forall 1 \leq r \leq 4$. Then
		$$
		\tF_{\bj, \ell}^{\sigma}(\lambda) = -\lambda \cdot \ell + \sigma_1 
		\mu_{i_1}(\lambda)+ \sigma_2 \mu_{i_2}(\lambda) + \sigma_3 \mu_{i_3}(\lambda)+ 
		\sigma_4 \mu_{i_4}(\lambda) \ . 
		$$
		If  $\ell \neq 0$, the condition  $\tF_{{\bf j}, 
			\ell}^\sigma(\lambda) \equiv 0$  fixes $\ell^{({\bf i}, \sigma)} \in \Z^\tk$ 
		uniquely, ${\bf i}:=(i_1, i_2, i_3, i_4)$.  By conservation of $\wt\cP_x$ we 
		have the condition
		\begin{equation}\label{hyperplane}
		\sum_{k} \tm_k  \ell_k^{({\bf i}, \sigma)} = 0
		\end{equation}
		defining a hyperplane, which can be excluded by suitably choosing the tangential 
		sites $\tm_k$ (recall that the functions $\mu_i (\lambda)$ are independent of 
		this choice, see Remark \ref{rmk:mus}). 
		
		If $\ell = 0$, we have $\sum_{r} \sigma_r n_r = \sum_r \sigma_r 
		n_r^2 = 0$. Then $\{n_1, n_3\}   = \{n_2, n_4\}$. One verifies easily that in 
		such case the sites $\jj_r$'s form a horizontal trapezoid (that could be even 
		degenerate).		
	\end{enumerate}
\end{proof}

\subsection{Analysis of monomials of the form $e^{\im \theta \cdot \ell} 
\,\yy^l  a_{j_1}^{\sigma_1} \, a_{j_2}^{\sigma_2} $}

	In this case, since the factor $\yy^l$ does not affect  the Poisson brackets,  admissible monomials (in the sense of Definition \ref{rem:adm3}) are non-resonant provided they do not belong to the set 
	$\fR_2$ introduced in Definition \ref{def:R2}.
\begin{lemma}\label{weakBNR2}
Any monomial of the form  $e^{\im \theta \cdot \ell} \, a_{\jj_1}^{\sigma_1} \, 
a_{\jj_2}^{\sigma_2} \yy_i $ with $(\bj, \ell, \sigma)\notin\fR_2$ admissible in 
the sense of Definition \ref{rem:adm3} is  Birkhoff non-resonant.
\end{lemma}
\begin{proof}
We skip the proof since it is analogous to Lemma 6.1 of \cite{Maspero-Procesi}.
\end{proof}

\subsection{Quantitative measure estimate}

We are now in a position to prove our quantitative non-resonance estimate. 
Recall that, by Theorem \ref{thm:reducibility}, the frequencies $\Omega_\jj(\lambda, \e)$ of  Hamiltonian \eqref{ham.bnf3} have the form  \eqref{as.omega}.
Expanding $\Omega_\jj(\lambda, \e)$ in Taylor series in powers of $\e$ we get that 
\begin{equation}
\label{bb0}
\omega(\lambda) \cdot \ell + \sigma_1  \Omega_{\jj_1}(\lambda, \e) +\sigma_2 \Omega_{\jj_2}(\lambda, \e) + \sigma_3 \Omega_{\jj_3}(\lambda, \e)
+ \sigma_4 \Omega_{\jj_4}(\lambda, \e) = 
\tK_{\bj, \ell}^\bs+ \e \  \tF_{\bj, \ell}^\bs(\lambda) + \e^2 \  \tG_{\bj, \ell}^\bs(\lambda,\e) \ , 
\end{equation}
where $\tK_{\bj, \ell}^\bs $ is defined in \eqref{solapostdoc}
and $\tF_{\bj, \ell}^\bs(\lambda) $ is defined in \eqref{patata}. 
We wish to prove that the set of $\lambda\in \cC^{(2)}_{\e}$ such that 
\begin{equation}
\label{eq:IIIm}
\abs{\omega(\lambda) \cdot \ell + \s_1\Omega_{\jj_1}(\lambda, \e)+ \s_2 \Omega_{\jj_2}(\lambda, \e)+\s_3 \Omega_{\jj_3}(\lambda, \e) + \s_4 \Omega_{\jj_4}(\lambda, \e)} \geq \e \frac{\gamma_2}{\la \ell \ra^{\tau_2}} \ , \qquad \forall \, (\bj,\ell,\bs)\in \fA_4\setminus\fR_4
\end{equation}
has positive measure for $\g_2$ and $\e$ small enough and $\tau_2$ large enough. 
We treat separately the cases $|\ell | \leq 4\tM_0$ and $|\ell| > 4\tM_0$. 
\subsubsection{Case $|\ell| \leq 4 \tM_0$}
We start with the   following lemma.

\begin{lemma}
	\label{lem:cono}
	There exist ${\tt k} \in \N$, such that for any $\gamma_c>0$ sufficiently small, there exists a compact domain   $\cC_{\rm c}\subset\cO_0$, with 
	$|\cO_0 \setminus \cC_{\rm c} | \sim \g_c^{1/{\tt k}}$ and
	\begin{equation}\notag
	\min\left\lbrace\abs{\tF_{\bj, \ell}^\bs(\lambda)} \colon \lambda \in \cC_{\rm c},  \, (\ell,\bj, \bs )\in \fA_4\setminus\fR_4, \, |\ell| \leq 4\tM_0\,,  \, \tK_{\bj, \ell}^\bs=0
	\right\rbrace   \geq \g_c  > 0 \ .
	\end{equation}
\end{lemma}
\begin{proof}See Lemma 6.4 of \cite{Maspero-Procesi}. The estimate on the measure  follows from classical results on sublevels of analytic functions.
\end{proof}
We can now  prove the following result.
\begin{proposition}
\label{prop:cC}
There exits $\e_c>0$ and  a set $\cC_c\subset \cO_0$  such that for any $\e \leq \e_c$, any $\lambda \in \cC_{\rm c}$, one has
 \begin{equation}
 \label{cCc.est}
 \abs{\omega(\lambda) \cdot \ell + \sum_{l = 1}^4 \s_l\Omega_{\jj_l}(\lambda, \e)} \geq \frac{\g_c \e }{2}    \ , \qquad \forall \, (\bj,\ell,\bs)\in \fA_4\setminus\fR_4 \ , \quad |\ell| \leq 4 \tM_0. 
 \end{equation}
 Moreover, one has that $|\cO_0\setminus \cC_c|\leq \alpha \e_c^\kappa$  where $\alpha, \kappa$ do not depend on $\e_c$.
\end{proposition}

\begin{proof}
By the very definition of $\tM_0$ in \eqref{theta.est} and the estimates on the eigenvalues given in Theorem \ref{thm:reducibility4}, one has   $\sup_{\lambda \in \cO_0} |\tF_{\bj, \ell}^\bs(\lambda)| \leq    8\, \tM_0 $ and
$ \sup_{\lambda \in \cO_{0}} |\tG_{\bj, \ell}^\bs(\lambda)| \leq     4\, \tM_0$. 
 Assume first that  $\tK_{\bj, \ell}^\bs \in \Z\setminus\{0\}$, then if $\e_c$ is sufficiently small and for $\e < \e_c$  one has
$$
\abs{ \eqref{bb0} } \geq |\tK_{\bj, \ell}^\bs| - \e 8 \tM_0 -  \e^2 4 \tM_0 \geq \frac{1}{2}.
$$
Hence, for such $\ell$'s,  \eqref{cCc.est} is trivially fulfilled $\forall \lambda \in \cO_0$. If instead $\tK_{\bj, \ell}^\bs = 0$, 
we use  Lemma \ref{lem:cono} with $\gamma_c=10\tM_0 \e_c$ to obtain a set $\cC_c\subset \cO_0$, such that for any $\lambda \in \cC_{c}$ and any $(\bj,\ell,\bs)\in \fA_4\setminus\fR_4$ with $ |\ell| \leq 4 \tM_0$
$$
\abs{\eqref{bb0}} \geq \e \g_c -   \e^2 4\tM_0  \geq \frac{\e\g_c}{2}=:C\e \  . 
$$
\end{proof}

\subsubsection{ Case $|\ell| > 4\tM_0$}
In this case we prove the following result.
\begin{proposition}
\label{prop:Cs}
Fix $\e_\star >0$ sufficiently small and $\tau_\star >0$ sufficiently large. For any $\e < \e_\star$, there exists a set $\cC_\star \subset \cO_0$ such that 
$\abs{\cO_0 \setminus \cC_\star} \lesssim \e_\star ^\kappa
$
(with $\alpha, \kappa$ independent of $\e_\star$), and for any $\lambda \in \cC_\star$ and $ |\ell| > 4 \tM_0$ one has 
\begin{equation}
 \label{cCs.est}
 \abs{\omega(\lambda) \cdot \ell + \sum_{l = 1}^4 \s_l\Omega_{\jj_l}(\lambda, \e)} \geq \gamma_\star \frac{\e}{\la \ell \ra^{\tau_\star}}  \ . 
 \end{equation}
 for some constant $\gamma_\star$ depending on $\e_\star$.
\end{proposition} 

To prove the proposition, first define,  
for $1 \leq i \leq \tk$ and $0 \leq k \leq \tk$,  the functions
\begin{equation}\notag
\widehat\tF_{i,k}(\lambda)  = 
\begin{cases}
\e \mu_{i}(\lambda) & \mbox{ if }  k = 0 \\
\e \mu_{i,k}^+ (\lambda) & \mbox{ if }  1 \leq i < k \leq \tk \\
\e \mu_{i,k}^- (\lambda) & \mbox{ if }  1 \leq k < i \leq \tk \\
0 & \mbox{ if}\; 1 \leq  i=k  \leq \tk
\end{cases}
\end{equation}
The right hand side of \eqref{bb0} is always of the form 
\begin{equation}
\label{cc}
\begin{aligned}
\omega  (\lambda) \cdot \ell + K  
&+\eta_1 \widehat\tF_{i_1,k_1}(\lambda) + \eta_2 \widehat\tF_{i_2,k_2}(\lambda) +  \eta_3 \widehat\tF_{i_3,k_3}(\lambda) +
\eta_4 \widehat\tF_{i_4,k_4}(\lambda)   \\
&+ \eta_{11} \frac{\Theta_{m_1}(\lambda, \e)}{\la m_1 \ra^2} + \eta_{12} \frac{\Theta_{m_2}(\lambda, \e)}{\la m_2 \ra^2} +  \eta_{13} \frac{\Theta_{m_3}(\lambda, \e)}{\la m_3 \ra^2} 
+\eta_{14} \frac{\Theta_{m_4}(\lambda, \e)}{\la m_4 \ra^2} \\
&+ \eta_{21} \frac{\Theta_{m_1,n_1}(\lambda, \e)}{\la m_1\ra^2 +  \la n_1 \ra^2} +   \eta_{22} \frac{\Theta_{m_2,n_2}(\lambda, \e)}{\la m_2 \ra^2 +  \la n_2 \ra^2} +   \eta_{23} \frac{\Theta_{m_3,n_3}(\lambda, \e)}{\la m_3 \ra^2+ \la n_3 \ra^2 } 
+   \eta_{24} \frac{\Theta_{m_4,n_4}(\lambda, \e)}{\la m_4 \ra^2+ \la n_4 \ra^2 } \\
& {+  \eta_{31} \frac{\varpi_{m_1}(\lambda, \e)}{\la m_1 \ra} + \eta_{32} \frac{\varpi_{m_2}(\lambda, \e)}{\la m_2 \ra} +  \eta_{33} \frac{\varpi_{m_3}(\lambda, \e)}{\la m_3 \ra}}
+  \eta_{34} \frac{\varpi_{m_4}(\lambda, \e)}{\la m_4 \ra}
\end{aligned}
\end{equation}
for a particular choice of  $K \in \Z$, $m_i\in \Z,n_i\in N\Z\setminus \{0\}$ and  $\eta_r, \eta_{j j'} \in \{ -1, 0, 1 \}$. 
Therefore it is enough to show \eqref{cCs.est} where the left hand side is replaced by \eqref{cc}.  
\begin{proof}[Proof of Proposition \ref{prop:Cs}]
If the integer $K$ is sufficiently large, namely 
$|K| \geq 4 \,   |\ell| \displaystyle{\max_{1 \leq i \leq \tk} }
	(\tm_i^2) $, then 
the quantity in the left hand side of \eqref{cCs.est} is far from zero. More precisely one has  
	\begin{align*}
	|\eqref{cc} | & \geq |K| - |\omega(\lambda)| \, |\ell| - \sum_{r=1}^4 \abs{\widehat\tF_{i_r,k_r}}^{\cO_1} - \sum_{r=1}^4 \frac{\abs{\Theta_{m_r}(\cdot, \e )}^{\cO_1}}{\la m_r \ra^2}  - \sum_{r=1}^4 \frac{\abs{\Theta_{m_r, n_r}(\cdot, \e )}^{\cO_1}}{\la m_r \ra^2 + \la n_r \ra^2 }  {- \sum_{r=1}^4 \frac{\abs{\varpi_{m_r}(\cdot, \e )}^{\cO_1}}{\la m_r \ra} }\\
	\notag
	& \geq 4 \, \max_{1 \leq i \leq \tk} (\tm_i^2) \,  |\ell| - \max_{1 \leq i \leq \tk} (\tm_i^2) \,  |\ell| - \e |\ell| - 4\e \tM_0-4\e^2 \tM_0 \geq  \tM_0 \ .
	\end{align*}

	
So from now on we restrict ourselves to the   case $|K| \leq 4 \,   |\ell| \displaystyle{\max_{1 \leq i \leq \tk} }
	(\tm_i^2) $. We will repeatedly use the following  result, which is an easy variant of Lemma 5 of  \cite{Poschel96a}.
\begin{lemma}
	\label{c.2}
	Fix arbitrary $K \in \Z$,  $m_i\in \Z,n_i\in \Z\setminus \{0\}$, $\eta_j, \eta_{j j'} \in \{ -1, 0, 1 \}$. 
	For any $\al >0$ one has
	\begin{equation}\notag 
	{\rm meas}(\{ \lambda \in \cO_{0}:  |\eqref{cc}| < \e\alpha \}) < 16 \alpha |\ell|^{-1} \ .
	\end{equation}
\end{lemma}
The proof relies on the fact that all the functions appearing in \eqref{cc} are Lipschitz in $\lambda$, for full details see e.g. Lemma C.2 of \cite{Maspero-Procesi}. 

Now, let us fix 
\begin{equation}
\label{cond12}
\g_\star =\frac{\e_\star \tM_0}{100}.
\end{equation}
We construct the set $\cC_\star$ by induction on the number $n$ defined by 
	$$n:=|\eta_{1,1}|+ \cdots + |\eta_{3,4}| \leq 12 $$
	which is nothing but the number of nonzero coefficients in 
	\eqref{cc} . For every $0 \leq n \leq 12$ we construct $(i)$ a positive increasing sequence $\tau_n$ and $(ii)$ a sequence of nested sets $\cC^n = \cC^n(\g_\star, \tau_n)$   such that 
	\begin{enumerate}
	\item There exists $C >0$, independent of $\e$ and $\g_\star$, s.t. 
	\begin{equation}
	\label{meas.cn}
	\meas(\cO_0 \setminus \cC^{0}) \leq C \g_\star \ , \quad 
	\meas(\cC^n \setminus \cC^{n+1}) \leq C \g_\star 
	\end{equation} 
	\item For $\lambda \in \cC^n$ and  $|\ell| \geq 4 \tM_0$  one has 
	\begin{align}
	\label{eta.ind}
	\Big| \eqref{cc}  \Big| \geq \frac{\e \ \g_\star}{\la \ell\ra ^{\tau_n}} \ .
	\end{align}
	\end{enumerate}

	Then the proposition follows by taking $\cC_\star := \cC^{12}$,  $\tau_\star = \tau_{12}$, so that one has $\abs{\cO_0 \setminus \cC_\star} \leq 13 C \g_\star \sim \g_\star$, provided $\g_\star$ is small enough. \\
	
\noindent\textbf{Case $n=0$:}
	%
Define the set
	$$
	G_{K, \bi, \bk, \eta,\ell}^0(\g_\star, \tau_0) := 
	\left\{\lambda  \in \cO_0 \ : \ |\eqref{cc}| \leq \frac{\e \ \g_\star}{\la \ell\ra ^{\tau_0}} \  \mbox{ and } \ \eta_{j j'} = 0 \  \ \  \forall j, j'  \right\} \ ,
	$$
	where  $K \in \Z$ with $|K|\leq  4 \, \displaystyle{\max_{1 \leq i \leq \tk} (\tm_i^2)} \,  |\ell|$, $\bi = (i_1, i_2, i_3, i_4) \in \{1, \ldots \tk\}^4$, $\bk= (k_1, k_2, k_3, k_4) \in \{0,\ldots ,\tk\}^4$, $\ell \in \Z^\tk$ with $|\ell| \geq 4 \tM_0$, $\eta=(\eta_1,\eta_2,\eta_3, \eta_4) \in \{-1, 0, 1\}^4$.
	 By Lemma \ref{c.2} with $\alpha = \g_\star \la \ell \ra^{-\tau_0}$
	we have
	$$
	\meas \left( G_{K, \bi, \bk, \eta,\ell}^0(\g_\star, \tau_0)\right) \leq \frac{  16 \g_\star}{ \, \la \ell\ra ^{\tau_0+1}} \ .
	$$
	Taking the union over all the possible values of $K, \bi, \bk,\eta, \ell$ one gets that 
	$$
	\meas \left ( \bigcup_{|\ell| \geq 4 \tM_0, \ \bi, \bk,\eta \atop
		|K| \leq 4 \, \max_i (\tm_i^2) \,  |\ell| }  G_{K, \bi, \bk, \eta, \ell}^0(\g_\star, \tau_0) \right) \leq  C(\tk) \, \g_\star \,  \sum_{|\ell| \geq 4 \tM_0} \frac{1}{   \, \la \ell\ra ^{\tau_0}} \leq C \g_\star  \ ,
	$$
	which is finite provided $\tau_0 \geq \tk+1$.
	Letting 
	$$\cC^0 := \cO_0 \setminus \bigcup_{|\ell| \geq 4 \tM_0, \ \bi, \bk,\eta \atop
		|K| \leq 4 \, \max_i (\tm_i^2) \,  |\ell| }   G_{K, \bi, \bk,\eta, \ell}^0(\g_\star, \tau_0)  $$
	one has clearly that 
	$\meas (\cO_0 \setminus \cC^0) \leq C\g_\star$ and   for $\lambda \in \cC^0 $ we have
	\begin{equation}\notag
	\abs{ \omega(\lambda) \cdot \ell + K  
		+ \sum_{r=1}^4 \eta_j \widehat\tF_{i_r,k_r}(\lambda)   } \geq \frac{\e \ \g_\star}{ \la \ell\ra ^{\tau_0}}
	\end{equation}
	for any admissible choice of $\ell,K,\bi,\bk,\eta$. This proves the inductive step for $n=0$.\\
	
\noindent\textbf{Case $n \leadsto n+1$:}  Assume that \eqref{eta.ind} 
holds for any possible choice of 
	$\eta_{11}, \ldots, \eta_{34}$ s.t. $|\eta_{11}|+ \cdots + |\eta_{34}| \leq  n \leq 11 $ for some $(\tau_j)_{j = 1}^n$. 
	We prove now the step $n+1$.
Let us fix $\tau_{n+1} \geq \td + 1+ 6 \tau_n$.	We shall show that for each $| \ell | \geq 4 \tM_0$, the set
\begin{equation}
\label{setC}
G_\ell^{n+1}:= \left\lbrace \lambda \in \cC^n \colon  \abs{\eqref{cc}} \leq 
\frac{\e \g_\star}{\la \ell \ra^{\tau_{n+1}}} \ ,  \ \ \  |\eta_{11}| +\ldots+ |\eta_{34}|  = n + 1 \right\rbrace
\end{equation}
	has measure $\displaystyle{\leq \frac{C(\tk) \g_\star}{\la \ell \ra^{\tk +1}}}$.
	Thus  defining
	$$
	\cC^{n+1} := \cC^n\setminus \bigcup_{|\ell| \geq 4 \tM_0 \atop  }  G^{n+1}_{ \ell}(\g_\star, \tau_{n+1}) .
	$$
we obtain the claimed estimates  \eqref{meas.cn} and \eqref{eta.ind}.  
To estimate the measure of \eqref{setC} we split  in three cases.\\
	
		\medskip

\noindent {\em Case 1}: Assume that
	\begin{equation}\notag
	\exists \ m_i \mbox{ s.t. } |m_i| \geq  \la \ell \ra^{\tau_n}
	\end{equation}
(of course we also assume that  one of the coefficients $\eta_{1i}, \eta_{2i}, \eta_{3i}$ is not null). 	W.l.o.g. assume it is $m_4$. Then we treat all the terms in \eqref{cc} which contain $m_4$ as perturbations, and we estimate all the other terms using the inductive assumption. 
	Here the details: first  we have 
	$$
	\abs{\frac{\Theta_{m_4}(\lambda, \e)}{\la m_4\ra^2 }} +  \abs{\frac{\Theta_{m_4,n_4}(\lambda, \e)}{\la m_4 \ra^2 +   \la n_4 \ra^2 }}+
	\abs{\frac{\varpi_{m_4}(\lambda, \e)}{\la m_4\ra }} 
	\leq \frac{\tM_0 \, \e^2}{\la \ell \ra^{\tau_n}} .
	$$
	By the inductive assumption \eqref{eta.ind} and   \eqref{cond12}, for 
any $\lambda \in \cC^n$ one has
	\begin{align*}
	\Big| \eqref{cc}  \Big|   \geq& 
	\left| \omega(\lambda) \cdot \ell + K  
	+\sum_{j=1}^4 \eta_i \widehat\tF_{i_j,k_j}(\lambda)   + \sum_{j=1}^3 \eta_{1j} \frac{\Theta_{m_r}(\lambda, \e)}{\la m_j \ra^2} 
		+ \sum_{j=1}^3\eta_{2j} \frac{\Theta_{m_j,n_j}(\lambda, \e)}{\la m_j\ra^2 + \la n_j \ra^2}  
	+ \sum_{j=1}^3 \eta_{3j} \frac{\varpi_{m_j}(\lambda, \e)}{\la m_j \ra} 
	\right| \\
	&-  \frac{\tM_0 \, \e^2}{\la \ell \ra^{\tau_n}} \\
	\geq&  \frac{\e \ \g_\star}{ \la \ell\ra ^{\tau_n}} -  \frac{\tM_0 \, \e^2}{\la \ell \ra^{\tau_n}} 
	\geq \frac{\e \ \g_\star}{2 \la \ell\ra ^{\tau_{n}}} \geq \frac{\e \ \g_\star}{ \la \ell\ra ^{\tau_{n+1}}}
	\end{align*}
	provided $\tau_{n+1} \geq \tau_n +1$. Therefore, in this case, there are no $\lambda$'s contributing to the set \eqref{setC}.\\
	
\noindent{\em Case 2:} Assume that 
	\begin{equation}\notag
	\exists \ n_i \mbox{ s.t. } |n_i|^2 \geq  \la \ell \ra^{\tau_n}  \  
	\end{equation}
	(and again  we also assume that one of  the coefficients $ \eta_{2i}$ is not null). 
	W.l.o.g. assume it is $n_4$. 
Similarly to the previous case,  we treat the term in \eqref{cc} which contains $n_4$ as a perturbation, and we estimate all the other terms using the inductive assumption. 	
	More precisely we have 
	$$
	\abs{\frac{\Theta_{m_4,n_4}(\lambda, \e)}{\la m_4 \ra^2 +  \la n_4 \ra^2}} \leq \frac{\tM_0 \, \e^2}{\la \ell \ra^{\tau_n}}  , 
	$$
	so  by the inductive assumption \eqref{eta.ind} and   \eqref{cond12}
	\begin{align*}
	\Big| \eqref{cc}  \Big| \geq & 
	\Bigg| \omega(\lambda) \cdot \ell + K  +\sum_{j=1}^4 \eta_i \widehat\tF_{i_j,k_j}(\lambda)   +  \sum_{j=1}^4 \eta_{1j} \frac{\Theta_{m_j}(\lambda, \e)}{\la m_j \ra^2}  
	+ \sum_{j=1}^3\eta_{2j} \frac{\Theta_{m_j,n_j}(\lambda, \e)}{\la m_j\ra^2+  \la n_jj \ra^2} 
	+ \sum_{j=1}^4 \eta_{3j} \frac{\varpi_{m_j}(\lambda, \e)}{\la m_j \ra}\Bigg| \\ 
	&
	-  \frac{\tM_0 \, \e^2}{\la \ell \ra^{\tau_n}} \\
	\geq & \frac{\e \ \g_\star}{2 \la \ell\ra ^{\tau_{n}}} \geq \frac{\e \ \g_\star}{\la \ell\ra ^{\tau_{n+1}}}
	\end{align*}
	provided $\tau_{n+1} \geq \tau_n +1$. 
	Also in this case, there are no $\lambda$'s contributing to the set \eqref{setC}.\\
	
\noindent{\em Case 3:} We have  
	 $$|m_i|\, , |n_i|^2 \leq \la \ell \ra^{\tau_n}$$
	  for all the $m_i, n_i$ that appear  in \eqref{cc} with nonzero coefficients. Furthermore, recall that we are considering just the case 
	  $|K| \leq 4 \, \max_i (\tm_i^2) \, | \ell | . $
	  Thus we are left with a finite number of cases and we can impose a finite number of Melnikov conditions.
	So  define the sets
	$$
	G^{n+1}_{K,  \bi, \bk,  \eta, \ell, \bm, \bn}(\g_\star, \tau_{n+1}) := \left\{\lambda  \in \cC^n \ : \ |\eqref{cc}| \leq \frac{\e \ \g_\star}{ \la \ell\ra ^{\tau_{n+1}}} \ , \qquad  |\eta_{11}|+ \cdots + |\eta_{34}| =   n+1 \right\} .
	$$
	By Lemma \ref{c.2} with $\alpha = \g / \la \ell \ra^{\tau_{n+1}}$ we have
	\begin{equation}
	\label{atl}
	\meas \left( G^{n+1}_{K,  \bi, \bk,  \eta, \ell, \bm, \bn}(\g_\star, \tau_{n+1}) \right) \leq \frac{ 16\g_\star}{ \, \la \ell\ra ^{\tau_{n+1}+1}} \ ,
	\end{equation}
	and taking the union over the possible values of $K,  \bi, \bk,  \eta,  \bm, \bn$ one gets that
	$$
	G_\ell^{n+1} \equiv \bigcup_{ \bi, \bk, \eta} \ \ \ \bigcup_{|m_i| \, ,\  |n_i|^2 \leq \la \ell \ra^{\tau_n}  \atop
		|K| \leq 4 \, \max_i (\tm_i^2) \,  |\ell| } G_{ K,  \bi, \bk,  \eta, \ell, \bm, \bn}(\g_\star, \tau_{n+1}) .
	$$
	Estimate \eqref{atl} gives immediately
	$$
	\meas \Big( G_\ell^{n+1} \Big) \leq C(\tk) \, \g_\star \frac{ \la \ell \ra^{1+ 6 \tau_n}}{ \, \la \ell\ra ^{\tau_{n+1}+1}} \leq   \frac{ C(\tk) \, \g_\star }{  \la \ell\ra ^{\tk +1 }} 
	$$
	which is what we claimed. 
	
\end{proof}

We can finally prove Proposition \ref{hopeful thinking}.
\begin{proof}[Proof of Proposition \ref{hopeful thinking}]
Fix $ \g_c =  \g_\star =: \g_2$ sufficiently small, and put $\e_2 := \min(\e_c, \e_\star)$, $\tau_2 := \tau_\star$ and  $\cC^{(2)} := \cC_{\rm c} \cap \cC_\star$. Propositions \ref{prop:cC} and \ref{prop:Cs} guarantee that for any $\lambda \in \cC^{(2)}$, estimate \eqref{eq:IIIm} is fulfilled. Finally one has
$\abs{\cC^{(1)} \setminus \cC^{(2)}} \lesssim \g_2^{1/{\tt k}} + \g_2 \sim  \g_2^{1/{\tt k}}$.
\end{proof}

\bibliography{references}
\bibliographystyle{alpha}
\end{document}